\pdfoutput=1

\documentclass{amsart}

\usepackage[margin=1.2 in]{geometry}

\usepackage{verbatim}
\usepackage{xcolor}
\usepackage{tikz}
\usepackage{tikz-cd}
\usepackage{arydshln}
\usepackage{caption}
\usepackage{subcaption}
\usepackage{graphicx}
\usepackage[hyphens]{url}
\usepackage{amsmath}
\usepackage{bbm}
\usepackage{hyperref}
\usepackage{mathrsfs}

\newtheorem{theorem}{Theorem}[section]
\newtheorem{proposition}[theorem]{Proposition}
\newtheorem{lemma}[theorem]{Lemma}

\theoremstyle{definition}

\theoremstyle{remark}
\newtheorem{remark}[theorem]{Remark}

\numberwithin{equation}{section}

\newcommand{\mcg}{\mathrm{Mod}_g}
\newcommand{\mc}{\mathbf{g}}
\renewcommand{\tt}{\mathcal{T}_g}

\newcommand{\qt}{\mathcal{Q}\mathcal{T}_g}

\newcommand{\qut}{\mathcal{Q}^1\mathcal{T}_g}

\newcommand{\qum}{\mathcal{Q}^1\mathcal{M}_g}

\newcommand{\mf}{\mathcal{MF}_g}
\newcommand{\pmf}{\mathcal{PMF}_g}

\newcommand{\IT}{\Sigma}

\newcommand{\supp}{\mathrm{supp}}

\newcommand{\M}{\mathrm{M}}
\newcommand{\sh}{\smash{\widehat{s}}}

\newcounter{count}

\newcounter{counterk} 
 
\newcounter{counterl} 
 
\newcounter{counterc} 
 
\newcounter{countercc} 
 
\newcounter{countere} 
 
\newcounter{counterd} 
 
\newcounter{countern} 
 
\newcounter{counterr}
 
\begin{document}

\title[Effective mapping class group dynamics II]{Effective mapping class group dynamics II:\\ Geometric intersection numbers}

\author{Francisco Arana--Herrera}
\address{Department of Mathematics, Stanford University, 450 Jane Stanford Way, Stanford, CA 94305, USA.}
\email{farana@stanford.edu}

\begin{abstract}
	We show that the action of the mapping class group on the space of closed curves of a closed surface effectively tracks the corresponding action on Teichmüller space in the following sense: for all but quantitatively few mapping classes, the information of how a mapping class moves a given point of Teichmüller space determines, up to a power saving error term, how it changes the geometric intersection numbers of a given closed curve with respect to arbitrary geodesic currents. Applications include an effective estimate describing the speed of convergence of Teichmüller geodesic rays to the boundary at infinity of Teichmüller space, an effective estimate comparing the Teichmüller and Thurston metrics along mapping class group orbits of Teichmüller space, and, in the sequel, effective estimates for countings of filling closed geodesics on closed, negatively curved surfaces.
\end{abstract}

\maketitle


\thispagestyle{empty}

\tableofcontents

\section{Introduction}

Much progress has been made in the last 40 years towards understanding the dynamics of the mapping class group on different spaces: Teichmüller space \cite{ABEM12}, the space of singular measured foliations \cite{Mas85,Mir08b}, the space of geodesic currents \cite{Mir16,ES16,RS19}. Nevertheless, effective estimates describing these dynamics have remained quite elusive, at least until very recently. The first quantitative estimates with power saving error terms for countings of mapping class group orbits of simple closed curves were proved by Eskin, Mirzakhani, and Mohammadi in \cite{EMM19}. The first quantitative estimates with power saving error terms for countings of mapping class group orbits of Teichmüller space were proved by the author in the prequel \cite{Ara20b}. The main goal of this paper is to further expand our understanding of the effective dynamics of the mapping class group by studying how the corresponding action on the space of closed curves changes their geometric intersection numbers.

More concretely, we show that the action of the mapping class group on the space of closed curves of a closed surface effectively tracks the corresponding action on Teichmüller space in the following sense: for all but quantitatively few mapping classes, the information of how a mapping class moves a given point of Teichmüller space determines, up to a power saving error term, how it changes the geometric intersection numbers of a given closed curve with respect to arbitrary geodesic currents. This description provides a new perspective for studying the action of the mapping class group on the space of closed curves of a closed surface in terms of the corresponding action on Teichmüller space.

This perspective is very useful for tackling a wide variety of related effective counting problems. In the sequel \cite{Ara21a}, we combine the main result of this paper with theorems in the prequel \cite{Ara20b} to prove quantitative estimates with power saving error terms for countings of filling closed geodesics on closed, negatively curved surfaces. These estimates complement the effective estimates of Eskin, Mirzakhani, and Mohammadi for countings of simple closed curves \cite{EMM19}, effectivize asymptotic counting results of Mirzakhani, Erlandsson, and Souto \cite{Mir16,ES16}, and solve an open problem advertised by Wright \cite[Problem 18.2]{Wri19} for a generic class of closed curves. 

The perspective introduced in this paper can also be used to shed new light on the geometry and dynamics of Teichmüller space. Using the main theorem of this paper we show that all but quantitatively few Teichmüller geodesic segments joining a point in Teichmüller space to points in a mapping class group orbit of Teichmüller space converge at an effective rate to the boundary at infinity of Teichmüller space. This result corresponds to an effective version of a theorem of Masur \cite{Ma82b}. We also prove a completely new effective estimate comparing the Teichmüller and Thurston metrics along mapping class group orbits of Teichmüller space. In the sequel \cite{Ara21a}, we combine this result with theorems in the prequel \cite{Ara20b}  to prove quantitatives estimates with power saving error terms for countings of mapping class group orbits of Teichmüller space with respect to the Thurston metric, effectivizing asymptotic counting results of Rafi and Souto \cite{RS19}.

This paper addresses the strictly topological question of estimating the geometric intersection numbers of closed curves on surfaces using geometry and dynamics. More concretely, the proof of the main theorem of this paper is based on a novel combination of ideas from three different sources: flat geometry of quadratic differentials, dynamics of straight line flows on surfaces, and Teichmüller dynamics. Using the flat geometry of quadratic differentials we reduce the problem of estimating the geometric intersection numbers of closed curves to a question about the equidistribution rate of long horizontal segments of quadratic differentials. We study this question using the renormalization dynamics of the Teichmüller geodesic flow.

\subsection*{Statement of the main theorem.} For the rest of this section we fix an integer $g \geq 2$ and a connected, oriented, closed surface $S_g$ of genus $g$. Denote by $\mcg$ the mapping class group of $S_g$. Denote by $\tt$ the Teichmüller space of marked complex structures on $S_g$. Denote by $d_\mathcal{T}$ the Teichmüller metric on $\mathcal{T}_g$. Consider the marking changing action of $\mcg$ on $\tt$. This action is properly discontinuous and preserves the Teichmüller metric.

Denote by $\qut$ the Teichmüller space of marked, unit area holomorphic quadratic differentials on $S_g$ and by $\pi \colon \qut \to \tt$ the natural projection to $\mathcal{T}_g$. Denote by $\mathcal{MF}_g$ the space of singular measured foliations on $S_g$ and by $\Re(q),\Im(q) \in \mf$ the vertical an horizontal foliations of $q \in \qut$. 

Denote by $\mathcal{C}_g$ the space of geodesic currents on $S_g$ Denote by $i(\cdot,\cdot)$ the geometric intersection number pairing on $\mathcal{C}_g$. Following Thurston \cite{T80} and Bonahon \cite{Bon88}, we interpret closed curves and singular measured foliations on $S_g$ as elements of $\mathcal{C}_g$.

Denote by $\Delta \subseteq \mathcal{T}_g \times \mathcal{T}_g$ the diagonal of $\mathcal{T}_g \times \mathcal{T}_g$. Let $S(X) := \pi^{-1}(X)$ for every $X \in \mathcal{T}_g$. Consider the maps $q_s, q_e \colon \mathcal{T}_g \times \mathcal{T}_g - \Delta \to \mathcal{Q}^1\mathcal{T}_g$  which to every pair  of distinct points $X,Y \in \mathcal{T}_g$ assign the quadratic differentials $q_s(X,Y) \in S(X)$ and $q_e(X,Y) \in S(Y)$ corresponding to the tangent directions at $X$ and $Y$ of the unique Teichmüller geodesic segment from $X$ to $Y$. 

In the prequel \cite{Ara20b} we showed there exist constants $A = A(g) > 0$ and $\kappa = \kappa(g) > 0$ depending only on $g$ such that for every $X,Y \in \mathcal{T}_g$,
\begin{equation}
\label{eq:max_growth}
|\{\mc \in  \mcg \colon d_\mathcal{T}(X,\mc.Y) \leq R \}| = A \cdot e^{(6g-6)R} + O_{X,Y}\big(e^{(6g-6-\kappa)}\big).
\end{equation}
Let $X,Y \in \mathcal{T}_g$, $C > 0$, and $\kappa > 0$. Motivated by (\ref{eq:max_growth}), we say that a subset of mapping classes $\M \subseteq \mcg$ is $(X,Y,C,\kappa)$-sparse if the following bound holds for every $R > 0$,
\[
|\{\mc \in  \M \colon d_\mathcal{T}(X,\mc.Y) \leq R \}| \leq C \cdot e^{(6g-6-\kappa)R}.
\] 

The following effective estimate for the geometric intersection numbers of closed curves in a given mapping class group orbit with respect to arbitrary geodesic currents is the main result of this paper. A stronger version of this result will be introduced in \S 9 as Theorem \ref{theo:main_strong}. A version of this result yielding stronger conclusions for simple closed curves will be introduced in \S 9 as Theorem \ref{theo:main_unif}.

\begin{theorem}
	\label{theo:main} 
	There exists a constant $\kappa = \kappa(g) > 0$ such that the following holds. For every $X,Y \in \mathcal{T}_g$ and every closed curve $\beta$ on $S_g$, there exists a constant $C = C(X,Y,\beta) > 0$ and an $(X,Y,C,\kappa)$-sparse subset $M = M(X,Y,\beta) \subseteq \mcg$ such that for every geodesic current $\alpha$ on $S_g$ and every $\mc \in \mcg \setminus M$, if $r := d_\mathcal{T}(X,\mc.Y)$, $q_s := q_s(X,\mc.Y)$, and $q_e := q_e(X,\mc.Y)$, then
	\[
	i(\alpha,\mc.\beta) = i(\alpha,\Re(q_s)) \cdot i(\mc.\beta,\Im(q_e)) \cdot e^r + O_{X,Y,\alpha,\beta}\left(e^{(1-\kappa)r}\right).
	\]
\end{theorem}

Theorem \ref{theo:main} can be interpreted as follows: for all but quantitatively few mapping classes $\mc \in \mcg$, to estimate the geometric intersection number $i(\alpha,\mc.\beta)$, it is not necessary to know how $\alpha$ and $\mc.\beta$ interact between themselves, but rather, it is enough to know how they independently interact with objects determined by the action of $\mc$ on $\mathcal{T}_g$.

Every negatively curved metric on $S_g$ induces a geodesic current whose geometric intersection number with any closed curve is equal to the length of the unique geodesic representative of the curve with respect to the negatively curved metric. Thus, Theorem \ref{theo:main} provides effective estimates for the lengths of closed geodesics of a given topological type on any closed, negatively curved surface.

\begin{remark}
In this paper, whenever we write expressions of the form $q_s(X,\mc.Y)$ or $q_e(X,\mc.Y)$ for $X,Y \in \mathcal{T}_g$ and $\mc \in \mcg$ as in the statement of Theorem \ref{theo:main}, we will always assume $\mc.Y \neq X$. For any pair of points $X,Y \in \mathcal{T}_g$, the number of mapping classes $\mc \in \mcg$ such that $\mc.Y = X$ can be bounded uniformly in terms of $g$. Thus, making this assumption will never affect our results as any such mapping classes can always be included in the sparse subsets of $\mcg$ to be discarded.
\end{remark}

\subsection*{Effective convergence to the boundary at infinity of Teichmüller space.} By uniformization, $\mathcal{T}_g$ can be canonically identified with the Teichmüller space of marked hyperbolic structures on $S_g$. Given a closed curve $\beta$ on $S_g$ and a marked hyperbolic structure $X \in \mathcal{T}_g$, denote by $\ell_{\beta}(X)$ the length of the unique geodesic representative of $\beta$ with respect to $X$. Following Thurston \cite{T80}, the length $\ell_{\lambda}(X)$ of a singular measured foliation $\lambda \in \mathcal{MF}_g$ with respect to a marked hyperbolic structure $X \in \mathcal{T}_g$ can be defined in an analogous way. Denote by $\{a_t\}_{t \in \mathbf{R}}$ the Teichmüller geodesic flow on $\qut$. Denote by $\pmf := \mf / \mathbf{R}^+$ the space of projective singular measured foliations on $S_g$. 

In \cite{Ma82b}, Masur showed that if a marked quadratic differential $q \in \qut$ is such that its vertical foliation $\Re(q) \in \mf$ is uniquely ergodic, then the corresponding Teichmüller geodesic ray $\{\pi(a_t q)\}_{t \in \mathbf{R}^+} \subseteq \mathcal{T}_g$ converges in the Thurston compactification to the projective measured foliation $[\Re(q)] \in \pmf$. More concretely, if $q \in \qut$ is such that $\Re(q) \in \mf$ is uniquely ergodic, then, for every simple closed curve $\beta$ on $S_g$, the following asymptotic holds as $t \to +\infty$,
\[
\ell_{\beta}(\pi(a_t q)) \sim i(\beta,\Re(q)) \cdot \ell_{\Im(q)}(\pi(a_t q)).
\]

Using Theorem \ref{theo:main} we will prove the following effectivization of Masur's theorem. A stronger version of this result will be introduced in \S9 as Theorem \ref{theo:main_2_unif}. A version of this result that holds for non-simple closed curves will be introduced in \S9 as Theorem \ref{theo:main_2_strong}.

\begin{theorem}
	\label{theo:main_2}
	There exists a constant $\kappa = \kappa(g) > 0$ such that the following holds. For every $X,Y \in \mathcal{T}_g$ there exists a constant $C = C(X,Y) > 0$ and an $(X,Y,C,\kappa)$-sparse subset of mapping classes $M = M(X,Y) \subseteq \mcg$ such that for every simple closed curve $\beta$ on $S_g$ and every $\mc \in \mcg \setminus M$, if $r := d_\mathcal{T}(X,\mc.Y)$ and $q_s := q_s(X,\mc.Y)$, then
	\[
	\ell_{\beta}(\pi(a_r q_s)) = i(\beta,\Re(q_s)) \cdot \ell_{\Im(q_s)}(\pi(a_r q_s)) + O_{X,Y,\beta}\left(e^{(1-\kappa)r}\right).
	\]
\end{theorem}

\subsection*{Comparing the Teichmüller and Thurston metrics.} In analogy with how the Teichmüller metric quantifies the minimal dilation among quasiconfomal maps between marked complex structures, the Thurston metric, introduced by Thurston in \cite{Thu98}, quantifies the minimal Lipschitz constant among Lipschitz maps between marked hyperbolic structures. More concretely, for every pair of marked hyperbolic structures $X,Y \in \mathcal{T}_g$,
\[
d_\mathrm{Thu}(X,Y) := \log\left( \inf_{f \colon X \to Y} \mathrm{Lip}(f)\right),
\]
where the infimum runs over all Lipschitz maps $f \colon X \to Y$ in the homotopy class given by the markings of $X$ and $Y$, and where $\mathrm{Lip}(f)$ denotes the Lipschitz constant of such a map.

Denote by $\mathcal{S}_g$ the set of all simple closed curves on $S_g$. For every $X \in \mathcal{T}_g$ consider the function $D_X \colon \mf \to \smash{\mathbf{R}^+}$ which to every singular measured foliation $\lambda \in \mf$ assigns the value
\begin{equation*}
D_X(\lambda) := \sup_{\beta \in \mathcal{S}_g}\left( \frac{i(\beta,\lambda)}{\ell_\beta(X)}\right).
\end{equation*}

Using Theorem \ref{theo:main} we will prove the following effective estimate comparing the Teichmüller and Thurston metrics along mapping class group orbits of Teichmüller space. A stronger version of this result will be introduced in \S9 as Theorem \ref{theo:main_3_strong}.

\begin{theorem}
	\label{theo:main_3}
	There exists a constant $\kappa = \kappa(g) > 0$ such that the following holds. For every $X,Y \in \mathcal{K}$ there exists a constant $C =C(X,Y) > 0$ and an $(X,Y,C,\kappa)$-sparse subset $M = M(X,Y) \subseteq \mcg$ such that for every $\mc \in \mcg \setminus M$, if $r := d_\mathcal{T}(X,\mc.Y)$, $q_s := q_s(X,\mc.Y)$, and $q_e := q_e(X,\mc.Y)$, then
	\[
	d_\mathrm{Thu}(X,\mc.Y) = d_\mathcal{T}(X,\mc.Y) + \log D_X(\Re(q_s)) + \log \ell_{\Im(q_e)}(\mc.Y) + O_{X,Y}\left(e^{-\kappa r}\right).
	\]
\end{theorem}

\begin{remark}
	The terms $\log D_X(\Re(q_s))$ and $\log \ell_{\Im(q_e)}(\mc.Y)$ in Theorem \ref{theo:main_3} can be bounded uniformly away from $\pm \infty$ when $X$ and $Y$ vary over a compact subset of $\mathcal{T}_g$. Thus, Theorem \ref{theo:main_3} is in tension with the fact that the Teichmüller and Thurston metrics are not expected to differ by a bounded amount.
\end{remark}

\subsection*{Sketch of the proof of Theorem \ref{theo:main}.} By work of Bonahon \cite{Bon88}, weighted closed curves are dense in the space of geodesic currents. Thus, to prove Theorem \ref{theo:main}, it is enough to consider the case where $\alpha$ is a closed curve. Geometric intersection numbers of closed curves can be estimated using the flat geometry of quadratic differentials: the number of transverse intersections between any pair of closed flat geodesics of a quadratic differential is a good approximation of the geometric intersection number of their isotopy classes. See Proposition \ref{prop:flat_int}. In particular, given $\mc \in \mcg$, we can estimate $i(\alpha, \mc.\beta)$ by considering flat geodesic representatives of $\alpha$ and $\mc.\beta$ with respect to the singular flat metric induced by $q_e:= q_e(X,\mc.Y)$ on $S_g$. To estimate the number of transverse intersections between these flat geodesic representatives we proceed in several steps.

First, we use the flat geometry of $q_s := q_s(X,\mc.Y)$ to construct a rectangular decomposition of $\alpha$ on $q_s$, i.e., a concatenation of horizontal and vertical segments of $q_s$ that suitably approximate a flat geodesic representative of $\alpha$ on $q_s$. Transporting this rectangular decomposition to $q_e$ using the Teichmüller geodesic flow yields a rectangular decomposition of $\alpha$ on $q_e$ with long horizontal segments and short vertical segments. We show that the number of transverse intersections between the horizontal segments of this decomposition and any flat geodesic representative of $\mc.\beta$ on $q_e$ is a good approximation of the quantity we are trying to estimate. See Proposition \ref{prop:rect_int}.

Second, we use the flat geometry of $q_e$ to estimate the number of such intersections by constructing an immersed collar around a flat geodesic representative of $\mc.\beta$ on $q_e$ supporting a sufficiently regular bump function whose integral along the maximal horizontal segments of the collar is constant. We show that if we consider the Lebesgue measure on the horizontal segments of the rectangular decomposition of $\alpha$ on $q_e$ constructed above, the integral of the bump function of $\mc.\beta$ with respect to this measure is a good approximation of the quantity we are trying to estimate. See Proposition \ref{prop:bump_4}.

Third, we estimate this integral using the dynamics of the horizontal foliation of $q_e$. More precisely, we use the fact that, under suitable recurrence conditions on the corresponding Teichmüller geodesic, horizontal segments of $q_e$ equidistribute at an effective rate towards the singular flat area form of $q_e$. Following the general approach of Athreya and Forni \cite{AF08}, we prove a version of this fact suitable to our purposes. See Theorem \ref{theo:equid_3}. Putting together the approximations above yields an estimate with the desired leading term. See Theorems \ref{theo:prelim_1} and \ref{theo:prelim_2}.

To ensure the relevant Teichmüller geodesics satisfy the desired recurrence conditions, we discard a sparse subset of mapping classes. The estimates of Eskin, Mirzakhani, and Rafi on the number of thin Teichmüller geodesic segments joining a point in Teichmüller space to points in a mapping class group orbit of Teichmüller space \cite{EMR19} play a crucial role in this step.

It remains to control the quality of the approximations. The error terms of the approximations are large when either the length of the shortest saddle connections of $q_s$ or $q_e$ is small, or when the minimal slope of a flat geodesic representative of $\mc.\beta$ on $q_e$ is small. Using methods introduced in the prequel \cite{Ara20b}, we discard a sparse subset of mapping classes to control these quantities in a way that guarantees the approximations have a power saving error term. See Theorems \ref{theo:sparse_1} and \ref{theo:sparse_4}. 

\subsection*{Organization of the paper.} In \S 2 we discuss some aspects of the flat geometry of quadratic differentials. In \S 3 we introduce a method for constructing rectangular decompositions of flat geodesics of quadratic differentials. In \S 4 we describe a procedure for constructing immersed collars and bump functions of flat geodesics of quadratic differentials. In \S 5 we show the horizontal segments of a quadratic differential equidistribute at an effective rate towards the singular flat area form of the quadratic differential under suitable recurrence conditions on the corresponding Teichmüller geodesic. In \S 6 we combine results from \S2 -- 5 to prove a preliminary quantitative estimate for the geometric intersection numbers of closed curves with respect to arbitrary geodesic currents. In \S 7 we show the recurrence conditions needed to apply this preliminary estimate hold in most cases of interest. In \S 8 we show how to control the error terms in this preliminary estimate. In \S9 we combine results from \S 6 -- 8 to prove Theorems \ref{theo:main}, \ref{theo:main_2}, and \ref{theo:main_3}, as well as the different versions of them alluded to above.

\subsection*{Notation.} Let $A,B \in \mathbf{R}$ and $*$ be a set of parameters. We write $A \preceq_* B$ if there exists a constant $C= C(*) > 0$ depending only on $*$ such that $A \leq C \cdot B$. We write $A \asymp_* B$ if $A \preceq_* B $ and $B \preceq_* A$. We write $A = O_*(B)$ if there exists a constant $C = C(*) > 0$ depending only on $*$ such that $|A| \leq C \cdot B$.

\subsection*{Acknowledgments.} The author is very grateful to Alex Wright and Steve Kerckhoff for their invaluable advice, patience, and encouragement. The author would also like to thank Alex Eskin, Ian Frankel, Jayadev Athreya, and Jon Chaika for very helpful and enlightening conversations. This work got started while the author was participating in the \textit{Dynamics: Topology and Numbers} trimester program at the Hausdorff Research Institute for Mathematics (HIM). The author is very grateful for the HIM's hospitality and for the hard work of the organizers of the trimester program.

\section{Flat geometry of quadratic differentials}

\subsection*{Outline of this section.} In this section we discuss some aspects of the flat geometry of quadratic differentials. The main result of this section is Proposition \ref{prop:flat_int}, which shows that the number of transverse intersections between any pair of closed flat geodesics of a quadratic differential is a good approximation of the geometric intersection number of their isotopy classes. Proposition \ref{eq:int_bound}, which bounds the number of intersections between non-parallel straight line segments of quadratic differentials, will also play an important role in later sections of this paper.

\subsection*{Quadratic differentials.} Let $X$ be a closed Riemann surface and $K$ be its canonical bundle. A \textit{quadratic differential} $q$ on $X$ is a holomorphic section of the symmetric square $K \vee K$. In local coordinates, $q = f(z) \thinspace dz^2$ for some holomorphic function $f(z)$. If $X$ has genus $g$, the number of zeroes of $q$ counted with multiplicity is $4g-4$. The zeroes of $q$ are also called \textit{singularities}. We sometimes denote quadratic differentials by $(X,q)$ to keep track of the Riemann surface they are defined on.

A \textit{half-translation structure} on a surface $S$ is an atlas of charts to $\mathbf{C}$ on the complement of a finite set of points $\Sigma \subseteq S$ whose transition functions are of the form $z \mapsto \pm z + c$ with $c \in \mathbf{C}$. Every quadratic differential $q$ gives rise to a half-translation structure on the Riemann surface it is defined on by considering local coordinates on the complement of the zeroes of $q$ for which $q = dz^2$. Viceversa, every half-translation structure induces a quadratic differential on its underlying surface by pulling back the differential $dz^2$ on the corresponding charts. 

The notion of \textit{straight line} makes sense for a surface endowed with a half translation structure and in particular for a closed Riemann surface endowed with a quadratic differential $q$. A \textit{cylinder curve} of $q$ is a closed straight line intersecting no zeroes. A \textit{saddle connection} of a $q$ is a straight line segment joining two zeroes and having no zeroes in its interior. The notions of \textit{absolute value of slope} and \textit{parallelism} of straight line segments also make sense in this context.

Pulling back the standard Euclidean metric on $\mathbf{C}$ using the charts of a half-translation structure induces a singular flat metric on the underlying surface. In particular, every quadratic differential $q$ gives rise to a singular flat metric on the Riemann surface $X$ it is defined on. This metric is smooth away from the zeroes of $q$ and has a singularity of cone angle $(k+2) \pi$ at every zero of order $k$. The diameter of $q$, denoted $\mathrm{diam}(q)$, is the diameter of $X$ with respect to this metric. Denote by $A_q$ the singular flat area form induced by $q$ on $X$. The area of $q$, denoted $\mathrm{Area}(q)$, is the area of $X$ with respect to $A_q$. Denote by $\ell_{\alpha}(q)$ the flat length of a saddle connection $\alpha$ of $q$ and by $\ell_{\min}(q)$ the flat length of the shortest saddle connections of $q$. Denote $\smash{\ell_{\min}^\dagger(q)} := \min\{1,\ell_{\min}(q)\}$. Denote by $\ell_{\beta}(q)$ the flat length of a closed curve $\beta$ on $X$ and by $\mathrm{sys}(q)$ the flat length of the shortest not null-homotopic closed curves on $X$. 

Pulling back the measured foliations corresponding to the $1$-forms $dx$ and $dy$ on $\mathbf{C}$ using the charts of a half-translation structure induces a pair of singular measured foliation on the underlying surface. For the half translation structure induced by a quadratic differential $q$, we denote these singular measured foliations by $\Re(q)$ and $\Im(q)$, and refer to them as the \textit{vertical} and \textit{horizontal} foliations of $q$. Segments of leaves of  $\Re(q)$ and $\Im(q)$ are called \textit{vertical} and \textit{horizontal}, respectively.

\subsection*{Flat geodesics.} Closed geodesics with respect to the singular flat metric induced by a quadratic differential on its underlying Riemann surface can be described explicitly as follows.

\begin{proposition}
	\label{prop:flat_geod_1}
	Let $q$ be a quadratic differential on a closed Riemann surface $X$. Then, a closed geodesic with respect to the singular flat metric induced by $q$ on $X$ must either be a cylinder curve or a concatenation of saddle connections meeting at angles $\geq \pi$ on both sides. 
\end{proposition}

We refer to flat geodesics of quadratic differentials that are not cylinder curves as \textit{singular flat geodesics}. Let us recall the following standard fact. This fact implies that if $q$ is a quadratic differential on a closed Riemann surface $X$ of genus $g \geq 2$, then $\ell_{\min}(q) \leq \mathrm{sys}(q)$.

\begin{proposition}
	\label{prop:flat_rep}
	Let $q$ be a quadratic differential on a closed Riemann surface $X$ of genus $g \geq 2$. Then, in any homotopy class of closed curves of $X$ there exists a flat geodesic representative. Moreover, this representative is unique except when it is homotopic to a cylinder curve, in which case there exists a full cylinder worth of flat geodesic representatives bounded by singular flat geodesics.
\end{proposition}

Let $q$ be a quadratic differential on a closed Riemann surface $X$ of genus $g \geq 2$. Identify the universal cover of $X$ with the open unit disk $\mathbf{D} \subseteq \mathbf{C}$ in the complex plane. The singular flat metric induced by $q$ on $X$ lifts to a singular flat metric on $\mathbf{D}$ whose geodesics can be characterized as in Proposition \ref{prop:flat_geod_1}. We say a curve on $\mathbf{D}$ is \textit{simple} if it does not intersect itself. We say two simple curves $\alpha$ and $\beta$ on $\mathbf{D}$ form a \textit{bigon} if there exists an embedded closed disk in $\mathbf{D}$ whose boundary is the union of an arc of $\alpha$ and an arc of $\beta$ intersecting at exactly two points. A direct application of the Jordan curve theorem and the Gauss-Bonnet theorem yields the following result.

\begin{proposition}
	\label{prop:flat_geod_2}
	Let $q$ be a quadratic differential on a closed Riemann surface $X$ of genus $g \geq 2$. Identify the universal cover of $X$ with the open unit disk $\mathbf{D} \subseteq \mathbf{C}$ on the complex plane and lift the singular flat metric induced by $q$ on $X$ to a singular flat metric on $\mathbf{D}$. Then, flat geodesics on $\mathbf{D}$ are simple. Furthermore, pairs of flat geodesics on $\mathbf{D}$ do not form bigons.
\end{proposition}

\subsection*{Geometric intersection numbers.} Let $S$ be a closed surface. The \textit{geometric intersection number} $i(\alpha, \beta)$ of a pair of closed curves $\alpha$ and $\beta$ on $S$ is defined as the minimal number of intersections among pairs of transverse closed curves homotopic to $\alpha$ and $\beta$. 

Let $q$ be a quadratic differential on a closed Riemann surface $X$ of genus $g \geq 2$. A pair of flat geodesics $\alpha$ and $\beta$ on $X$ might intersect at a zero of $q$ and/or might share an arc. We refer to such intersections as \textit{non-transverse}. These intersections can be homotoped away if and only if the corresponding lifts of $\alpha$ and $\beta$ to $\mathbf{D}^2$ have unlinked endpoints on the boundary at infinity $\mathbf{S}^1 := \partial \mathbf{D}$. See Figure \ref{fig:non_transverse} for examples. Denote by $I(\alpha,\beta)$ the number of transverse intersections between $\alpha$ and $\beta$. The following proposition is the main result of this section. 

\begin{figure}[h]
	\centering
	\begin{subfigure}[b]{0.4\textwidth}
		\centering
		\includegraphics[width=0.6\textwidth]{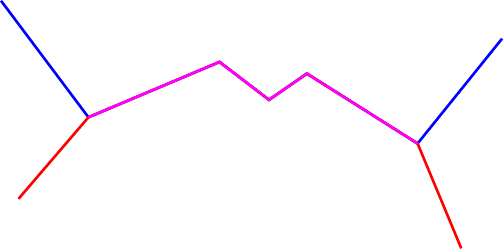}
		\caption{Can be homotoped away.}
	\end{subfigure}
	\quad \quad \quad
	~ 
	\begin{subfigure}[b]{0.4\textwidth}
		\centering
		\includegraphics[width=0.6\textwidth]{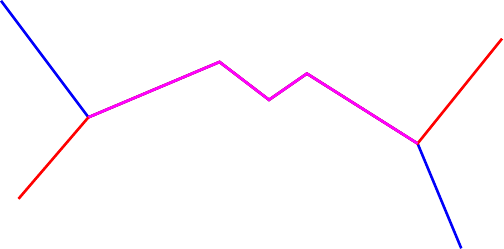}
		\caption{Cannot be homotoped away.}
	\end{subfigure}
	\caption{Examples of non-transverse intersections of flat geodesics.} 
	\label{fig:non_transverse}
\end{figure}

\begin{proposition}
	\label{prop:flat_int}
	Let $q$ be a quadratic differential on a closed Riemann surface $X$ of genus $g \geq 2$. Suppose $\alpha$ and $\beta$ are closed flat geodesics of $q$ with $n$ and $m$ saddle connections counted with mutiplicity and the convention that $n = 0$ and/or $m = 0$ if $\alpha$ and/or $\beta$ are cylinder curves. Then,
	\[
	I(\alpha,\beta) \leq i(\alpha,\beta) \leq I(\alpha,\beta) + n \cdot m.
	\]
\end{proposition}
	
\begin{proof}
	Notice that there are at most $n \cdot m$ non-transverse intersections between $\alpha$ and $\beta$. One can homotope these intersections to obtain a pair of transverse closed curves with at most $I(\alpha,\beta) + n \cdot m$ intersections. This proves the upper bound. By Proposition \ref{prop:flat_geod_2}, lifts of $\alpha$ and $\beta$ to $\mathbf{D}$ are simple and do not form bigons. In particular, as homotopies of $\alpha$ and $\beta$ on $X$ do not change the endpoints on the boundary at infinity $\mathbf{S}^1 = \partial \mathbf{D}$ of their lifts, the transverse intersections of $\alpha$ and $\beta$ cannot be homotoped away; see Figure \ref{fig:intersection}. This proves the lower bound.
\end{proof}

\begin{figure}[h!]
	\centering
	\includegraphics[width=.3\textwidth]{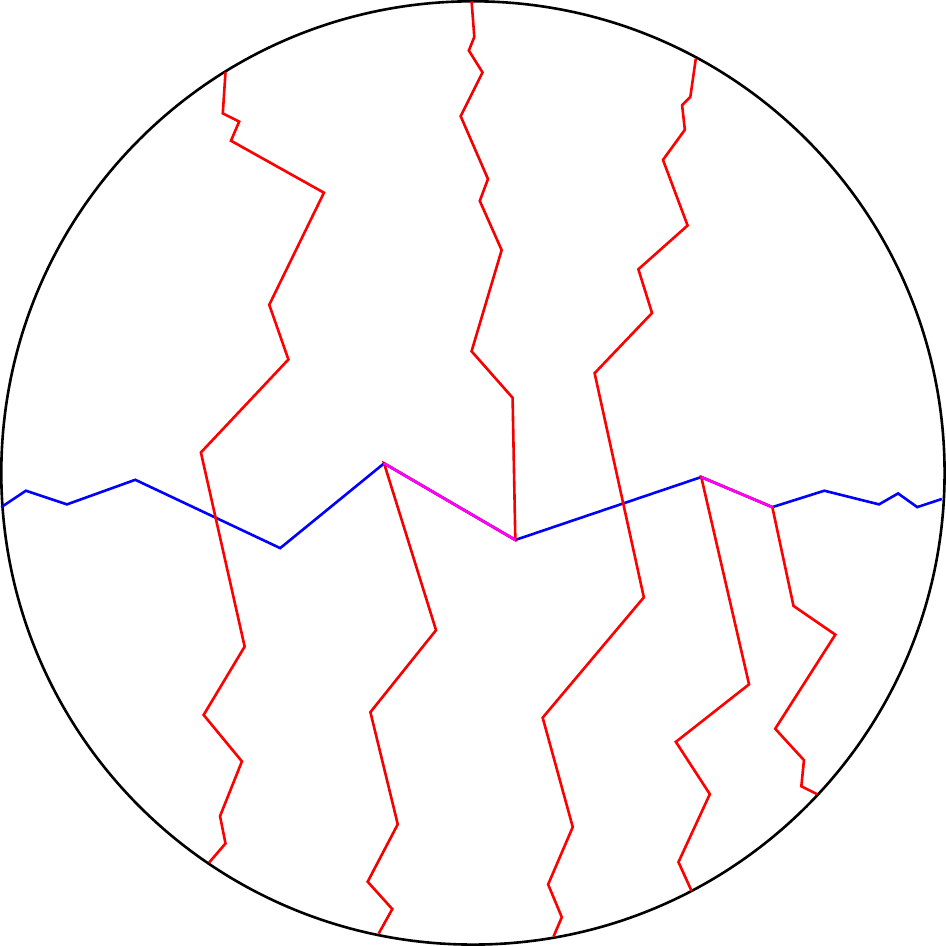}
	\caption{Proof of the lower bound of Proposition \ref{prop:flat_int}.} \label{fig:intersection} 
\end{figure}

Propositions \ref{prop:flat_rep}  and \ref{prop:flat_int} will allow us to estimate geometric intersection numbers of closed curves using the flat geometry of quadratic differentials.

\subsection*{Intersections of straight line segments.} We end this section with an estimate on the number of intersections between non-parallel straight line segments of quadratic differentials. This estimate will play an important role in several proofs of this paper.

\begin{proposition}
	\label{eq:int_bound}
	Let $q$ be a quadratic differential on a closed Riemann surface $X$. Suppose $\beta$ and $\sigma$ are non-parallel straight line segments on $X$ with $\ell_{\sigma}(q) \leq \mathrm{sys}(q)/2$. Then,
	\[
	\# (\beta \cap \sigma) \leq 1 + 2 \ell_{\beta}(q)/\mathrm{sys}(q).
	\]
\end{proposition}

\begin{proof}
	Traversing $\beta$ between any two consecutive intersections with $\sigma$ and connecting these intersections along $\sigma$ yields a simple closed curve with turning angle at most $\pi$. See Figure \ref{fig:int_bound} for an example. In particular, by the Gauss-Bonnet theorem, this curve is not null-homotopic. It follows that the distance $\ell$ along $\beta$ between any two consecutive intersections with $\sigma$ satisfies
	\[
	\ell \geq \mathrm{sys}(q) - \ell_{\sigma}(q) \geq \mathrm{sys}(q)/2.
	\]
	As this holds for every pair of consecutive intersections of $\beta$ with $\sigma$, 
	\[
	\ell_{\beta}(q) \geq (\# (\beta \cap \sigma) - 1) \cdot \mathrm{sys}(q)/2.
	\]
	Rearranging terms in this bound we conclude
	\[
	\# (\beta \cap \sigma) \leq 1 + 2 \ell_{\beta}(q)/\mathrm{sys}(q). \qedhere
	\]
	\begin{figure}[h!]
		\centering
		\includegraphics[width=.15\textwidth]{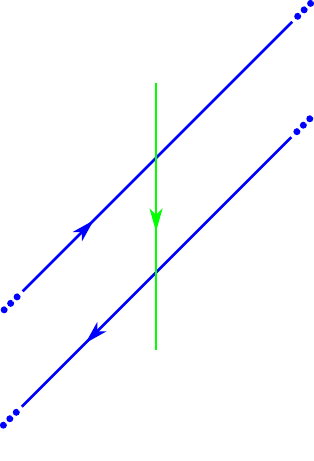}
		\caption{Proof of Proposition \ref{eq:int_bound}.} \label{fig:int_bound} 
	\end{figure}
\end{proof}

\section{Rectangular decompositions of flat geodesics}

\subsection*{Outline of this section} In this section we introduce a method for constructing rectangular decompositions of flat geodesics of quadratic differentials. The main result of this section is Proposition \ref{prop:rect_int}, which shows that these decompositions can be used to estimate the number of transverse intersections between pairs of flat geodesics of quadratic differentials. We first consider the case of cylinder curves and then the case of saddle connections and singular flat geodesics.

\subsection*{Rectangular decompositions of cylinder curves.} Let $q$ be a quadratic differential on a closed Riemann surface $X$ of genus $g \geq 2$ and $\alpha$ be a cylinder curve of $q$. By a \textit{rectangular decomposition} of $\alpha$ we mean a concatenation of vertical and horizontal segments of $q$ that closely approximates $\alpha$. We construct a rectangular decomposition of $\alpha$ by snugly covering it with finitely many embedded, flat, non-singular rectangles and then drawing a rectangular decomposition as in Figure \ref{fig:rect_decomp_cyl}.

\begin{figure}[h!]
	\centering
	\includegraphics[width=.4\textwidth]{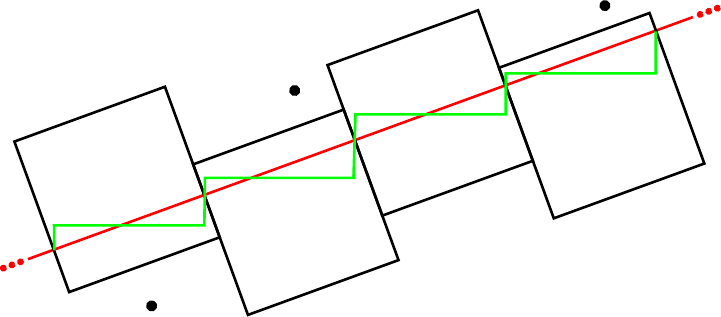}
	\caption{Constructing a rectangular decomposition of a cylinder curve.} \label{fig:rect_decomp_cyl} 
\end{figure}

Fix an orientation on $\alpha$ and consider a unit speed parametrization $\alpha \colon [0,\ell_{\alpha}(q)] \to X$ consistent with this orientation. Starting from any point of $\alpha$ we can flow orthogonal to $\alpha$ at unit speed for a maximal interval of times until we hit a zero of $q$. For every $t \in [0,\ell_{\alpha}(q)]$ denote by $[s_{\min}(\alpha(t)),s_{\max}(\alpha(t))]$ this interval. For every $t \in [0,\ell_{\alpha}(q)]$ and every $s \in [s_{\min}(\alpha(t)),s_{\max}(\alpha(t))]$ denote by $p_s(\alpha(t)) \in X$ the point reached by flowing orthogonal to $\alpha$ at unit speed from $\alpha(t)$ for time $s$.

To construct the desired rectangular decomposition we first find embedded, flat, non-singular rectangles around $\alpha$. See Figure \ref{fig:embedded_square} for an example. Recall that $\ell_{\min}(q)$ denotes the length of the shortest saddle connections of $q$. Fix $t_0,t_1 \in [0,\ell_{\alpha}(q)]$ with $0 < t_1 - t_0 \leq \ell_{\min}(q)/4$. Consider the parameters 
\begin{gather*}
a := \min_{t \in [t_0,t_1]} s_{\min}(\alpha(t)) < 0, \\
b := \max_{t \in [t_0,t_1]} s_{\max}(\alpha(t)) > 0. 
\end{gather*} 
Notice that it can never be the case that $a \geq - \ell_{\min}(q)/8$ and $b \leq \ell_{\min}(q)/8$, else $q$ would have a saddle connection of length $\leq \ell_{\min}(q)/2$. Consider $\Delta \in (-\ell_{\min}(q)/8,\ell_{\min}(q)/8)$ given by
\[
\Delta := \left \lbrace
\begin{array}{c l}
a/2 + \ell_{\min}(q)/8 & \text{if } a \geq - \ell_{\min}(q)/8,\\
-\ell_{\min}(q)/8 + b/2 & \text{if } b \leq \ell_{\min}(q)/8,\\
0 & \text{if } a < - \ell_{\min}(q)/8 \text{ and } b > \ell_{\min}(q)/8.
\end{array} \right. 
\]
Using this parameter we define a map $R \colon [t_0,t_1] \times [-\ell_{\min}(q)/8,\ell_{\min}(q)/8] \to X$ as follows,
\[
R(t,s) := p_{s + \Delta}(\alpha(t)). \label{eq:square}
\]
By construction, this map is a flat immersion whose image contains no zeroes of $q$. Furthermore, as the following proposition shows, this map is an embedding. Denote by $R(\alpha,t_0,t_1) \subseteq X$ the embedded, flat, non-singular rectangle obtained as the image of this map.

\begin{figure}[h!]
	\centering
	\includegraphics[width=.15\textwidth]{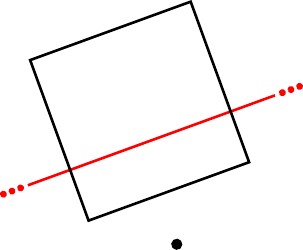}
	\caption{Embedded, flat, non-singular rectangle around a cylinder curve.} \label{fig:embedded_square} 
\end{figure}

\begin{proposition}
	\label{prop:embed}
	The map $R \colon [t_0,t_1] \times [-\ell_{\min}(q)/8,\ell_{\min}(q)/8] \to X$ is an embedding.
\end{proposition}

\begin{proof}
	Suppose for the sake of contradiction that two points in $[t_0,t_1] \times [-\ell_{\min}(q)/8,\ell_{\min}(q)/8]$ get identified through $R$. Consider a path between these points as in Figure \ref{fig:embed}. The image of this path under $R$ is a closed curve of flat length $\leq \ell_{\min}(q)/2$ and total turning angle $\leq \pi$. Without loss of generality we can assume this curve is simple. Then, by the Gauss-Bonnet theorem, this curve is not null-homotopic. It follows that $\mathrm{sys}(q) \leq \ell_{\min}(q)/2$, which is a contradiction.
\end{proof}

\begin{figure}[h!]
	\centering
	\includegraphics[width=.1\textwidth]{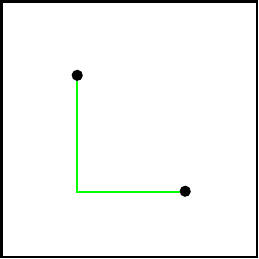}
	\caption{Proof of Proposition \ref{prop:embed}.} \label{fig:embed} 
\end{figure}

Let $0 =: t_0 < t_1 < \dots < t_{n-1} < t_n := \ell_\alpha(q)$ be a partition of $[0,\ell_{\alpha}(q)]$ such that $t_{i} - t_{i-1} = \ell_{\min}(q)/4$ for every $i \in \{1,\dots,n-1\}$ and $t_n - t_{n-1} \leq \ell_{\min}(q)/4$. In particular, $n \preceq \ell_{\alpha}(q)/ \ell_{\min}(q)$. Consider the finite collection of embedded, flat, non-singular rectangles in $X$ given by $\{R(\alpha,t_i,t_{i+1})\}_{i=0}^{n-1}$. Denote $\alpha_i := \alpha([t_i,t_{i+1}]) \subseteq R(\alpha,t_i,t_{i+1})$ for every $i \in \{0,\dots,n-1\}$. In each rectangle $R(\alpha,t_i,t_{i+1})$ consider a rectangular decomposition of $\alpha_i$ as in Figure \ref{fig:rect_decomp}; there exists a whole interval worth of such decompositions except when $\alpha_i$ is vertical or horizontal. Joining these rectangular decompositions yields a rectangular decomposition of $\alpha$. Denote by $i(\alpha,\Re(q))$ the total transverse measure of $\alpha$ with respect to $\Re(q)$; this quantity is equal to the geometric intersection number of $\alpha$ and $\Re(q)$ as defined by Thurston \cite[Exposé 5]{FLP12}.

\begin{figure}[h!]
	\centering
	\includegraphics[width=.15\textwidth]{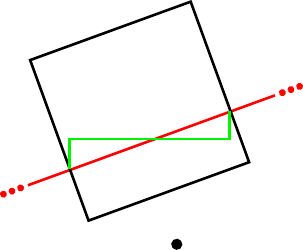}
	\caption{Rectangular decomposition in an embedded, flat, non-singular rectangle.} \label{fig:rect_decomp} 
\end{figure}

\begin{proposition}
	\label{prop:rect_decomp_cyl}
	Let $q$ be a quadratic differential on a closed Riemann surface $X$ of genus $g \geq 2$ and $\alpha$ be a cylinder curve of $q$. The construction above yields a rectangular decomposition of $\alpha$ with $n \preceq \ell_{\alpha}(q)/ \ell_{\min}(q)$ horizontal and vertical segments of length $\leq \ell_{\min}(q)/2$. The sum of the lengths of the horizontal segments of this decomposition is equal to $i(\alpha,\Re(q))$.
\end{proposition}

\subsection*{Rectangular decompositions of saddle connections.} Let $q$ be a quadratic differential on a closed Riemann surface $X$ of genus $g \geq 2$ and $\alpha$ be a saddle connection of $q$. We construct a rectangular decomposition of $\alpha$ using a procedure similar to the one introduced above. See Figure \ref{fig:rect_decomp_sad} for an example. 

\begin{figure}[h!]
	\centering
	\includegraphics[width=.4\textwidth]{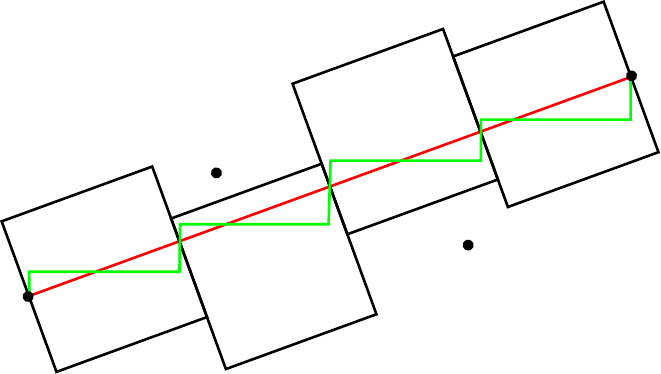}
	\caption{Constructing a rectangular decomposition of a saddle connection.} \label{fig:rect_decomp_sad} 
\end{figure}

Fix an orientation on $\alpha$ and consider a unit speed parametrization $\alpha \colon [0,\ell_{\alpha}(q)] \to X$ consistent with this orientation. Let $t_0,t_1 \in [0,\ell_{\alpha}(q)]$ with $0 < t_1 - t_0 \leq \ell_{\min}(q)/4$. Define a flat immersion $R \colon [t_0,t_1] \times [-\ell_{\min}(q)/8,\ell_{\min}(q)/8] \to X$ in the same way as above; it makes sense to flow perperdicular to $\alpha$ starting from any of its endpoints. Notice that if $t_0 = 0$ or $t_1 = \ell_{\alpha}(q)$, then $\Delta = 0$. An argument similar to the one used in the proof of Proposition \ref{prop:embed} shows that the map $R$ is actually an embedding. Just as above, denote by $R(\alpha,t_0,t_1) \subseteq X$ the embedded, flat rectangle obtained as the image of this map. The only situation in which this rectangle contains a singularity is when $t_0 = 0$ or $t_1 = \ell_{\alpha}(q)$, in which case the only singularity it contains is one of the endpoints of $\alpha$.

Let $0 =: t_0 < t_1 < \dots < t_{n-1} < t_n := \ell_\alpha(q)$ be a partition of $[0,\ell_{\alpha}(q)]$ such that $t_{i} - t_{i-1} = \ell_{\min}(q)/4$ for every $i \in \{1,\dots,n-1\}$ and $t_n - t_{n-1} \leq \ell_{\min}(q)/4$. In particular, $n \preceq \ell_{\alpha}(q)/ \ell_{\min}(q)$. Consider the collection of embedded, flat rectangles in $X$ given by $\{R(\alpha,t_i,t_{i+1})\}_{i=0}^{n-1}$. Denote $\alpha_i := \alpha([t_i,t_{i+1}]) \subseteq R(\alpha,t_i,t_{i+1})$ for every $i \in \{0,\dots,n-1\}$. In each rectangle $R(\alpha,t_i,t_{i+1})$ consider a rectangular decomposition of $\alpha_i$ as in Figures \ref{fig:rect_decomp} or \ref{fig:rect_decomp_2}; there exists a whole interval worth of such decompositions except when $\alpha_i$ is vertical or horizontal. Joining these rectangular decompositions yields a rectangular decomposition of $\alpha$. Denote by $i(\alpha,\Re(q))$ the total transverse measure of $\alpha$ with respect to $\Re(q)$.

\begin{figure}[h!]
	\centering
	\includegraphics[width=.15\textwidth]{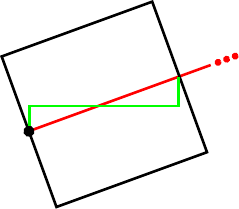}
	\caption{Rectangular decomposition in an embedded, flat rectangle.} \label{fig:rect_decomp_2} 
\end{figure}

\begin{proposition}
	\label{prop:rect_decomp_sad}
	Let $q$ be a quadratic differential on a closed Riemann surfaces $X$ of genus $g \geq 2$ and $\alpha$ be a saddle connection of $q$. The construction above yields a rectangular decomposition of $\alpha$ with $n \preceq \ell_{\alpha}(q)/ \ell_{\min}(q)$ horizontal and vertical segments of length $\leq \ell_{\min}(q)/2$. The sum of the lengths of the horizontal segments of this decomposition is equal to $i(\alpha,\Re(q))$.
\end{proposition}

\subsection*{Rectangular decompositions of singular flat geodesics.} Let $q$ be a quadratic differential on a closed Riemann surface $X$ of genus $g \geq 2$ and $\alpha$ be a singular flat geodesic of $q$. We can construct a rectangular decomposition of $\alpha$ by joining rectangular decompositions of the saddle connections of $\alpha$ constructed using the method introduced above; concatenations of consecutive vertical or horizontal segments are allowed. Notice that $\alpha$ has $\preceq \ell_{\alpha}(q)/\ell_{\min}(q)$ saddle connections. Denote by $i(\alpha,\Re(q))$ the total transverse measure of $\alpha$ with respect to $\Re(q)$; this quantity is equal to the geometric intersection number of $\alpha$ and $\Re(q)$ as defined by Bonahon \cite{Bon88}.

\begin{proposition}
	\label{prop:rect_decomp_sing}
	Let $q$ be a quadratic differential on a closed Riemann surface $X$ of genus $g \geq 2$ and $\alpha$ be a singular flat geodesic of $q$. The construction above yields a rectangular decomposition of $\alpha$ with $n \preceq \ell_{\alpha}(q)/ \ell_{\min}(q)$ horizontal and vertical segments of length $\leq \ell_{\min}(q)/2$. The sum of the lengths of the horizontal segments of this decomposition is equal to $i(\alpha,\Re(q))$.
\end{proposition}

\subsection*{The Teichmüller geodesic flow.} The group $\mathrm{SL}(2,\mathbf{R})$ acts naturally on half-translation structures by postcomposing the corresponding charts with the linear action of $\mathrm{SL}(2,\mathbf{R})$ on $\mathbf{C} =\mathbf{R}^2$. In particular, $\mathrm{SL}(2,\mathbf{R})$ acts naturally on quadratic differentials. The \textit{Teichmüller geodesic flow} on quadratic differentials corresponds to the action of the diagonal subgroup $\{a_t\}_{t \in \mathbf{R}} \subseteq \mathrm{SL}(2,\mathbf{R})$ given by
\begin{equation}
\label{eq:diag}
a_t := \left( 
\begin{array}{c c}
e^t & 0 \\
0 & e^{-t}
\end{array} 
\right).
\end{equation}

The vertical and horizontal segments of a quadratic differential are preserved by the Teichmüller geodesic flow. Flowing for time $t \in \mathbf{R}$ dilates horizontal segments by $e^t$ and contracts vertical segments by $e^{-t}$. More generally, the straight line segments of a quadratic differential are preserved by the Teichmüller geodesic flow. Given $t \in \mathbf{R}$, a quadratic differential $q$, and a straight line segment $\gamma$ of $q$, denote by $A_t \gamma$ the transport of $\gamma$ to $a_t q$. By Proposition \ref{prop:flat_geod_1}, the flat geodesics of a quadratic differentials are also preserved by the Teichmüller geodesic flow. Given $t \in \mathbf{R}$, a quadratic differential $q$, and a flat geodesic $\alpha$ of $q$, denote by $A_t\alpha$ the transport of $\alpha$ to $a_tq$.

\subsection*{Geometric intersection numbers and rectangular decompositions.} Recall that if $\alpha$ an $\beta$ are flat geodesics of a quadratic differential, then $I(\alpha,\beta)$ denotes the number of transverse intersections between $\alpha$ and $\beta$ as defined in \S 2. The following proposition shows that that the rectangular decompositions constructed above can be used to estimate the number of transverse intersections between pairs of flat geodesics of quadratic differentials. The statement incorporates the Teichmüller geodesic flow motivated by applications in later sections.

\begin{proposition}
	\label{prop:rect_int}
	Let $q$ be a quadratic differential on a closed Riemann surface $X$ of genus $g \geq 2$ and $\alpha$ be a flat geodesic of $q$. Then, there exists a sequence $\smash{\{\gamma_i\}_{i=1}^n}$ of $n \preceq \ell_{\alpha}(q)/\ell_{\min}(q)$ horizontal segments of $q$ of lengths $\ell_{\gamma_i}(q) \leq \ell_{\min}(q)/2$, such that $\sum_{i=1}^n \ell_{\gamma_i}(q) = i(\alpha,\Re(q))$, and satisfying the following property. Let $t \in \smash{\mathbf{R}^+}$, $\beta$ be a flat geodesic of $a_t q$ with no horizontal segments, and $\smash{\{\beta_j\}_{j=1}^m}$ be the sequence of saddle connections of $\beta$ or $\beta$ itself if $\beta$ is a cylinder curve. Then,
	\[
	\bigg\vert I(A_t \alpha,\beta) - \sum_{i=1}^k \sum_{j=1}^m \#(A_t\gamma_i \cap \beta_j)\bigg\vert \preceq \frac{\ell_\alpha(q) \cdot \ell_{\beta}(a_t q)}{\ell_{\min}(q) \cdot \ell_{\min}(a_t q)}.
	\]
\end{proposition}

\begin{proof}
	Consider a rectangular decomposition of $\alpha$ on $q$ constructed using the methods introduced above. Let $\smash{\{\gamma_i\}_{i=1}^n}$ with $n \preceq \ell_{\alpha}(q)/\ell_{\min}(q)$ be the sequence of horizontal segments of this rectangular decomposition. By construction, $\ell_{\gamma_i}(q) \leq \ell_{\min}(q)/2$ for every $i \in \{1,\dots,n\}$ and $\sum_{i=1}^n \ell_{\gamma_i}(q) = i(\alpha,\Re(q))$. It remains to check that $\smash{\{\gamma_i\}_{i=1}^n}$ satisfies the desired property. Let $t \in \smash{\mathbf{R}^+}$, $\beta$ be a flat geodesic of $a_t q$ with no horizontal segments, and $\smash{\{\beta_j\}_{j=1}^m}$ be the sequence of saddle connections of $\beta$ or $\beta$ itself if $\beta$ is a cylinder curve.
	
	 Denote by $\smash{\alpha_*}$ one of the saddle connections of $\alpha$ or $\alpha$ itself if $\alpha$ is a cylinder curve. We assume that $\smash{\alpha_*}$ is neither vertical nor horizontal. These simpler cases can be dealt with appealing to similar arguments. Let $R := \smash{R(\alpha_*,t_0,t_1)} \subseteq X$ be one of the embedded flat rectangles used in the construction of the rectangular decomposition of $\alpha_*$. Transporting this rectangle to $a_t q$ yields an embedded flat parallelogram $A_t R$ on the corresponding Riemann surface. Drawing the segments of the flat geodesic $\beta$ that intersect this parallelogram yields a picture as in Figure \ref{fig:rect_decomp_teich}.
	
	\begin{figure}[h!]
		\centering
		\begin{tabular}{c c c}
			\includegraphics[scale=.8]{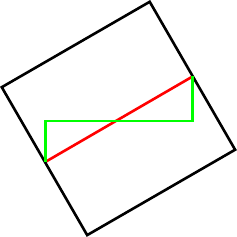} & \begin{tabular}{c} $\xrightarrow{\left(\begin{array}{c c}
				e^t & 0 \\
				0 & e^{-t}
				\end{array}\right)} $  \\ \\ \\ \\ \\ \\ \end{tabular}& \begin{tabular}{c} \includegraphics[scale=.8]{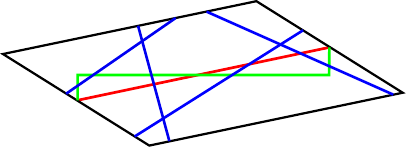} \\[1.2cm] \end{tabular}
		\end{tabular}
		\vspace{-1.6cm}
		\captionsetup{width=\linewidth}
		\caption{Transporting an embedded flat rectangle using the Teichmüller geodesic flow.} \label{fig:rect_decomp_teich} 
	\end{figure}

	Denote $\smash{\alpha_*^\dagger} := \alpha_*([t_0,t_1])\subseteq R$ and $\smash{A_t \alpha_*^\dagger} \subseteq A_t R$ the transport of $\alpha_*^\dagger$ to $a_t q$. Each segment of $\beta$ in $A_t R$ fits into exactly one of two categories. It is either a \textit{good segment}, meaning it does not intersect $\smash{A_t \alpha_*^\dagger}$ on the boundary of $A_t R$ and intersects $\smash{A_t \alpha_*^\dagger}$ the same number of times it intersects the horizontal segment of the rectangular decomposition of $\smash{A_t \alpha_*^\dagger}$ in $A_t R$, or it is a \textit{bad segment}, meaning it intersects $\smash{A_t \alpha_*^\dagger}$ on the boundary of $A_t R$ or it does not intersect $\smash{A_t \alpha_*^\dagger}$ the same number of times it intersects the horizontal segment of the rectangular decomposition of $\smash{A_t \alpha_*^\dagger}$ in $A_t R$. See Figures \ref{fig:good_seg} and \ref{fig:bad_seg} for pictures of all the possible cases in each category.
	
	\begin{figure}[h]
		\centering
		\begin{subfigure}[b]{0.4\textwidth}
			\centering
			\includegraphics[width=0.6\textwidth]{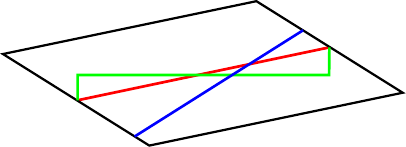}
			\caption{Case I.}
		\end{subfigure}
		\quad \quad \quad
		~ 
		\begin{subfigure}[b]{0.4\textwidth}
			\centering
			\includegraphics[width=0.6\textwidth]{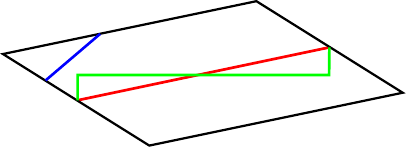}
			\caption{Case II.}
		\end{subfigure}
		\caption{Good segments.} 
		\label{fig:good_seg}
	\end{figure}

	\begin{figure}[h]
		\centering
		\begin{subfigure}[b]{0.4\textwidth}
			\centering
			\includegraphics[width=0.6\textwidth]{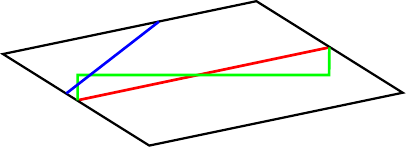}
			\caption{Case I.}
		\end{subfigure}
		\quad \quad \quad
		~ 
		\begin{subfigure}[b]{0.4\textwidth}
			\centering
			\includegraphics[width=0.6\textwidth]{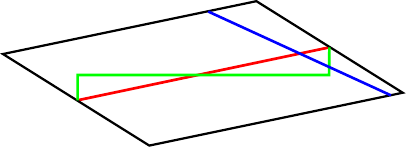}
			\caption{Case II.}
		\end{subfigure} 
		\\[.5cm]
		\begin{subfigure}[b]{0.4\textwidth}
			\centering
			\includegraphics[width=0.6\textwidth]{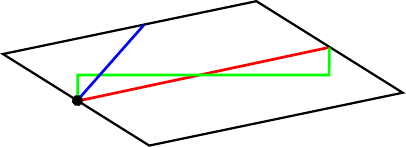}
			\caption{Case III.}
		\end{subfigure}
		\quad \quad \quad
		~ 
		\begin{subfigure}[b]{0.4\textwidth}
			\centering
			\includegraphics[width=0.6\textwidth]{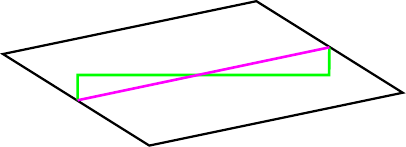}
			\caption{Case IV.}
		\end{subfigure}
				\\[.5cm]
		\begin{subfigure}[b]{0.4\textwidth}
			\centering
			\includegraphics[width=0.6\textwidth]{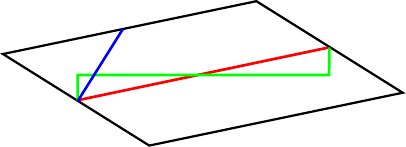}
			\caption{Case V.}
		\end{subfigure}
		\quad \quad \quad
		~ 
		\begin{subfigure}[b]{0.4\textwidth}
			\centering
			\includegraphics[width=0.6\textwidth]{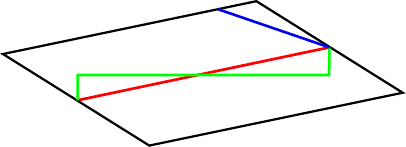}
			\caption{Case VI.}
		\end{subfigure}
		\caption{Bad segments.} 
		\label{fig:bad_seg}
	\end{figure}
	
	For good segments, the quantities we are trying to compare are equal. For bad segments the situation is much different. In cases I and II the quantities we are trying to compare differ by $1$. Cases III and IV correspond to non-transverse intersections between $A_t \alpha$ and $\beta$. We consider case V as bad to avoid double counting intersections shared by two consecutive rectangles. We consider case VI as bad to avoid missing intersections on the boundary of $a_t R$. We show there are few bad segments.
	
	Let $j \in \{1,\dots,m\}$ and $\beta_j$ be a saddle connection of $\beta$ or $\beta$ itself if $\beta$ is a cylinder curve. We assume that $\beta_j$ is not vertical. This simpler case can be dealt with using similar arguments. We bound the number of bad segments of $\beta_j$ in $A_t R$. Notice that bad segments always intersect a vertical segment of the rectangular decomposition of $A_t \smash{\alpha_*^\dagger}$ in $A_t R$. There are two such segments and each one has length $\leq  e^{-t} \cdot \ell_{\min}(q)/2$. It is enough then to bound the number of intersections of $\beta_j$ with any vertical segment $\sigma$ of $a_t q$ of length $\leq e^{-t} \cdot \ell_{\min}(q)/2 \leq \ell_{\min}(a_t q) / 2 \leq \mathrm{sys}(a_t q) / 2$. By Proposition \ref{eq:int_bound},
	\[
	\# (\beta_j \cap \sigma) \leq 1 + 2 \ell_{\beta_j}(a_tq)/\mathrm{sys}(a_tq) \preceq \ell_{\beta_j}(a_tq)/\ell_{\min}(a_tq). 
	\]
	
	Adding these contributions over all $j \in \{1,\dots,m\}$ and all embedded flat rectangles used to construct the rectangular decomposition of $\alpha$ on $q$, of which there are $\preceq \ell_{\alpha}(q)/\ell_{\min}(q)$ many, we obtain the desired estimate.
\end{proof}

\section{Immersed collars and bump functions of flat geodesics}

\subsection*{Outline of this section} In this section we describe a procedure for constructing immersed collars and bump functions of flat geodesics of quadratic differentials. The main result of this section is Proposition \ref{prop:bump_4}, which shows the bump functions we construct can be integrated to estimate the number of intersections between horizontal segments and flat geodesics of quadratic differentials. We first introduce Abelian differentials and weighted Sobolev spaces of functions on Riemann surfaces. We then proceed to construct immersed collars and bump functions of flat geodesics, first for cylinder curves and then for saddle connections and singular flat geodesics.

\subsection*{Abelian differentials.} Let $X$ be a closed Riemann surface and $K$ be its canonical bundle. An \textit{Abelian differential} $\omega$ on $X$ is a holomorphic section of $K$. In local coordinates, $\omega = f(z) \thinspace dz$ for some holomorphic function $f$. If $X$ has genus $g$, the number of zeroes of $\omega$ counted with multiplicity is $2g-2$. The zeroes of $\omega$ are also called \textit{singularities}. We sometimes denote Abelian differentials by $(X,\omega)$ to keep track of the Riemann surface they are defined on.

A \textit{translation structure} on a surface $S$ is an atlas of charts to $\mathbf{C}$ on the complement of a finite set of points $\Sigma \subseteq S$ whose transition functions are of the form $z \mapsto z + c$ with $c \in \mathbf{C}$. Every Abelian differential $\omega$ gives rise to a translation structure on the Riemann surface it is defined on by considering local coordinates on the complement of the zeroes of $\omega$ for which $\omega = dz$. Viceversa, every translation structure induces an Abelian differential on its underlying surface by pulling back the differential $dz$ on the corresponding charts. 

The notion of \textit{straight line} makes sense for a surface endowed with a translation structure and in particular for a closed Riemann surface endowed with an Abelian differential $\omega$. A \textit{cylinder curve} of $\omega$ is a closed straight line intersecting no zeroes. A \textit{saddle connection} of a $\omega$ is a straight line segment joining two zeroes and having no zeroes in its interior. The notions of \textit{slope} and \textit{parallelism} of straight line segments also make sense in this context.

Pulling back the standard Euclidean metric on $\mathbf{C}$ using the charts of a translation structure induces a singular flat metric on the underlying surface. In particular, every Abelian differential $\omega$ gives rise to a singular flat metric on the Riemann surface $X$ it is defined on. This metric is smooth away from the zeroes of $\omega$ and has a singularity of cone angle $(2k+2) \pi$ at every zero of order $k$. The diameter of $\omega$, denoted $\mathrm{diam}(\omega)$, is the diameter of $X$ with respect to this metric. Denote by $A_\omega$ the singular flat area form induced by $\omega$ on $X$. The area of $\omega$, denoted $\mathrm{Area}(\omega)$, is the area of $X$ with respect to $A_\omega$. Denote by $\ell_{\alpha}(\omega)$ the flat length of a saddle connection $\alpha$ of $\omega$ and by $\ell_{\min}(\omega)$ the flat length of the shortest saddle connections of $\omega$. Denote $\smash{\ell_{\min}^\dagger}(\omega) := \min\{1,\ell_{\min}(\omega)\}$.

Pulling back the $1$-forms $dx$ and $dy$ on $\mathbf{C}$ using the charts of a translation structure induces a pair of singular $1$-forms on the underlying surface. For the translation structure induced by an Abelian differential $\omega$ on a closed Riemann surface $X$ we denote these singular $1$-forms by $\mathrm{Re}(\omega)$ and $\mathrm{Im}(\omega)$. These 1-forms induce oriented singular measured foliations $\Re(\omega)$ and $\Im(\omega)$ on $X$. We refer to these as the \textit{vertical} and \textit{horizontal} foliations of $\omega$. Segments of the leaves of  $\Re(\omega)$ and $\Im(\omega)$ are called \textit{vertical} and \textit{horizontal}, respectively. Denote by $\partial x_\omega$ and $\partial y_\omega$ the pair of canonical continuous orthonormal vector fields induced by $\Im(\omega)$ and $\Re(\omega)$ on $X$ away from the zeroes of $\omega$. 

Every quadratic diffential $q$ on a closed Riemann surface $X$ induces a canonical double cover $h \colon Y \to X$ onto which $q$ pulls back to the square of an Abelian differential $\omega$. This cover is branched over the odd order zeroes of $q$. 

\subsection*{Weighted Sobolev spaces of functions.} Let $(X,\omega)$ be an Abelian differential. The space $L^2_\omega(X)$ is the space of measurable functions $\phi \colon X \to \mathbf{R}$ that are square integrable with respect to $A_\omega$. This is a Hilbert space when endowed with the norm
\[
\|\phi\|_{0,\omega} := \left(\int_S |\phi|^2 \thinspace dA_\omega \right)^{1/2}.
\]
The \textit{weighted Sobolev space of functions} $H_\omega^1(X) \subseteq L^2_\omega(X)$ is defined as
\begin{equation}
\label{eq:sob}
H_\omega^1(X) := \{\phi \in L^2_\omega(X) \ | \ \partial x_\omega \phi, \thinspace\partial y_\omega \phi \in L^2_\omega(X)\},
\end{equation}
where $\partial x_\omega \phi$ and $\thinspace \partial y_\omega \phi$ denote weak derivatives. This is a Hilbert space when endowed with the norm
\begin{equation}
\label{eq:sob_norm}
\|\phi\|_{1,\omega} := \|\phi\|_{0,\omega} + \|\partial x_\omega \phi\|_{0,\omega} + \|\partial y_\omega \phi\|_{0,\omega}.
\end{equation}

Let $(X,q)$ be a quadratic differential. Consider the canonical double cover $h \colon Y \to X$ onto which $q$ pulls back to the square of an Abelian differential $\omega$. Recall that $A_q$ denotes the singular flat area form induced by $q$ on $X$. The space $L^2_q(X)$ is the space of measurable functions $\phi \colon X \to \mathbf{R}$ that are square integrable with respect to $A_q$, or, equivalently, whose lift $\phi \circ h \colon Y \to \mathbf{R}$ belongs to $L_\omega^2(Y)$. This is a Hilbert space when endowed with the norm
\[
\|\phi\|_{0,q} := \|\phi \circ h\|_{0,\omega}.
\]
The \textit{weighted Sobolev space of functions} $H_q^1(X) \subseteq L^2_q(X)$ is defined as
\begin{equation}
\label{eq:sob_q}
H_q^1(X) := \{\phi \in L^2_q(X) \ | \ \phi \circ h\in H_\omega^1(Y)\}.
\end{equation}
This is a Hilbert space when endowed with the norm
\begin{equation}
\label{eq:sob_norm_q}
\|\phi\|_{1,q} := \|\phi \circ h\|_{1,\omega}.
\end{equation}

\subsection*{Immersed collars and bump functions of cylinder curves.} Let $q$ be a quadratic differential on a closed Riemann surface $X$ of genus $g \geq 2$ and $\beta$ be a non-horizontal cylinder curve of $q$. By an \textit{immersed collar} of $\beta$ we mean a piecewise linear immersion of a flat annulus into $X$ whose core curve is $\beta$. We construct an immersed collar of $\beta$ by horizontally shearing an immersed flat annulus around $\beta$ away from the zeroes of $q$ in a continuous piecewise linear way. See Figure \ref{fig:immersed_collar_cyl} for an example. 

\begin{figure}[h!]
	\centering
	\includegraphics[width=.15\textwidth]{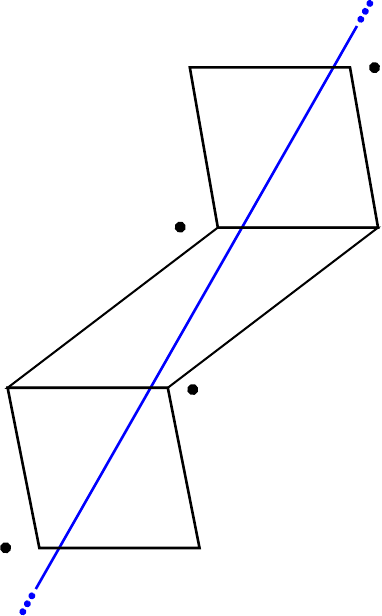}
	\caption{Constructing an immersed collar of a cylinder curve.} \label{fig:immersed_collar_cyl} 
\end{figure}

Fix an orientation on $\beta$ and consider a unit speed parametrization $\beta \colon [0,\ell_{\beta}(q)] \to X$ consistent with this orientation. For the rest of this discussion we identify the endpoints of the interval $[0,\ell_{\beta}(q)]$. Starting from any point of $\beta$ we can flow horizontally at unit speed for a maximal interval of times until we hit a zero of $q$. For every $t \in [0,\ell_{\beta}(q)]$ denote by $[s_{\min}(\beta(t)),s_{\max}(\beta(t))]$ this interval. For every $t \in [0,\ell_{\beta}(q)]$ and every $s \in [s_{\min}(\beta(t)),s_{\max}(\beta(t))]$ denote by $h_s(\beta(t)) \in X$ the point reached by flowing horizontally at unit speed from $\beta(t)$ for time $s$.

Notice that, for every $t \in [0,\ell_{\beta}(q)]$, it can never be the case that $s_{\min}(\beta(t)) \geq -\ell_{\min}(q)/8$ and $s_{\max}(\beta(t)) \leq \ell_{\min}(q)/8$, else $q$ would have a saddle connection of length $\leq \ell_{\min}(q)/4$. A similar argument shows that if $t_0,t_1 \in [0,\ell_{\beta}(q)]$ are two times such that $s_{\min}(\beta(t_i)) \geq -\ell_{\min}(q)/8$ or $s_{\max}(\beta(t_i)) \leq \ell_{\min}(q)/8$ then $|t_1  - t_0| \geq \ell_{\min}(q)/4$. In particular, there are $\preceq \ell_{\beta}(q)/\ell_{\min}(q)$ times $t \in [0,\ell_{\beta}(q)]$ such that $s_{\min}(\beta(t)) \geq -\ell_{\min}(q)/8$ or $s_{\max}(\beta(t)) \leq \ell_{\min}(q)/8$. For each one of these times define $\Delta(t) \in (-\ell_{\min}(q)/8,\ell_{\min}(q)/8)$ as
\[
\Delta(t) := \left \lbrace
\begin{array}{c l}
s_{\min}(\beta(t))/2 + \ell_{\min}(q)/8 & \text{if } s_{\min}(\beta(t)) \geq - \ell_{\min}(q)/8,\\
-\ell_{\min}(q)/8 + s_{\max}(\beta(t))/2 & \text{if } s_{\max}(\beta(t))\leq \ell_{\min}(q)/8.
\end{array} \right. 
\]
Denote by $\Delta \colon [0, \ell_{\beta}(q)] \to (-\ell_{\min}(q)/8,\ell_{\min}(q)/8)$ the linear interpolation of these values. 

Consider the immersed collar $\iota \colon [0, \ell_{\beta}(q)]  \times [-\ell_{\min}(q)/8,\ell_{\min}(q)/8] \to X$ of $\beta$ given by
\[
\iota(t,s) := h_{s+\Delta(t)}(\beta(t)).
\]
By construction, this map is an immersion whose image contains no zeroes of $q$. The following result, which can be proved using arguments similar to those in the proof of Proposition \ref{prop:embed}, quantifies the extent to which this map is an embedding.

\begin{proposition}
	\label{prop:collar_embed}
	Let $t_0,t_1 \in [0,\ell_{\beta}(q)]$ with $0 < |t_1 - t_0| \leq \ell_{\min}(q)/4$. Then, the restriction of the map $\iota \colon [0, \ell_{\beta}(q)]  \times [-\ell_{\min}(q)/8,\ell_{\min}(q)/8] \to X$ to $[t_0,t_1] \times [-\ell_{\min}(q)/8,\ell_{\min}(q)/8]$ is an embedding. In particular, the map $\iota$ is $n$ to $1$ for $n \preceq \ell_{\beta}(q)/\ell_{\min}(q)$.
\end{proposition}

The quintessential feature of the immersed collar $\iota$ is that it supports a sufficiently regular bump function that can be integrated to estimate the number of intersections between horizontal segments of $q$ and $\beta$. We construct this bump function in the following way. Fix a smooth symmetric function $\varphi \colon [-1,1] \to [0,1]$ with $\supp(\varphi) \subseteq (-1,1)$ and such that
\[
\int_{-1}^{1} \varphi(x) \thinspace dx = 1.
\]
For every $\epsilon > 0$ let $\varphi_\epsilon \colon [-\epsilon,\epsilon] \to [0,1/\epsilon]$ be the function given by $\varphi_\epsilon(x) := \varphi(x/\epsilon)/\epsilon$. Consider the function $\phi^* \colon [0,\ell_{\beta}(q)] \times [-\ell_{\min}(q)/8,\ell_{\min}(q)/8] \to [0,8/\ell_{\min}(q)]$ given by
\[
\phi^*(t,s) := \varphi_{\ell_{\min}(q)/8}(s).
\]
The immersed collar $\iota$ supports the non-negative bump function $\phi \colon X \to \mathbf{R}$ given by
\[
\phi(x) := \sum_{(t,s) \in \iota^{-1}(x)} \phi^*(t,s),
\]
where we interpret empty sums as having value $0$. 

By Proposition \ref{prop:collar_embed}, this function is a well defined, continuous, piecewise smooth, and in particular belongs to the weighted Sobolev space $H^1_q(X)$. Recall that $\Im(q)$ denotes the horizontal foliaton of $q$. Denote by $i(\beta,\Im(q))$ the total transverse measure of $\beta$ with respect to $\Im(q)$. By construction,
\[
\int_X \phi \thinspace dA_q = i(\beta,\Im(q)).
\]

Recall that $\ell_{\min}^\dagger(q) := \min\{1,\ell_{\min}(q)\}$. Denote $v_\beta(q) := i(\beta,\Im(q))/\ell_\beta(q)$. Let $\|\varphi\|_{\mathcal{C}^1}$ be the $\mathcal{C}^1$ norm of $\varphi$. The following result quantifies the regularity of the bump function $\phi$.

\begin{proposition}
	\label{prop:regularity_1}
	The function $\phi \in H^1_q(X)$ satisfies
	\[
	\|\phi\|_{1,q} \preceq \mathrm{Area}(q)^{1/2} \cdot \ell_{\beta}(q) \cdot \smash{\ell_{\min}^\dagger}(q)^{-3} \cdot v_\beta(q)^{-1} \cdot \| \varphi\|_{\mathcal{C}^1}.
	\]
\end{proposition}

\begin{proof}
	We begin by collecting a couple of facts about the functions used to construct $\phi$. We denote $\sup$ norms by $\|\cdot \|_{\infty}$. Directly from the definition of $\Delta$ one can show that
	\begin{equation}
	\label{eq:delta_bd}
	\|\partial_t \Delta \|_\infty \preceq 1.
	\end{equation}
	Directly from the definition of $\phi^*$ one can show that
	\begin{equation}
	\label{eq:psi_bound}
	\|\phi^*\|_{\infty} \preceq \ell_{\min}(q)^{-1} \cdot \|\varphi\|_{\mathcal{C}^1}, \quad \partial_t \phi^* \equiv 0, \quad \|\partial_s \phi^*\|_\infty \preceq \ell_{\min}(q)^{-2} \cdot \|\varphi\|_{\mathcal{C}^1}.
	\end{equation}
	
	Let $\partial x_q$ and $\partial y_q$ be a choice of measurable orthonormal vector fields on $X$ defined away from the zeroes of $q$ and tangent to the singular foliations $\Im(q)$ and $\Re(q)$, respectively. To prove the desired estimate, we first bound the $\sup$ norms $\|\phi\|_{\infty}$, $\|\partial x_q \phi\|_{\infty}$, and $\|\partial y_q \phi \|_\infty$. 
	
	Directly from the definition of $\phi$, Proposition \ref{prop:collar_embed}, and (\ref{eq:psi_bound}), one can check that
	\begin{equation}
	\label{eq:bd_1}
	\|\phi\|_{\infty} \preceq \ell_{\beta}(q) \cdot \ell_{\min}(q)^{-2} \cdot \|\varphi\|_{\mathcal{C}^1}.
	\end{equation}
	Notice that if $\Delta \equiv 0$ then $\iota$ identifies the $\partial t$ vectors on its domain with unit vectors parallel to $\beta$ on its image. Denote these vectors by $\partial \beta$. More generally, if $\Delta \not\equiv 0$, (\ref{eq:delta_bd}) ensures $\iota$ identifies the vectors $\partial \beta$ on its image with vectors on its domain of the form
	\[
	\partial t + a \cdot \partial s, \quad |a| \preceq 1.
	\]
	Notice that, up to sign, $\iota$ identifies the $\partial s$ vectors on its domain with the vectors $\partial x_q$ on its image. From these facts, Proposition \ref{prop:collar_embed}, and (\ref{eq:psi_bound}),  it follows that
	\begin{equation}
	\label{eq:bd_2}
	\|\partial \beta \phi \|_{\infty} \preceq \ell_{\beta}(q) \cdot \ell_{\min}(q)^{-3} \cdot \|\varphi\|_{\mathcal{C}^1}, \quad \|\partial x_q \phi \|_{\infty} \preceq \ell_{\beta}(q) \cdot \ell_{\min}(q)^{-3} \cdot \|\varphi\|_{\mathcal{C}^1}.
	\end{equation}
	Let $m > 0$ be the absolute value of the slope of $\beta$. A direct computation shows that
	\[
	\partial y_q = \pm(1 + 1/m^2)^{1/2} \thinspace \partial \beta \pm (1/m) \thinspace \partial x_q.
	\]
	Using this equality, the bound $1/m \leq (1+1/m^2)^{1/2}$, and (\ref{eq:bd_2}), we deduce
	\begin{equation}
	\label{eq:bd_3}
	\| \partial y_q \phi \|_{\infty} \preceq (1+ 1/m^2)^{1/2} \cdot \ell_{\beta}(q) \cdot \ell_{\min}(q)^{-3} \cdot \|\varphi\|_{\mathcal{C}^1}.
	\end{equation}
	
	By general measure theory considerations,
	\[
	\|\phi\|_{1,q} \preceq \mathrm{Area}(q)^{1/2} \cdot (\|\phi\|_\infty + \|\partial x_q \phi\|_\infty + \|\partial y_q \phi\|_{\infty}).
	\]
	This together with (\ref{eq:bd_1}), (\ref{eq:bd_2}), and (\ref{eq:bd_3}) implies
	\begin{equation}
	\label{eq:bd_4}
	\|\phi\|_{1,q} \preceq \mathrm{Area}(q)^{1/2} \cdot (1+ 1/m^2)^{1/2} \cdot \ell_{\beta}(q) \cdot \smash{\ell_{\min}^\dagger}(q)^{-3} \cdot \|\varphi\|_{\mathcal{C}^1}.
	\end{equation}
	A direct computation shows that
	\[
	(1+1/m^2)^{1/2} = \ell_{\beta}(q)/i(\beta,\Im(q))= v_\beta(q)^{-1}.
	\]
	From this and (\ref{eq:bd_4}) we conclude
	\[
	\|\phi\|_{1,q} \preceq \mathrm{Area}(q)^{1/2}  \cdot \ell_{\beta}(q) \cdot \smash{\ell_{\min}^\dagger}(q)^{-3} \cdot v_\beta(q) \cdot \|\varphi\|_{\mathcal{C}^1}. \qedhere
	\]
\end{proof}

Recall that $\Re(q)$ denotes the vertical foliation of $q$ endowed with its natural transverse measure. The following proposition shows that the function $\phi$ can be integrated to estimate the number of intersections between horizontal segments of $q$ and $\beta$.

\begin{proposition}
	\label{prop:bump_int_1}
	Let $\gamma$ be a horizontal segment of $q$. Then,
	\[
	\bigg\vert \# (\gamma \cap \beta) -\int_\gamma \phi \thinspace d\Re(q) \bigg\vert \preceq \frac{\ell_{\beta}(q)}{\ell_{\min}(q)}.
	\]
\end{proposition}

\begin{proof}
	Directly from the construction of $\phi$ we see that the quantities being compared coincide except when one of the following \textit{bad situations} happens: $\gamma$ intersects $\beta$ but $\gamma$ does not completely cross the immersed collar $\iota$, or $\gamma$ intersects the immersed collar $\iota$ but $\gamma$ does not intersect $\beta$. See Figure \ref{fig:immersed_collar_int} for examples. These bad situations can happen multiple times. See Figure \ref{fig:immersed_collar_int_mult} for an example. 
	
		\begin{figure}[h]
		\centering
		\begin{subfigure}[b]{0.4\textwidth}
			\centering
			\includegraphics[width=0.5\textwidth]{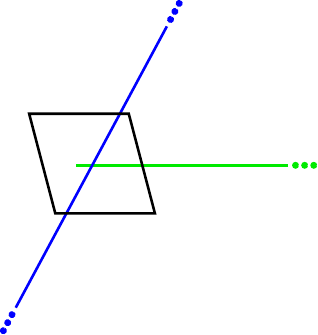}
			\caption{Case I.}
		\end{subfigure}
		\quad \quad \quad
		~ 
		\begin{subfigure}[b]{0.4\textwidth}
			\centering
			\includegraphics[width=0.5\textwidth]{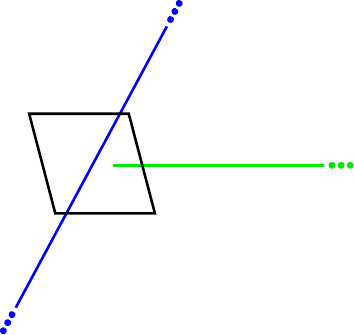}
			\caption{Case II.}
		\end{subfigure}
		\caption{Bad situations.} 
		\label{fig:immersed_collar_int}
	\end{figure}
	
	\begin{figure}[h!]
		\centering
		\includegraphics[width=.2\textwidth]{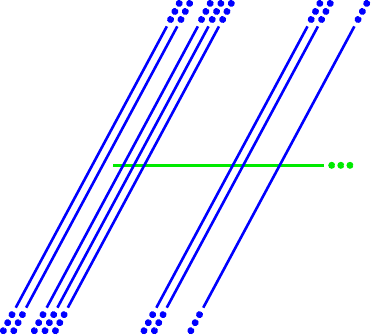}
		\caption{Multiple bad situations.} \label{fig:immersed_collar_int_mult} 
	\end{figure}
	
	Notice that, by construction, the support of $\phi$ is contained in horizontal segments of length $\ell_{\min}(q)/4$ across $\beta$. Thus, the number of times the bad situations can happen is bounded by twice, once per each endpoint of $\gamma$, the number of times $\beta$ can intersect a horizontal segment of length $\ell_{\min}(q)/2 \leq \mathrm{sys}(q)/2$. By Proposition \ref{eq:int_bound}, this quantity is bounded above by $1 + 2 \ell_{\beta}(q)/\mathrm{sys}(q)\preceq \ell_{\beta}(q)/\ell_{\min}(q)$.
\end{proof}

We summarize the main properties of the constructions above in the following proposition. This statement is tailored to applications in later sections. In particular, we restrict to unit area quadratic differentials and consider $\|\varphi\|_{\mathcal{C}^1}$ as a universal constant.

\begin{proposition}
	\label{eq:bump_1}
	Let $q$ be a unit area quadratic differential on a closed Riemann surface $X$ of genus $g \geq 2$ and $\beta$ be a cylinder curve of $q$. Then, there exists a non-negative, continuous, piecewise smooth function $\phi \in H_q^1(X)$ satisfying
	\begin{enumerate}
		\item $\displaystyle \int_X \phi \thinspace dA_q = i(\beta,\Im(q))$,
		\item $\|\phi\|_{1,q} \preceq \ell_{\beta}(q) \cdot \smash{\ell_{\min}^\dagger}(q)^{-3} \cdot v_\beta(q)^{-1}$,
	\end{enumerate}
	and such that for every horizontal segment $\gamma$ of $q$,
	\[
	\bigg\vert \# (\gamma \cap \beta) -\int_\gamma \phi \thinspace d\Re(q) \bigg\vert \preceq \frac{\ell_{\beta}(q)}{\ell_{\min}(q)}.
	\]
\end{proposition}

\subsection*{Immersed collars and bump functions of saddle connections.} Let $q$ be a quadratic differential on a closed Riemann surface $X$ of genus $g \geq 2$ and $\beta$ be a non-horizontal saddle connection of $q$. We construct an immersed collar of $\beta$ by following a procedure similar to the one introduced above. See Figure \ref{fig:immersed_collar_saddle} for an example. To simplify the notation, we assume $q$ only has simple zeroes. All the constructions that follow can be adapted to the general case via minor modifications.

\begin{figure}[h!]
	\centering
	\includegraphics[width=.35\textwidth]{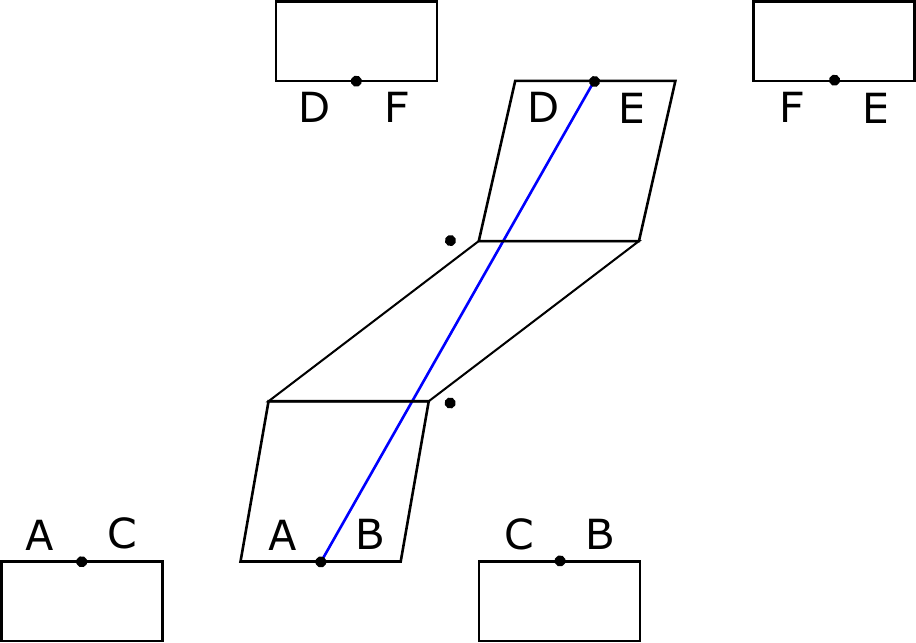}
	\caption{Constructing an immersed collar of a saddle connection.} \label{fig:immersed_collar_saddle} 
\end{figure}

Fix an orientation on $\beta$ and consider a unit speed parametrization $\beta \colon [0,\ell_{\beta}(q)] \to X$ consistent with this orientation. Starting from any point in the interior of $\beta$ we can flow horizontally at unit speed for a maximal interval of times until we hit a zero of $q$. Starting from any of the two endpoints of $\beta$ we can flow horizontally at unit speed along the two singular leaves of $\Im(q)$ closest to $\beta$ for a maximal interval of times until we hit a zero of $q$. For every $t \in [0,\ell_{\beta}(q)]$ denote by $[s_{\min}(\beta(t)),s_{\max}(\beta(t))]$ the corresponding interval. For every $t \in [0,\ell_{\beta}(q)]$ and every $s \in [s_{\min}(\beta(t)),s_{\max}(\beta(t))]$ denote by $h_s(\beta(t)) \in X$ the point reached by flowing horizontally at unit speed from $\beta(t)$ for time $s$.

Notice that, for every $t \in [0,\ell_{\beta}(q)]$, it can never be the case that $s_{\min}(\beta(t)) \geq -\ell_{\min}(q)/8$ and $s_{\max}(\beta(t)) \leq \ell_{\min}(q)/8$, else $q$ would have a saddle connection of length $\leq \ell_{\min}(q)/4$. A similar argument shows that if $t_0,t_1 \in [0,\ell_{\beta}(q)]$ are two times such that $s_{\min}(\beta(t_i)) \geq -\ell_{\min}(q)/8$ or $s_{\max}(\beta(t_i)) \leq \ell_{\min}(q)/8$ then $|t_1  - t_0| \geq \ell_{\min}(q)/4$. In particular, there are $\preceq \ell_{\beta}(q)/\ell_{\min}(q)$ times $t \in [0,\ell_{\beta}(q)]$ such that $s_{\min}(\beta(t)) \geq -\ell_{\min}(q)/8$ or $s_{\max}(\beta(t)) \leq \ell_{\min}(q)/8$. For each of ones of these times define $\Delta(t) \in (-\ell_{\min}(q)/8,\ell_{\min}(q)/8)$ as
\[
\Delta(t) := \left \lbrace
\begin{array}{c l}
s_{\min}(\beta(t))/2 + \ell_{\min}(q)/8 & \text{if } s_{\min}(\beta(t)) \geq - \ell_{\min}(q)/8,\\
-\ell_{\min}(q)/8 + s_{\max}(\beta(t))/2 & \text{if } s_{\max}(\beta(t))\leq \ell_{\min}(q)/8.
\end{array} \right. 
\]
Fix $\Delta(0) := 0$ and $\Delta(\ell_{\beta}(q)) := 0$. Denote by $\Delta \colon [0, \ell_{\beta}(q)] \to (-\ell_{\min}(q)/8,\ell_{\min}(q)/8)$ the linear interpolation of these values. 

We construct a singular rectangle with two cone points as in Figure \ref{fig:sing_rect}. Consider the rectangles
\begin{gather*}
R := [0,\ell_{\beta}(q)] \times [-\ell_{\min}(q)/8,\ell_{\min}(q)/8],\\
R_{0,0}:= [-\ell_{\min}(q)/8,0] \times [-\ell_{\min}(q)/8,\ell_{\min}(q)/8], \\
R_{0,1}:= [-\ell_{\min}(q)/8,0] \times [-\ell_{\min}(q)/8,\ell_{\min}(q)/8], \\
R_{1,0} := [0,\ell_{\min}(q)/8] \times [-\ell_{\min}(q)/8,\ell_{\min}(q)/8], \\
R_{1,1} := [0,\ell_{\min}(q)/8] \times [-\ell_{\min}(q)/8,\ell_{\min}(q)/8]. 
\end{gather*}
On these rectangles consider the equivalence relation $\sim$ generated by the identifications
\begin{gather*}
(0,s) \in R \sim (0,s) \in R_{0,0} \text{ if } s \leq 0, \\
(0,s) \in R \sim (0,s) \in R_{0,1} \text{ if } s \geq 0, \\
(0,s) \in R_{0,0} \sim (0,-s) \in R_{0,1} \text{ if } s \geq 0, \\
(\ell_{\beta}(q),s) \in R \sim (0,s) \in R_{1,0} \text{ if } s \leq 0, \\
(\ell_{\beta}(q),s) \in R \sim (0,s) \in R_{1,1} \text{ if } s \geq 0, \\
(0,s) \in R_{0,0} \sim (0,-s) \in R_{1,1} \text{ if } s \geq 0.
\end{gather*}
Consider the singular rectangle with two cone points
\[
R_* := R \sqcup R_{0,0} \sqcup R_{0,1} \sqcup R_{1,0} \sqcup R_{1,1} \thinspace/ \sim.
\]

\begin{figure}[h!]
	\centering
	\includegraphics[width=.25\textwidth]{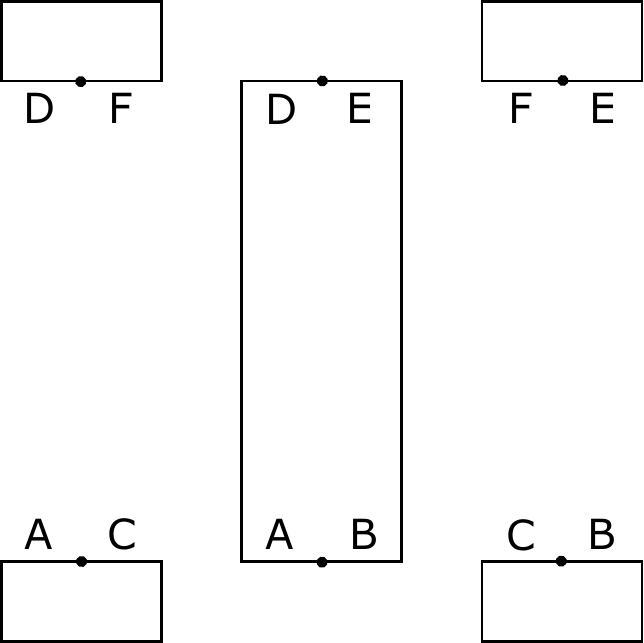}
	\caption{Singular rectangle with two cone points.} \label{fig:sing_rect} 
\end{figure}

Consider the three singular leaves of $\Re(q)$ at $\beta(0)$. Among the two singular leaves that are farthest from $\beta$, denote by $\sigma_{0,0}$ the singular leaf that is closest counterclockwise from $\beta$ and by $\sigma_{0,1}$ the singular leaf that is closest clockwise from $\beta$. Consider unit speed parametrizations $\sigma_{0,j} \colon [0,\ell_{\min}(q)/8] \to X$ of the initial segments of these leaves. For every $t \in [0,\ell_{\min}(q)/4]$ and every $s \in [-\ell_{\min}(q)/8,\ell_{\min}(q)/8]$ denote by $h_s(\sigma_{0,j}(t)) \in X$ the point reached by flowing horizontally at unit speed from $\sigma_{0,j}(t)$ for time $s$; when $t = 0$ we flow along the singular leaves of $\Im(q)$ closest to $\sigma_{0,j}$. These flows are well defined. Indeed, if this was not the case, then $q$ would have a saddle connection of length $\leq \ell_{\min}(q)/4$.

Analogously, consider the three singular leaves of $\Re(q)$ at $\beta(\ell_{\beta}(q))$. Among the two singular leaves that are farthest from $\beta$, denote by $\sigma_{1,0}$ the singular leaf that is closest clockwise from $\beta$ and by $\sigma_{1,1}$ the singular leaf that is closest counterclockwise from $\beta$. Consider unit speed parametrizations $\sigma_{1,j} \colon [0,\ell_{\min}(q)/8] \to X$ of the initial segments of these leaves. For every $t \in [0,\ell_{\min}(q)/4]$ and every $s \in [-\ell_{\min}(q)/8,\ell_{\min}(q)/8]$ denote by $h_s(\sigma_{1,j}(t)) \in X$ the point reached by flowing horizontally at unit speed from $\sigma_{1,j}(t)$ for time $s$; when $t = 0$ we flow along the singular leaves of $\Im(q)$ closest to $\sigma_{1,j}$. The same argument used above shows these flows are well defined.

Consider the immersed collar $\iota \colon R_* \to X$ of $\beta$ given by
\[
\iota(t,s) := \left\lbrace \begin{array}{c l}
h_{s+\Delta(t)}(\beta(t)) & \text{ if $(t,s) \in R$}, \\
h_s(\sigma_{0,j}(-t))&  \text{ if $(t,s) \in R_{0,j}$ for $j \in \{0,1\}$},\\
h_s(\sigma_{1,j}(t))& \text{ if $(t,s) \in R_{1,j}$ for $j \in \{0,1\}$}.
\end{array}  \right.
\]
By construction, this map is an immersion whose image contains exactly the zeroes of $q$ corresponding to the endpoints of $\beta$. The following result, which can be proved using arguments similar to those in the proof of Proposition \ref{prop:embed}, quantifies the extent to which this map is an embedding. 

\begin{proposition}
	\label{prop:collar_embed_2}
	The restriction of the map $\iota \colon R_* \to X$ to $R_{i,j}$  is an embedding for every $i,j \in \{0,1\}$. In addition, if $t_0,t_1 \in [0,\ell_{\beta}(q)]$ are such that $0 < |t_1 - t_0| \leq \ell_{\min}(q)/4$, then the restriction of $\iota$ to $[t_0,t_1] \times [-\ell_{\min}(q)/8,\ell_{\min}(q)/8] \subseteq R$ is an embedding. In particular, $\iota$ is $n$ to $1$ for $n \preceq \ell_{\beta}(q)/\ell_{\min}(q)$.
\end{proposition}

The quintessential feature of the immersed collar $\iota$ is that it supports a one parameter family of pairs of sufficiently regular bump functions that can be integrated to estimate the number of intersections between horizontal segments of $q$ and $\beta$. We construct these families in the following way. Fix a smooth symmetric function $\varphi \colon [-1,1] \to [0,1]$ with $\supp(\varphi) \subseteq (-1,1)$, constant in a neighborhood of $0$, and such that
\[
\int_{-1}^{1} \varphi(x) \thinspace dx = 1.
\]
For every $\epsilon > 0$ let $\varphi_\epsilon \colon [-\epsilon,\epsilon] \to [0,1/\epsilon]$ be the function given by $\varphi_\epsilon(s) := \varphi(s/\epsilon)/\epsilon$. Fix a smooth function $\psi \colon [0,1] \to [0,1]$ with $\supp(\psi) \subseteq [0,1)$ and $\supp(1-\psi) \subseteq (0,1]$. For every $\delta > 0$ let $\psi_\delta  \colon [0,\delta] \to [0,1]$ be the function given by $\psi_\delta(t) := \psi(t/\delta)$. For every $0 < \delta < \ell_{\min}(q)/8$ consider the function $\phi_{0,\delta}^* \colon R_* \to [0,8/\ell_{\min}(q)]$ given by
\[
\phi_{0,\delta}^*(t,s) := \left\lbrace \begin{array}{c l}
\varphi_{\ell_{\min}(q)/8}(s) & \text{ if $(t,s) \in R$ and $t \in [\delta, \ell_{\beta}(q)- \delta]$}, \\
(1-\psi_\delta(t)) \cdot \varphi_{\ell_{\min}(q)/8}(s) & \text{ if $(t,s) \in R$ and $t \in [0,\delta]$}, \\
(1-\psi_\delta(\ell_{\beta}(q)-t)) \cdot \varphi_{\ell_{\min}(q)/8}(s) & \text{ if $(t,s) \in R$ and $t \in [\ell_\beta(q) - \delta, \ell_{\beta}(q)]$}, \\
0 & \text{ if $(t,s) \in R_{i,j}$ for $i,j \in \{0,1\}$}.
\end{array} \right.
\]
For every $0 < \delta < \ell_{\min}(q)/8$ consider the function $\phi_{1,\delta}^* \colon R_* \to [0,8/\ell_{\min}(q)]$ given by
\[
\phi_{1,\delta}^*(t,s) := \left\lbrace \begin{array}{c l}
\varphi_{\ell_{\min}(q)/8}(s) & \text{ if $(t,s) \in R$}, \\
\psi_\delta(-t) \cdot \varphi_{\ell_{\min}(q)/8}(s) & \text{ if $(t,s) \in R_{0,j}$ for $j \in \{0,1\}$ and $t \in [-\delta,0]$}, \\
\psi_\delta(t) \cdot \varphi_{\ell_{\min}(q)/8}(s) & \text{ if $(t,s) \in R_{1,j}$ for $j \in \{0,1\}$ and $t \in [0,\delta]$}, \\
0 & \text{ if $(t,s) \in R_{i,j}$ for $i,j \in \{0,1\}$ and $|t| \geq \delta$}.
\end{array} \right.
\]
For every $i \in \{0,1\}$ and every $0<\delta < \ell_{\min}(q)/8$, the immersed collar $\iota$ supports the non-negative bump function $\phi_{i,\delta} \colon X \to \mathbf{R}$ given by
\[
\phi_{i,\delta}(x) := \sum_{(t,s) \in \iota^{-1}(x)} \phi_{i,\delta}^*(t,s),
\]
where we interpret empty sums as having value $0$. 

By Proposition \ref{prop:collar_embed_2}, these functions are well defined, continuous, piecewise smooth, and in particular belong to the weighted Sobolev space $H^1_q(X)$. Denote by $i(\beta,\Im(q))$ the total transverse measure of $\beta$ with respect to $\Im(q)$.  By construction, 
\begin{enumerate}
	\item $0 \leq \phi_{0,\delta} \leq \phi_{1,\delta}$,
	\item $\displaystyle \int_X \phi_{0,\delta} \thinspace dA_q \leq i(\beta,\Im(q)) \leq \int_X \phi_{1,\delta} \thinspace dA_q$,
	\item $\displaystyle \int_X (\phi_{1,\delta} -\phi_{0,\delta}) \thinspace dA_q \preceq \delta \cdot \ell_{\min}(q)$.
\end{enumerate}

Recall that $\ell_{\min}^\dagger(q) := \min\{1,\ell_{\min}(q)\}$. Denote $v_\beta(q) := i(\beta,\Im(q))/\ell_{\beta}(q)$. Let $\|\varphi\|_{\mathcal{C}^1}$ and $\|\psi\|_{\mathcal{C}^1}$ be the $\mathcal{C}^1$ norms of $\varphi$ and $\psi$. The following result, which can be proved using arguments similar to those in the proof of Proposition \ref{prop:regularity_1}, quantifies the regularity of the bump functions $\phi_{i,\delta}$.

\begin{proposition}
	\label{prop:regularity_2}
	For every $i \in \{0,1\}$ and every $0 < \delta < \smash{\ell_{\min}^\dagger}(q)/8$,
	\[
	\|\phi_{i,\delta}\|_{1,q} \preceq \mathrm{Area}(q)^{1/2} \cdot \ell_{\beta}(q) \cdot\smash{\ell_{\min}^\dagger}(q)^{-3} \cdot v_\beta(q)^{-1} \cdot \delta^{-1} \cdot \| \varphi\|_{\mathcal{C}^1} \cdot \| \psi\|_{\mathcal{C}^1}.
	\]
\end{proposition}

The following result, which can be proved using arguments similar to those in the proof of Proposition \ref{prop:bump_int_1}, shows that the functions $\phi_{i,\delta}$ can be integrated to estimate the number of intersections between horizontal segments of $q$ and $\beta$.

\begin{proposition}
	\label{prop:bump_int_2}
	Let $\gamma$ be a horizontal segment of $q$. Then, for every $0 < \delta < \ell_{\min}(q)/8$,
	\begin{gather*}
	\int_{\gamma} \phi_{0,\delta} \thinspace d\Re(q) - \#(\gamma \cap \beta)\preceq \frac{\ell_{\beta}(q)}{\ell_{\min}(q)}, \\
	\#(\gamma \cap \beta) - \int_{\gamma} \phi_{1,\delta} \thinspace d\Re(q) \preceq \frac{\ell_{\beta}(q)}{\ell_{\min}(q)}.
	\end{gather*}
\end{proposition}

We summarize the main properties of the constructions above in the following proposition. This statement is tailored to applications in later sections. In particular, we restrict to unit area quadratic differentials and consider $\|\varphi\|_{\mathcal{C}^1}$ and $\|\psi\|_{\mathcal{C}^1}$ as universal constants.

\begin{proposition}
	\label{eq:bump_2}
	Let $q$ be a unit area quadratic differential on a closed Riemann surface $X$ of genus $g \geq 2$ and $\beta$ be a non-horizontal saddle connection of $q$. Then, for every $0 < \delta < \smash{\ell_{\min}^\dagger}(q)/8$, there exists a pair of non-negative, continuous, piecewise smooth functions $\phi_{0,\delta}, \thinspace \phi_{1,\delta} \in H_q^1(X)$ satisfying
	\begin{enumerate}
		\item $0 \leq \phi_{0,\delta} \leq \phi_{1,\delta}$,
		\item $\displaystyle \int_X \phi_{0,\delta} \thinspace dA_q \leq i(\beta,\Im(q)) \leq \int_X \phi_{1,\delta} \thinspace dA_q$,
		\item $ \displaystyle \int_X (\phi_{1,\delta} -\phi_{0,\delta}) \thinspace dA_q \preceq \delta \cdot \ell_{\min}(q)$,
		\item $\|\phi_{i,\delta}\|_{1,q} \preceq \ell_{\beta}(q) \cdot \smash{\ell_{\min}^\dagger}(q)^{-3} \cdot v_\beta(q)^{-1} \cdot \delta^{-1}$ for $i \in \{0,1\}$,
	\end{enumerate}
	and such that for every horizontal segment $\gamma$ of $q$,
	\begin{gather*}
	\int_{\gamma} \phi_{0,\delta} \thinspace d\Re(q) - \#(\gamma \cap \beta)\preceq \frac{\ell_{\beta}(q)}{\ell_{\min}(q)}, \\
	\#(\gamma \cap \beta) - \int_{\gamma} \phi_{1,\delta} \thinspace d\Re(q) \preceq \frac{\ell_{\beta}(q)}{\ell_{\min}(q)}.
	\end{gather*}
\end{proposition}

\subsection*{Bump functions of singular flat geodesics.} Let $q$ be a quadratic differential on a closed Riemann surface $X$ of genus $g \geq 2$ and $\beta$ be a singular flat geodesic of $q$ with no horizontal saddle connections. Denote by $\smash{\{\beta_j\}_{j=1}^m}$ be the sequence of saddle connections of $\beta$. Let $v_\beta(q) := \min_{i=1,\dots,m} v_{\beta_j}(q)$. Adding up the bump functions provided by Proposition \ref{eq:bump_2} for the saddle connections $\smash{\{\beta_j\}_{j=1}^m}$ yields a one parameter family of pairs of sufficiently regular bump functions that can be integrated to estimate the number of intersections between horizontal segments of $q$ and $\beta$.

\begin{proposition}
	\label{eq:bump_3}
	Let $q$ be a unit area quadratic differential on a closed Riemann surface $X$ of genus $g \geq 2$ and $\beta$ be a singular flat geodesic of $q$ with no horizontal saddle connections. Denote by $\{\beta_j\}_{j=1}^m$ the sequence of saddle connections of $\beta$. Then, for every $0 < \delta < \smash{\ell_{\min}^\dagger}(q)/8$, there exists a pair of non-negative, continuous, piecewise smooth functions $\phi_{0,\delta}, \thinspace \phi_{1,\delta} \in H_q^1(X)$ satisfying
	\begin{enumerate}
		\item $0 \leq \phi_{0,\delta} \leq \phi_{1,\delta}$,
		\item $\displaystyle \int_X \phi_{0,\delta} \thinspace dA_q \leq i(\beta,\Im(q)) \leq \int_X \phi_{1,\delta} \thinspace dA_q$,
		\item $ \displaystyle \int_X (\phi_{1,\delta} -\phi_{0,\delta}) \thinspace dA_q \preceq \delta \cdot \ell_{\beta}(q)$,
		\item $\|\phi_{i,\delta}\|_{1,q} \preceq \ell_{\beta}(q) \cdot\smash{\ell_{\min}^\dagger}(q)^{-3} \cdot v_\beta(q)^{-1} \cdot \delta^{-1}$ for $i \in \{0,1\}$,
	\end{enumerate}
	and such that for every horizontal segment $\gamma$ of $q$,
	\begin{gather*}
	\int_{\gamma} \phi_{0,\delta} \thinspace d\Re(q) - \sum_{j=1}^m \#(\gamma \cap \beta_j)\preceq \frac{\ell_{\beta}(q)}{\ell_{\min}(q)}, \\
	\sum_{j=1}^m \#(\gamma \cap \beta_j) - \int_{\gamma} \phi_{1,\delta} \thinspace d\Re(q) \preceq \frac{\ell_{\beta}(q)}{\ell_{\min}(q)}.
	\end{gather*}
\end{proposition}

\subsection*{Bump functions of flat geodesics.} Let $q$ be a quadratic differential on a closed Riemann surface $X$ of genus $g \geq 2$ and $\beta$ be a flat geodesic of $q$ with no horizontal segments. Denote by $\{\beta_j\}_{j=1}^m$ the sequence of saddle connections of $\beta$ or the singleton $\{\beta\}$ if $\beta$ is a cylinder curve. Let $v_\beta(q) := \min_{i=j,\dots,m} v_{\beta_j}(q)$. We end this section by summarizing Propositions \ref{eq:bump_1} and \ref{eq:bump_3} into one statement.

\begin{proposition}
	\label{prop:bump_4}
	Let $q$ be a unit area quadratic differential on a closed Riemann surface $X$ of genus $g \geq 2$ and $\beta$ be a flat geodesic of $q$ with no horizontal segments. Denote by $\{\beta_j\}_{j=1}^m$ the sequence of saddle connections of $\beta$ or the singleton $\{\beta\}$  if $\beta$ is a cylinder curve. Then, for every $0 < \delta < \smash{\ell_{\min}^\dagger(q)}/8$, there exist non-negative, continuous, piecewise smooth functions $\phi_{0,\delta}, \thinspace \phi_{1,\delta} \in H_q^1(X)$ satisfying
	\begin{enumerate}
		\item $0 \leq \phi_{0,\delta} \leq \phi_{1,\delta}$,
		\item $\displaystyle \int_X \phi_{0,\delta} \thinspace dA_q \leq i(\beta,\Im(q)) \leq \int_X \phi_{1,\delta} \thinspace dA_q$,
		\item $ \displaystyle \int_X (\phi_{1,\delta} -\phi_{0,\delta}) \thinspace dA_q \preceq \delta \cdot \ell_{\beta}(q)$,
		\item $\|\phi_{i,\delta}\|_{1,q} \preceq \ell_{\beta}(q) \cdot \smash{\ell_{\min}^\dagger}(q)^{-3} \cdot v_\beta(q)^{-1} \cdot \delta^{-1}$ for $i \in \{0,1\}$,
	\end{enumerate}
	and such that for every horizontal segment $\gamma$ of $q$,
	\begin{gather*}
	\int_{\gamma} \phi_{0,\delta} \thinspace d\Re(q) - \sum_{j=1}^m \#(\gamma \cap \beta_j)\preceq \frac{\ell_{\beta}(q)}{\ell_{\min}(q)}, \\
	\sum_{j=1}^m \#(\gamma \cap \beta_j) - \int_{\gamma} \phi_{1,\delta} \thinspace d\Re(q) \preceq \frac{\ell_{\beta}(q)}{\ell_{\min}(q)}.
	\end{gather*}
\end{proposition}

\section{Effective equidistribution of straight line flows}

\subsection*{Outline of this section.} In this section we show that horizontal segments of a quadratic differential equidistribute over the underlying Riemann surface at an effective rate if the corresponding Teichmüller geodesic flow orbit satisfies appropriate recurrence conditions. See Theorem \ref{theo:equid_3}. We first prove an analogous result for Abelian differentials, see Theorem \ref{theo:equid_2}, and then deduce the corresponding result for quadratic differentials by passing to canonical double covers. We follow the general approach of Athreya and Forni \cite{AF08} but quantify our statements in a way suitable to applications in later sections. We begin with a brief overview of the tools introduced in the work of Athreya and Forni. 

\subsection*{Strata of Abelian differentials.} Recall that if $\omega$ is an Abelian differential on a closed Riemann surface $X$ of genus $g \geq 2$, the number of zeroes of $\omega$ counted with multiplicity is $2g-2$. Denote by $\Omega_g^1$ the moduli space of unit area, genus $g \geq 2$ Abelian differentials. This space can be stratified according to the order of the zeroes, with one stratum per integer partition of $2g-2$. For the rest of this section we fix $g \geq 2$ and denote by $\Omega^1 \subseteq \Omega^1_g$ an arbitrary stratum of $\Omega^1_g$.

\subsection*{The Teichmüller geodesic flow.} Recall that Abelian differentials are in one-to-one correspondence with translation structures on Riemann surfaces. The group $\mathrm{SL}(2,\mathbf{R})$ acts naturally on Abelian differentials by postcomposing the charts of the corresponding translation structures with the linear action of $\mathrm{SL}(2,\mathbf{R})$ on $\mathbf{C} =\mathbf{R}^2$. This action preserves the order of the zeroes and the area. In particular, it preserves $\Omega^1_g$ and its stratification. The \textit{Teichmüller geodesic flow} on $\Omega_g^1$ is the flow corresponding to the the diagonal subgroup $\{a_t\}_{t \in \mathbf{R}} \subseteq \mathrm{SL}(2,\mathbf{R})$ introduced in (\ref{eq:diag}).

\subsection*{Weighted Sobolev spaces of currents.} Let $(X,\omega)$ be an Abelian differential. Denote by $\Sigma \subseteq X$ the set of zeroes of $\omega$. Recall that the oriented singular measured foliations $\Im(\omega)$ and $\Re(\omega)$ on $X$ induce a pair of canonical continuous orthonormal vector fields $\partial x_\omega$ and $\partial y_\omega$ on $X \setminus \Sigma$. Recall the definition of the weighted Sobolev space of functions $H_\omega^1(X)$ introduced in (\ref{eq:sob}) and of its norm $\|\cdot \|_{1,\omega}$ introduced in (\ref{eq:sob_norm}). Denote by $\Psi^1(X\setminus \Sigma)$ the space of measurable $1$-forms on $X \setminus \Sigma$. The \textit{weighted Sobolev space of $1$-forms} $W_\omega^1(X) \subseteq \Psi^1(X\setminus \Sigma)$ is defined as 
\[
W_\omega^1(X) := \{\psi \in \Psi^1(S \setminus \Sigma) \ | \ (\iota_{\partial x_\omega} \psi, \iota_{\partial y_\omega} \psi) \in H_\omega^1(X) \times H_\omega^1(X) \},
\]
where $\iota_{\partial x_\omega} \psi$ and $\iota_{\partial y_\omega} \psi$ denote contractions. The space $W_\omega^1(X)$ inherits a Hilbert space structure from its natural identification with $H_\omega^1(X) \times H_\omega^1(X)$. Denote its norm by $\|\cdot\|_{1,\omega}$. The \textit{weighted Sobolev space of currents} $W_\omega^{-1}(X)$ is defined as the Hilbert space dual to $W_\omega^1(X)$. Denote its norm by $\|\cdot \|_{-1,\omega}$.

Recall that $\ell_{\min}(\omega)$ denotes the length of the shortest saddle connections of $\omega$. Denote by $\ell_\gamma(\omega)$ the length of a straight line segment $\gamma$ of $\omega$. By the Sobolev trace theorem, see for instance \cite[Chapter 4]{A03}, any piecewise smooth path on $X$ induces a current that belongs to $W_\omega^{-1}(X)$. The following result of Forni bounds the Sobolev norm of currents induced by straight line segments of $\omega$.

\begin{lemma}\cite[Lemma 9.2]{F02}
	\label{lem:path_sov_bound}
	Let $(X,\omega) \in \Omega^1_g$ and $\gamma$ be a straight line segment of $\omega$. Then,
	\[
	\|\gamma\|_{-1,\omega} \preceq_{g}  1+ \ell_\gamma(\omega)/\ell_{\min}(\omega).
	\]
\end{lemma}

Recall that $\omega$ induces a pair of canonical singular $1$-forms $\mathrm{Re}(\omega)$ and $\mathrm{Im}(\omega)$ on $X$. These $1$-forms can be considered as currents on $X$ by duality with respect to integration of wedge products. Notice  
\begin{equation}
	\label{eq:basic_bound}
	\|\mathrm{Re}(\omega)\|_{-1,\omega} = 1, \ \|\mathrm{Im}(\omega)\|_{-1,\omega} = 1.
\end{equation}
In terms of this identification, the \textit{tautological subspace} $J_\omega(X) \subseteq W_\omega^{-1}(X)$ is defined as 
\[
J_\omega(X) := \mathbf{R} \cdot \mathrm{Re}(\omega) \oplus \mathbf{R} \cdot \mathrm{Im}(\omega).
\]
Denote by $J_\omega^{\perp}(X)$ the orthogonal of $J_\omega(X)$ in $W_\omega^{-1}(X)$ with respect to integration of wedge products. The subspace of \textit{closed currents} $Z_\omega^{-1}(X) \subseteq W_\omega^{-1}(X)$ is defined as
\[
Z_\omega^{-1}(X) := \{\gamma \in W_\omega^{-1}(X) \ | \ d\gamma = 0\},
\]
where the exterior derivative operator $d$ is defined in the weak sense with respect to an appropriate space of test functions \cite[\S6]{F02}. The subspace $Z_\omega^{-1}(X) \subseteq W_\omega^{-1}(X)$ is closed and contains $J_\omega(X)$. Closed piecewise smooth paths on $X$ induce closed currents that belong to $Z_\omega^{-1}(X)$.

\subsection*{The cocycle of currents.} The moduli space $\smash{\Omega^1_g}$ supports an orbifold bundle $\smash{W_g^{-1}}$ whose fiber above $(X,\omega) \in \smash{\Omega_g}$ is the weighted Sobolev space of currents $\smash{W^{-1}_\omega}(X)$. This bundle supports a cocycle $A := \{A_t \}_{t \in \mathbf{R}}$ over the Teichmüller geodesic flow which acts on fibers by parallel transport \cite[\S3.3]{AF08}. This cocycle satisfies the following inequality for every $\gamma \in W_\omega^{-1}(X)$ and every $t > 0$,
\begin{equation*}
	\label{eq:cocycle_bd}
	\|A_t \gamma\|_{-1,a_t\omega} \leq e^{2|t|} \cdot \|\gamma\|_{-1,\omega}.
\end{equation*}
The tautological subspaces, their orthogonals with respect to integration of wedge products, and the subspaces of closed currents fit into $A$-invariant subbundles $J_g, J_g^\perp,Z_g^{-1} \subseteq W_g^{-1}$.

\subsection*{Spectral gap.} Let $\delta_g \colon \smash{W_g^{-1}} \to \mathbf{R}$ be the distance function to the subbundle $\smash{Z_g^{-1}} \subseteq \smash{W_g^{-1}}$ of closed currents. More precisely, for every $(X,\omega) \in \Omega^1_g$, the restriction $\smash{\delta_g|_{W_\omega^{-1}(X)}} \colon W_\omega^{-1}(X) \to \mathbf{R}$ is the distance function on $W_\omega^{-1}(X)$ to the closed Hilbert subspace $Z_\omega^{-1}(X) \subseteq W_\omega^{-1}(X)$ with respect to the metric induced by the norm $\|\cdot\|_{-1,\omega}$. Fix a stratum $\Omega^1 \subseteq \Omega_g^1$. Given $K \subseteq \Omega^1$ compact and $\delta > 0$, let $\Gamma_K(\delta) \subseteq W_g^{-1}$ be the subset over $\Omega^1$ defined by the condition that for every $(X,\omega) \in \Omega^1$,
\[
\Gamma_K(\delta) \cap W_\omega^{-1}(X) := \{\gamma \in W_\omega^{-1}(X) \ | \ \forall t \in \mathbf{R} \colon a_t\omega \in K \Rightarrow \delta_g(A_t \gamma) \leq \delta \}.
\]
More concretely, $\Gamma_K(\delta)$ is the subset of all currents over $\Omega^1$ which stay at distance $\leq \delta$ from the subbundle of closed currents for all visiting times of the corresponding Teichmüller geodesic flow orbit to the compact subset $K \subseteq \Omega^1$. Lemma \ref{lem:path_sov_bound} provides a tool for checking when currents induced by straight line segments belong to subsets of this kind.

\begin{proposition}
	\label{prop:delta_close}
	Let $\Omega^1 \subseteq \Omega_g^1$ be a stratum and $K \subseteq \Omega^1$ be a compact subset. Then, there exists a constant $\delta = \delta(K) > 0$ such that for every $\omega \in \Omega^1$ and every straight line segment $\gamma$ of $\omega$, $\gamma \in \Gamma_K(\delta)$.
\end{proposition}

\begin{proof}
	Let $(X,\omega) \in \Omega^1$ and $\gamma$ be a straight line segment of $\omega$. Recall that $\mathrm{diam}(\omega)$ denotes the diameter of $\omega$. The endpoints of $\gamma$ can be joined along a flat geodesic path of length $\leq \mathrm{diam}(\omega)$ to obtain a closed path $\overline{\gamma} \in Z_\omega^{-1}(X)$. The difference $\overline{\gamma} - \gamma \in W_\omega^{-1}(X)$ is a concatenation of $\preceq \mathrm{diam}(\omega)/\ell_{\min}(\omega)$ straight line segments. By the triangle inequality and Lemma \ref{lem:path_sov_bound}, 
	\[
	\| \overline{\gamma}- \gamma \|_{-1,\omega} \preceq_g \mathrm{diam}(\omega)/\ell_{\min}(\omega).
	\]
	The quantities $\mathrm{diam}(\omega)$ and $\ell_{\min}(\omega)$ can be bounded uniformly away from $+\infty$ and $0$ on compact subsets of $\Omega^1$. The desired conclusion follows.
\end{proof}

For every $r > 0$ let $r^+ := \max\{1,r\}$. In \cite[Lemma 4.2]{AF08} a continuous function $\Lambda \colon \Omega^1 \to [0,1)$ quantifying the spectral gap of the Kontsevich-Zorich cocycle on the stratum $\Omega^1$ is contructed. The same function can be used to quantify the spectral gap of the cocyle $A$ as follows.

\begin{lemma} \cite[Lemma 4.5]{AF08}
	\label{lem:spec_gap}
	Let $\Omega^1 \subseteq \smash{\Omega_g^1}$ be a stratum, $K \subseteq \Omega^1$ be a compact subset, and $\delta > 0$. Consider $(X,\omega) \in \Omega^1$ and $\gamma \in \smash{J_\omega^\perp} \cap \Gamma_{K}(\delta)$. Let $\{s_j\}_{j=1}^k$ with $k \in \mathbf{N}^+$ be an increasing sequence of visiting times of the forward orbit $\{a_t \omega\}_{t \geq 0}$ to $K$. Denote $\omega_k := a_{s_k} \omega$. Then,
	\[
	\|\gamma\|_{-1,\omega} \preceq_{K} \delta^+ \cdot \|A_{s_k}\gamma\|_{-1,\omega_k}^+ \cdot \exp\left( \int_{s_1}^{s_k} \Lambda(a_t \omega) \thinspace dt\right) \cdot e^{2s_1} \cdot \left( \sum_{j=1}^{k-1} e^{2(s_{j+1} -s_j)} \right)^2.
	\]
\end{lemma}

\subsection*{Greedy partitions.} The following lemma corresponds to the output of a greedy algorithm for partitioning vertical segments of Abelian differentials according to increasing sequences of lengths.

\begin{lemma} \cite[Lemma 5.1]{AF08}
	\label{lem:part}
	Let $\omega \in \Omega_g^1$, $\gamma$ be a vertical segment of $\omega$, and $\{T_k\}_{k=1}^N$ with $N \in \mathbf{N}^+$ be an increasing sequence of positive real numbers. Denote $T_{N+1} := \ell_{\gamma}(\omega)$. Then, there exist a decomposition of $\gamma$ into consecutive subsegments
	\[
	\gamma = \gamma_0 + \sum_{k=1}^N \sum_{m=1}^{m_k} \gamma_{k,m}
	\]
	such that $\ell_{\gamma_0}(\omega) < T_1$, $\ell_{\gamma_{k,m}}(\omega) = T_k$ for every $k \in \{1,\dots, N\}$ and every $m \in \{1,\dots,m_k\}$, and $m_k \leq T_{k+1} \thinspace \smash{T_k^{-1}}$ for every $k \in \{1,\dots,N\}$.
\end{lemma}

\subsection*{Summability estimate.} The following summability estimate will play an important role in the proofs of the main results of this section.

\begin{lemma}
	\label{lem:sum}
	Let $\{s_j\}_{j=1}^k$ with $k \in \mathbf{N}^+$ be a sequence of positive numbers. Suppose there exists $s > 0$ such that $s_{j+1} \geq s_j +s$ for every $j \in \{1,\dots,k-1\}$. Let $I \subseteq \mathbf{R}^+$ compact. Then, for every $\lambda \in I$,
	\[
	\sum_{j=1}^k e^{\lambda s_j} \preceq_{I,s} e^{\lambda s_k}. 
	\]
\end{lemma}

\begin{proof}
	The condition $s_{j+1} \geq s_j +s$ for $j \in \{1,\dots,k-1\}$ ensures that, for every $n \in \mathbf{N}^+$, 
	\[
	|\{j \in \{1,\dots,k\} \colon \lceil s_j \rceil = n \}| \leq \lceil 1/s \rceil.
	\]
	It follows that, for every $\lambda \in I$,
	\[
	\sum_{j=1}^k e^{\lambda s_j} \leq \sum_{j=1}^k e^{\lambda \lceil s_j \rceil} \leq \left\lceil \frac{1}{s} \right\rceil \cdot \sum_{n=1}^{\lceil s_k \rceil} e^{\lambda n} \leq \left\lceil \frac{1}{s} \right\rceil \cdot \left( \frac{e^{2\lambda}}{e^\lambda - 1}\right) \cdot e^{\lambda s_k}\preceq_{I,s} e^{\lambda s_k}. \qedhere
	\]
\end{proof}

\subsection*{Effective equidistribution of straight line flows.} We now state and prove effective equidistribution theorems for vertical and horizontal segments of Abelian differentials.

Let $(X,\omega)$ be an Abelian differential. Recall that $\Re(\omega)$ and $\Im(\omega)$ denote the vertical and horizontal foliations of $\omega$ endowed with their corresponding transverse measures. By the Sobolev trace theorem, for every vertical segment $\gamma$ of $\omega$, every horizontal segment $\sigma$ of $\omega$, and every Sobolev function $\phi \in H_\omega^1(X)$, the following integrals are well defined,
\[
\int_{\gamma} \phi \thinspace d\Im(\omega), \quad \int_{\sigma} \phi \thinspace d\Re(\omega).
\]

Fix a stratum $\Omega^1 \subseteq \Omega_g^1$. Let $K \subseteq \Omega^1$ be a compact subset, $T > 0$, and  $\eta \in (0,1)$. An Abelian differential $\omega \in \Omega^1$ is said to be \textit{$(K,T,\eta)$-recurrent} if 
\begin{equation*}
	|\{t \in [0,T] \colon a_t \omega \notin K \}| \leq \eta T.
\end{equation*}

Let $K' \subseteq \Omega^1$ be a compact subset, $T > 0$, $\rho  \in (0,1)$, $\epsilon > 0$, and $s > 0$. A sequence of positive real numbers $\{s_n\}_{n=0}^{N+1}$ with $N \in \mathbf{N}^+$ is said to be a \textit{$(K',T,\rho ,\epsilon,s)$-itinerary} of $\omega \in \Omega^1$ if 
\begin{enumerate}
	\item $a_{s_n} \omega \in K'$ for every $n \in \{1, \dots, N\}$,
	\item $s_0 = \rho  T$,
	\item $s_N < T$, 
	\item $s_{N+1} = T$,
	\item $s_{n+1} \geq s_n + s$ for every $n \in \{1,\dots, N-1\}$,
	\item $s_{n+1} - s_n \leq \epsilon s_n$ for every $n \in \{0,\dots,N\}$.
\end{enumerate}

Let $(X,\omega)$ be an Abelian differential. Recall that $A_\omega$ denotes the flat area form induced $\omega$ on $X$ and that $\smash{\ell^\dagger_{\min}}(\omega) := \min\{1,\ell_{\min}(\omega)\}$. The following theorem shows that vertical segments of $\omega$ equidistribute towards $A_\omega$ at an effective rate if the Teichmüller geodesic flow orbit of $\omega$ satisfies appropriate recurrence conditions. Our proof follows the general outline of the proof of \cite[Theorem 1.1]{AF08}.

\begin{theorem}
	\label{theo:equid_1}
	Let $\Omega^1 \subseteq \Omega_g^1$ be a stratum. For every compact subset $K \subseteq \Omega^1$ and every $\eta \in (0,1)$ there exist constants $\rho_1 = \rho _1(K,\eta) > 0$, $\epsilon_1 = \epsilon_1(K,\eta) > 0$, and $\kappa_1 = \kappa_1(K,\eta) > 0$ such that for every $T > 0$, every compact subset $K' \subseteq \Omega^1$, every $0 < \rho  < \rho_1$, every $0 < \epsilon < \epsilon_1$, and every $s > 0$, the following holds. Suppose $(X,\omega) \in \Omega^1$ is $(K,T,\eta)$-recurrent and has a $(K',T,\rho ,\epsilon,s)$-itinerary, $\gamma$ is a vertical segment of $\omega$ of length $A e^T$ for some $A > 0$, and $\phi \in H^1_\omega(X)$. Then, 
	\[
	\bigg\vert \int_{\gamma} \phi \thinspace d\Im(\omega) - A e^T \int_X \phi \thinspace d A_\omega \bigg\vert \preceq_{K',s} A^+ \cdot \ell_{\min}^\dagger(\omega)^{-1} \cdot \|\phi\|_{1,\omega} \cdot e^{(1-\kappa_1)T}.
	\]
\end{theorem}

\begin{proof} 
	Fix a stratum $\Omega^1 \subseteq \Omega_g^1$, a compact subset $K \subseteq \Omega^1$, and $\eta \in (0,1)$. Let $T > 0$ arbitrary, $K' \subseteq \Omega^1$ an arbitrary compact subset, $\rho  \in (0,1)$ and $\epsilon >0$ to be restricted in the course of the proof, and $s > 0$ arbitrary. Suppose $(X,\omega) \in \Omega^1$ is $(K,T,\eta)$-recurrent and has a $(K',T,\rho ,\epsilon,s)$-itinerary $\{s_n\}_{n=0}^{N+1}$. Denote by $\langle \cdot, \cdot \rangle$ the duality pairing between currents in $W_\omega^{-1}(X)$ and $1$-forms in $W_\omega^{1}(X)$. Consider the current $\widehat{\gamma} := \gamma - Ae^T  \mathrm{Re}(\omega) \in W_\omega^{-1}(S)$. Notice that
	\begin{equation}
		\label{eq:intuition}
		\bigg\vert \int_{\gamma} \phi \thinspace d\Im(\omega) - A e^T \int_X \phi \thinspace d A_\omega \bigg\vert = \langle \widehat{\gamma}, \phi \cdot \mathrm{Im}(\omega) \rangle \leq \| \widehat{\gamma}\|_{-1,\omega} \cdot \|\phi\|_{1,\omega}.
	\end{equation}
	It follows that, to prove the desired estimate, it is enough to bound $\|\widehat{\gamma}\|_{-1,\omega}$. We do this by decomposing $\widehat{\gamma}$ using Lemma \ref{lem:part} with respect to an appropriate sequence of lengths and then applying Lemma \ref{lem:spec_gap} to bound the norm of each one of the terms in this decomposition.
	
	Let $T_n := A e^{s_n}$ for every $n \in \{1,\dots,N\}$. Denote $T_{N+1} := A e^{s_{N+1}} = A e^T = \ell_{\gamma}(\omega)$. Applying Lemma \ref{lem:part} to the vertical segment $\gamma$ of $\omega$ and the increasing sequence of positive real numbers $\{T_n\}_{n=1}^N$ yields a decomposition of $\gamma$ into consecutive subsegments
	\[
	\gamma = \gamma_0 + \sum_{k=1}^N \sum_{m=1}^{m_k} \gamma_{k,m} 
	\]
	such that $\tau:= \ell_{\gamma_0}(\omega) < T_1 = A e^{s_1}$, $\ell_{\gamma_{k,m}}(\omega) = T_k = Ae^{s_k}$ for every $k \in \{1,\dots,N\}$ and every $m \in \{1,\dots,m_k\}$, and $m_k \leq T_{k+1} \thinspace T_k^{-1} = e^{s_{k+1}- s_k}$ for every $k \in \{1,\dots,N\}$. Consider the currents
	\begin{gather*}
		\widehat{\gamma}_0 := \gamma_0 - \tau \cdot \mathrm{Re}(\omega) \in W_\omega^{-1}(X), \\
		\widehat{\gamma}_{k,m} := \gamma_{k,m} - T_k \cdot \mathrm{Re}(\omega) \in W_\omega^{-1}(X).
	\end{gather*}
	Decompose the current $\widehat{\gamma} \in W_\omega^{-1}(X)$ as follows,
	\begin{equation}
		\label{eq:decomp}
		\widehat{\gamma} = \widehat{\gamma}_0 + \sum_{k=1}^N \sum_{m=1}^{m_k} \widehat{\gamma}_{k,m}.
	\end{equation}
	We bound the $\|\cdot\|_{-1,\omega}$ norm of each one of the terms in this decomposition.
	
	We first bound $\|\widehat{\gamma}_0\|_{-1,\omega}$. By Lemma \ref{lem:path_sov_bound} and (\ref{eq:basic_bound}),
	\begin{equation*}
		\|\widehat{\gamma}_0\|_{-1,\omega} \preceq_g A^+ \cdot \ell_{\min}^{\dagger}(\omega)^{-1} \cdot e^{s_1} .
	\end{equation*}
	As $s_0 = \rho  T$ and $s_1 \leq (1+\epsilon) s_0$, we have $s_1 \leq (1+\epsilon)\rho  T$. In particular,
	\begin{equation*}
		\|\widehat{\gamma}_0\|_{-1,\omega} \preceq_g A^+ \cdot \ell_{\min}^{\dagger}(\omega)^{-1} \cdot e^{(1+\epsilon)\rho  T}.
	\end{equation*}
	Let $\rho_2 \in (0,1)$ and $\epsilon_2 > 0$ be small enough so that $(1+\epsilon_2) \rho_2 < 1$. Denote $\kappa_2 := 1- (1+\epsilon_2) \rho _2 > 0$. Then, if $0 < \rho  < \rho _2$ and $0 < \epsilon < \epsilon_2$,
	\begin{equation}
		\label{eq:gamma_0_estimate}
		\|\widehat{\gamma}_0\|_{-1,\omega} \preceq_g A^+ \cdot \ell_{\min}^{\dagger}(\omega)^{-1} \cdot e^{(1-\kappa_2) T} .
	\end{equation}
	
	Let $k \in \{1,\dots,N\}$ and $m \in \{1,\dots,m_k\}$. We now bound $\|\widehat{\gamma}_{k,m}\|_{-1,\omega}$. A direct computation shows that $\widehat{\gamma}_{k,m} \in J_\omega^\perp(X)$. Let $\delta =\delta(K') > 0$ be as in Proposition \ref{prop:delta_close}. As $\mathrm{Re}(\omega) \in J_\omega(X) \subseteq Z_\omega^{-1}(X)$, it follows from Proposition \ref{prop:delta_close} that $\widehat{\gamma}_{k,m} \in \Gamma_{K'}(\delta)$. We can thus apply Lemma \ref{lem:spec_gap} to estimate $\|\widehat{\gamma}_{k,m}\|_{-1,\omega}$. Denote $\omega_k:= a_{s_k} \omega \in K'$ as in Lemma \ref{lem:spec_gap}.
	
	Let $M \in \{1,\dots,N\}$ be such that $s_{M-1} <(1+\eta) T/2$ and $s_M \geq (1+\eta) T/2$. This integer exists under the conditions $\rho  \leq (1+\eta)/2$ and $(1+\epsilon)(1+\eta)/2 \leq 1$. Lemma \ref{lem:spec_gap} provides different types of estimates in the regimes $k < M$ and $k \geq M$, so we take care to consider these cases separately. 
	
	Suppose $k \in \{1,\dots,M-1\}$. Let us first bound $\|A_{s_k} \widehat{\gamma}_{k,m}\|_{-1,\omega_k}$. As $\gamma_{k,m}$ is a vertical segment of length $T_k = A e^{s_k}$ on $\omega$, the current $A_{s_k} \gamma_{k,m}$ is induced by a vertical segment of length $A$ on $\omega_k \in K'$. Notice also that $A_{s_k} \mathrm{Re}(\omega) = e^{-s_k} \mathrm{Re}(\omega_k)$. These facts together with Lemma \ref{lem:path_sov_bound} and (\ref{eq:basic_bound}) imply
	\begin{equation*}
		\|A_{s_k} \widehat{\gamma}_{k,m}\|_{-1,\omega_k} \preceq_{K'}  A^+.
	\end{equation*}
	Applying Lemma \ref{lem:spec_gap} then yields
	\begin{equation}
		\label{eq:bound_1}
		\|\widehat{\gamma}_{k,m}\|_{-1,\omega} \preceq_{K'} A^+ \cdot \exp\left(\int_{s_1}^{s_k} \Lambda(a_t \omega) \thinspace dt\right) \cdot e^{2s_1} \cdot \left( \sum_{j=1}^{k-1} e^{2(s_{j+1} - s_j)}\right)^2.
	\end{equation}
	As we are in the regime $k < M$, we consider the trivial bound
	\begin{equation}
		\label{eq:bound_2}
		\exp\left(\int_{s_1}^{s_k} \Lambda(a_t \omega) \thinspace dt\right) \leq e^{s_k}.
	\end{equation}
	The conditions $s_{j+1} - s_j \leq \epsilon s_j$ for $j \in \{1,\dots,k-1\}$ and $s_{j+1} \geq s_j + s$ for $j \in \{1,\dots,k-2\}$ together with Lemma \ref{lem:sum} imply
	\begin{equation}
		\label{eq:bound_3}
		\left( \sum_{j=1}^{k-1} e^{2(s_{j+1} - s_j)}\right)^2 \leq \left( \sum_{j=1}^{k-1} e^{2\epsilon s_j}\right)^2 \preceq_{\epsilon,s} e^{4\epsilon s_{k-1}} \leq e^{4\epsilon s_k}.
	\end{equation}
	Putting together (\ref{eq:bound_1}), (\ref{eq:bound_2}), and (\ref{eq:bound_3}), we deduce
	\begin{equation}
		\label{eq:bound_4}
		\|\widehat{\gamma}_{k,m}\|_{-1,\omega} \preceq_{K',\epsilon,s} A^+ \cdot e^{2s_1} \cdot e^{(1+4\epsilon)s_k}.
	\end{equation}
	
	Let us add all contributions coming from the regime $k < M$. The conditions $m_k \leq e^{s_{k+1} -s_k} \leq e^{\epsilon s_k}$ for $k \in \{1,\dots,M-1\}$ and $s_{k+1} \geq s_k + s$ for $k \in \{1,\dots,M-2\}$ together with Lemma \ref{lem:sum} imply
	\begin{equation}
		\label{eq:rep1}
		\sum_{k=1}^{M-1} m_k \thinspace e^{(1+4\epsilon)s_k} \leq \sum_{k=1}^{M-1} e^{(1+5\epsilon)s_k} \preceq_{\epsilon,s} e^{(1+5\epsilon)s_{M-1}}.
	\end{equation}
	From (\ref{eq:bound_4})  and (\ref{eq:rep1}) we deduce
	\[
	\sum_{k=1}^{M-1} m_k \thinspace \|\widehat{\gamma}_{k,m}\|_{-1,\omega} \preceq_{K',\epsilon,s} A^+ \cdot e^{2s_1} \cdot e^{(1+5\epsilon)s_{M-1}}.
	\]
	Recall $s_1 \leq (1+\epsilon)\rho  T$ and $s_{M-1} <(1+\eta) T/2$. It follows that
	\begin{equation*}
		\sum_{k=1}^{M-1} m_k \thinspace \|\widehat{\gamma}_{k,m}\|_{-1,\omega} \preceq_{K',\epsilon,s} A^+ \cdot e^{2(1+\epsilon) \rho  T} \cdot e^{(1+5\epsilon)(1+\eta) T/2}.
	\end{equation*}
	Let $\rho_3 = \rho_3(\eta) \in (0,1)$ and $\epsilon_3 = \epsilon_3(\eta) > 0$ be small enough so that $(1+5\epsilon_3)(1+\eta)/2 < 1$ and $2(1+\epsilon_3)\rho_3 < 1 - (1+5\epsilon_3)(1+\eta)/2$. Denote $\kappa_3 = \kappa_3(\eta) := 1 - (1+5\epsilon_3)(1+\eta)/2 - 2(1+\epsilon_3)\rho_3 > 0$. Then, if $0 < \rho  < \rho _3$ and $0 < \epsilon < \epsilon_3$,
	\begin{equation}
		\label{eq:mid_bound}
		\sum_{k=1}^{M-1} m_k \thinspace \|\widehat{\gamma}_{k,m}\|_{-1,\omega} \preceq_{K',\epsilon,s} A^+ \cdot e^{(1-\kappa_3) T}.
	\end{equation}
	
	Suppose now $k \in \{M,\dots,N\}$. The same arguments used to prove (\ref{eq:bound_1}) show that
	\begin{equation}
		\label{eq:bound_b1}
		\|\widehat{\gamma}_{k,m}\|_{-1,\omega} \preceq_{K'} A^+ \cdot \exp\left(\int_{s_1}^{s_k} \Lambda(a_t \omega) \thinspace dt\right) \cdot e^{2s_1} \cdot \left( \sum_{j=1}^{k-1} e^{2(s_{j+1} - s_j)}\right)^2.
	\end{equation}
	As we are in the regime $k \geq M$, we have $s_k \geq (1+\eta) T/2$. In particular, as $\omega$ is $(K,T,\eta)$-recurrent,
	\[
	|\{t \in [0,s_k] \colon a_t \omega \notin K \}| \leq 2\eta s_k/(1+\eta).
	\]
	As the function $\Lambda \colon \Omega^1 \to [0,1)$ is continuous, it follows that, for some $\kappa_4 = \kappa_4(K,\eta) \in (0,1)$, 
	\begin{equation}
		\label{eq:bound_b2}
		\exp\left(\int_{s_1}^{s_k} \Lambda(a_t \omega) \thinspace dt\right) \leq e^{(1-\kappa_4)s_k}.
	\end{equation}
	The same arguments used to prove (\ref{eq:bound_3}) show that
	\begin{equation}
		\label{eq:bound_b3}
		\left( \sum_{j=1}^{k-1} e^{2(s_{j+1} - s_j)}\right)^2 \preceq_{\epsilon,s} e^{4\epsilon s_k}.
	\end{equation}
	Putting together (\ref{eq:bound_b1}), (\ref{eq:bound_b2}), and (\ref{eq:bound_b3}), we deduce
	\begin{equation}
		\label{eq:bound_b4}
		\|\widehat{\gamma}_{k,m}\|_{-1,\omega} \preceq_{K',\epsilon,s} A^+ \cdot e^{2s_1} \cdot e^{(1-\kappa_4 + 4\epsilon)s_k}.
	\end{equation}
	
	Let us add all contributions coming from the regime $k \geq M$. The same arguments used to prove (\ref{eq:rep1}) show that
	\begin{equation*}
		\sum_{k=M}^{N} m_k \thinspace e^{(1-\kappa_4 + 4\epsilon)s_k}  \preceq_{\epsilon,s} e^{(1- \kappa_4 + 5\epsilon)s_{N}}.
	\end{equation*}
	From this and (\ref{eq:bound_b4}) we deduce
	\[
	\sum_{k=M}^{N} m_k \thinspace \|\widehat{\gamma}_{k,m}\|_{-1,\omega} \preceq_{K',\epsilon,s} A^+ \cdot e^{2s_1} \cdot e^{(1-\kappa_4 + 5\epsilon)s_{N}}.
	\]
	Recall $s_1 \leq (1+\epsilon)\rho  T$ and $s_N < T$. It follows that
	\begin{equation*}
		\sum_{k=M}^{N} m_k \thinspace \|\widehat{\gamma}_{k,m}\|_{-1,\omega} \preceq_{K',\epsilon,s} A^+ \cdot e^{2(1+\epsilon) \rho  T} \cdot e^{(1-\kappa_4+5\epsilon)T}.
	\end{equation*}
	Let $\rho _4 = \rho _4(K,\eta) \in (0,1)$ and $\epsilon_4 = \epsilon_4(K,\eta) > 0$ be small enough so that $1-\kappa_4 + 5\epsilon_4< 1$ and $2(1+\epsilon_4)\rho _4 < 1 - (1-\kappa_4 + 5\epsilon_4)$. Denote $\kappa_5 = \kappa_5(K,\eta) := 1 - (1-\kappa_4+5\epsilon_4) - 2(1+\epsilon_4)\rho _4 > 0$. Then, if $0 < \rho  < \rho _4$ and $0 < \epsilon < \epsilon_4$,
	\begin{equation}
		\label{eq:top_bound}
		\sum_{k=M}^{N} m_k \thinspace \|\widehat{\gamma}_{k,m}\|_{-1,\omega} \preceq_{K',\epsilon,s} A^+ \cdot e^{(1-\kappa_5) T}.
	\end{equation}
	
	Let $\rho_1 = \rho_1(K,\eta):= \min\{\rho_2,\rho _3,\rho_4\} > 0$, $\epsilon_1 = \epsilon_1(K,\eta):= \min\{\epsilon_2,\epsilon_3,\epsilon_4\} > 0$, and $\kappa_1 = \kappa_1(K,\eta):= \min\{\kappa_2,\kappa_3,\kappa_5\} > 0$. From (\ref{eq:decomp}), (\ref{eq:gamma_0_estimate}), (\ref{eq:mid_bound}), (\ref{eq:top_bound}), and the triangle inequality, we deduce that, if $0 < \rho  < \rho_1$ and $0 < \epsilon < \epsilon_1$,
	\[
	\|\widehat{\gamma}\|_{-1,\omega} \preceq_{K',\epsilon,s} A^+ \cdot \ell_{\min}^\dagger(\omega)^{-1} \cdot e^{(1-\kappa_1)T}.
	\]
	This together with (\ref{eq:intuition}) allows us to conclude that, if $0 < \rho  < \rho_1$ and $0 < \epsilon < \epsilon_1$,
	\[
	\bigg\vert \int_{\gamma} \phi \thinspace d\Im(\omega) - A e^T \int_X \phi \thinspace d A_\omega \bigg\vert \preceq_{K',\epsilon,s} A^+ \cdot \ell_{\min}^\dagger(\omega)^{-1} \cdot \|\phi\|_{1,\omega} \cdot e^{(1-\kappa_1)T}. \qedhere
	\]
\end{proof}

For every $\theta \in [0,2\pi]$ consider the rotation matrix $r_\theta \in \mathrm{SL}(2,\mathbf{R})$ given by
\[
r_{\theta} := \left( \begin{array}{c c}
\cos \theta & - \sin \theta \\
\sin \theta & \cos \theta
\end{array}\right).
\]
The action of $r_{\pi/2}$ on Abelian differentials exchanges the vertical and horizontal foliations. This fact together with Theorem \ref{theo:equid_1} allows us to deduce the following effective equidistribution theorem for horizontal segments of Abelian differentials.

\begin{theorem}
	\label{theo:equid_2}
	Let $\Omega^1 \subseteq \Omega_g^1$ be a stratum. For every compact subset $K \subseteq \Omega^1$ and every $\eta \in (0,1)$ there exist constants $\rho_1 = \rho_1(K,\eta) > 0$, $\epsilon_1 = \epsilon_1(K,\eta) > 0$, and $\kappa_1 = \kappa_1(K,\eta) > 0$ such that for every $T > 0$, every $K' \subseteq \Omega^1$ compact, every $0 < \rho  < \rho_1$, every $0 < \epsilon < \epsilon_1$, and every $s > 0$, the following holds. Suppose $(X,\omega) \in \Omega^1$ is such that $r_{\pi/2} \thinspace\omega$ is $(K,T,\eta)$-recurrent and has a $(K',T,\rho ,\epsilon,s)$-itinerary, $\gamma$ is a horizontal segment of $\omega$ of length $A e^T$ for some $A > 0$, and $\phi \in H^1_\omega(X)$. Then,
	\[
	\bigg\vert \int_{\gamma} \phi \thinspace d\Re(\omega) - A e^T \int_X \phi \thinspace d A_\omega \bigg\vert \preceq_{K',\epsilon,s} A^+ \cdot \ell_{\min}^\dagger(\omega)^{-1} \cdot \|\phi\|_{1,\omega} \cdot e^{(1-\kappa_1)T}.
	\]
\end{theorem}

A result analogous to Theorem \ref{theo:equid_2} for quadratic differentials can be deduced by passing to canonical double covers. To state this result precisely, we first introduce some notation.

\subsection*{Strata of quadratic differentials.} Recall that if $q$ is a quadratic differential on a closed Riemann surface $X$ of genus $g \geq 2$, the number of zeroes of $q$ counted with multiplicity is $4g-4$. Denote by $\qum$ the moduli space of unit area, genus $g \geq 2$ quadratic differentials. This space can be stratified according to the order of the zeroes and the condition of being the square of an Abelian differential. For the rest of this section we fix $g \geq 2$ and denote by $\mathcal{Q}^1 \subseteq \qum$ an arbitrary stratum of $\qum$. The natural $\mathrm{SL}(2,\mathbf{R})$ action on quadratic differentials introduced in the paragraph preceding (\ref{eq:diag}) preserves $\qum$ and its stratification.

\subsection*{Masur's compactness criterion.} Let $\mathcal{Q}^1 \subseteq \qum$ be a stratum of quadratic differentials. For every $\delta > 0$ consider the subset $K_\delta \subseteq \mathcal{Q}^1$ given by
\begin{equation}
	\label{eq:compact}
	K_\delta := \{q \in \mathcal{Q}^1 \ | \ \ell_{\min}(q) \geq \delta\}.
\end{equation}
Masur's compactness criterion \cite[Proposition 3.6]{MT02} ensures this subset is compact. An analogous result holds for strata of Abelian differentials $\Omega^1 \subseteq \Omega_g^1$.

Recall that every quadratic differential $(X,q)$ induces a canonical double cover of $X$ onto which $q$ pulls back to the square of an Abelian differential $\omega$. As this cover pulls back the zeroes of $q$ onto the zeroes of $\omega$, we must have $\ell_{\min}(q) = \ell_{\min}(\omega)$. In particular, by Masur's compactness criterion, given a stratum of quadratic differentials $\mathcal{Q}^1 \subseteq \qum$ and a compact subset $K \subseteq \mathcal{Q}^1$, there exists $g' > 0$, a stratum of Abelian differentials $\Omega^1 \subseteq \Omega^1_{g'}$, and a compact subset $K' \subseteq \Omega^1$ such that the Abelian differential $\omega$ coming from the canonical double cover of any $q \in K$ belongs to $K'$. 

\subsection*{Effective equidistribution of horizontal segments.} We now state the main result of this section. Let $(X,q)$ be a quadratic differential. Recall the definition of the weighted Sobolev space of functions $H_q^1(X)$ introduced in (\ref{eq:sob_q}). Recall that $\Re(q)$ denotes the  horizontal foliation of $\omega$ endowed with its corresponding transverse measure. By the Sobolev trace theorem, for every horizontal segment $\gamma$ of $q$ and every Sobolev function $\phi \in H_q^1(X)$, the following integral is well defined,
\[
\int_{\gamma} \phi \thinspace d\Re(q).
\]

Fix a stratum $\mathcal{Q}^1 \subseteq \qum$. Let $K \subseteq \mathcal{Q}^1$ be a compact subset, $T > 0$, and  $\eta \in (0,1)$. A quadratic differential $q \in \qum$ is said to be \textit{$(K,T,\eta)$-recurrent} if 
\begin{equation*}
	|\{t \in [0,T] \colon a_t q\notin K \}| \leq \eta T.
\end{equation*}

Let $K' \subseteq \mathcal{Q}^1$ be a compact subset, $T > 0$, $\rho  \in (0,1)$, $\epsilon > 0$, and $s > 0$. A sequence of positive real numbers $\{s_n\}_{n=0}^{N+1}$ with $N \in \mathbf{N}^+$ is said to be a \textit{$(K',T,\rho ,\epsilon,s)$-itinerary} of $q \in \qum$ if 
\begin{enumerate}
	\item $a_{s_n} q \in K'$ for every $n \in \{1, \dots, N\}$,
	\item $s_0 = \rho  T$,
	\item $s_N < T$, 
	\item $s_{N+1} = T$,
	\item $s_{n+1} \geq s_n + s$ for every $n \in \{1,\dots, N-1\}$,
	\item $s_{n+1} - s_n \leq \epsilon s_n$ for every $n \in \{0,\dots,N\}$.
\end{enumerate}
 
Let $(X,q)$ be a quadratic differential. Recall that $A_q$ denotes the singular flat area form induced $q$ on $X$ and that $\smash{\ell^\dagger_{\min}}(q) := \min\{1,\ell_{\min}(q)\}$. Recall the definition of the norm $\|\cdot \|_{1,q}$ on the weighted Sobolev space $H_q(X)$ introduced in (\ref{eq:sob_norm_q}). Notice that $-q = r_{\pi/2} \thinspace q$. The following theorem shows that horizontal segments of $q$ equidistribute towards $A_q$ at an effective rate if the Teichmüller geodesic flow orbit of $q$ satisfies appropriate recurrence conditions. This result follows directly from Theorem $\ref{theo:equid_2}$ by passing to canonical double covers.

\begin{theorem}
	\label{theo:equid_3}
	Let $\mathcal{Q}^1 \subseteq \qum$ be a stratum. For every compact subset $K \subseteq \mathcal{Q}^1$ and every $\eta \in (0,1)$ there exist constants $\rho_1 = \rho_1(K,\eta) > 0$, $\epsilon_1 = \epsilon_1(K,\eta) > 0$, and $\kappa_1 = \kappa_1(K,\eta) > 0$ such that for every $T > 0$, every $K' \subseteq \mathcal{Q}^1$ compact, every $0 < \rho  < \rho_1$, every $0 < \epsilon < \epsilon_1$, and every $s > 0$, the following holds. Suppose $(X,q) \in \qum$ is such that $-q$ is $(K,T,\eta)$-recurrent and has a $(K',T,\rho ,\epsilon,s)$-itinerary, $\gamma$ is a horizontal segment of $q$ of length $A e^T$ for some $A > 0$, and $\phi \in H^1_q(X)$. Then, 
	\[
	\bigg\vert \int_{\gamma} \phi \thinspace d\Re(q) - A e^T \int_X \phi \thinspace d A_q \bigg\vert \preceq_{K',\epsilon,s} A^+ \cdot \ell_{\min}^\dagger(q)^{-1} \cdot \|\phi\|_{1,q} \cdot e^{(1-\kappa_1)T}.
	\]
\end{theorem}

\section{A preliminary quantitative estimate for geometric intersection numbers}

\subsection*{Outline of this section.} In this section we combine results from \S 2 -- 5 to prove a preliminary quantitative estimate for the geometric intersection numbers of closed curves on a surface with respect to arbitrary geodesic currents in terms of the Teichmüller geodesic flow. See Theorem \ref{theo:prelim_2}. We first prove a quantitative estimate for the geometric intersection numbers between closed curves, see Theorem \ref{theo:prelim_1}, and then deduce  Theorem \ref{theo:prelim_2} via a density argument. Before stating Theorem \ref{theo:prelim_2}, we discuss some aspects of the theory of geodesic currents.

\subsection*{A preliminary quantitative estimate for geometric intersection numbers.} For the rest of this section we fix an integer $g \geq 2$ and a connected, oriented, closed surface $S_g$ of genus $g$. Recall that $\qut$ denotes the Teichmüller space of marked, unit area quadratic differentials on $S_g$, that $\mathcal{MF}_g$ denotes the space of singular measured foliations on $S_g$, and that $\Re(q),\Im(q) \in \mf$ denote the vertical and horizontal foliations of $q \in \qut$. The $\mathrm{SL}(2,\mathbf{R})$ action on half-translation structures introduced in \S 3 gives rise to an $\mathrm{SL}(2,\mathbf{R})$ action on $\qut$. The \textit{Teichmüller geodesic flow} on $\qut$ is the flow corresponding to the the diagonal subgroup $\{a_t\}_{t \in \mathbf{R}} \subseteq \mathrm{SL}(2,\mathbf{R})$ introduced in (\ref{eq:diag}). 

Recall that, given a pair of closed curves $\alpha$ and $\beta$ on $S_g$, their geometric intersection number $i(\alpha,\beta)$ is defined as the minimal number of intersections among pairs of transverse closed curves homotopic to $\alpha$ and $\beta$. Following Thurston \cite{T80} and Bonahon \cite{Bon88}, given a closed curve $\alpha$ on $S_g$ and a singular measured foliation $\lambda \in \mathcal{MF}_g$, we denote by $i(\alpha,\lambda)$ the minimal tranverse measure assigned by $\lambda$ to a closed curve homotopic to $\alpha$. If $\alpha$ is a closed curve on $S_g$ and $q \in \qut$, 
\[
i(\alpha,\Re(q)) := \min_{\alpha' \sim \alpha} \thinspace \int_{\alpha'} d\Re(q), \quad i(\alpha,\Im(q)) := \min_{\alpha' \sim \alpha} \thinspace \int_{\alpha'}  d\Im(q),
\]
where $\alpha'$ runs over all curves homotopic to $\alpha$. The minimum in this definition is attained by the geodesic representatives of $\alpha$ with respect to the singular flat metric induced by $q$ on $S_g$ \cite[\S 5.4]{AL17}.

Recall that $\qum$ denotes the moduli space of unit area quadratic differentials on $S_g$. Fix a stratum $\mathcal{Q}^1 \subseteq \qum$. Let $K \subseteq \mathcal{Q}^1$ be a compact subset, $T > 0$, and $\eta \in (0,1)$. A quadratic differential $q \in \qut$ is said to be \textit{$(K,T,\eta)$-recurrent} if its projection to $\qum$ is $(K,T,\eta)$-recurrent as defined in \S5. Let $K' \subseteq \mathcal{Q}^1$ be a compact subset, $T > 0$, $\rho \in (0,1)$, $\epsilon > 0$, and $s > 0$.  A sequence of positive real numbers $\{s_n\}_{n=0}^{N+1}$ with $N \in \mathbf{N}^+$ is said to be a \textit{$(K',T,\rho,\epsilon,s)$-itinerary} of $q \in \qut$ if it is a  $(K',T,\rho,\epsilon,s)$-itinerary of the projection of $q$ to $\qum$ as defined in \S5.

Let $\beta$ be a closed curve on $S_g$. Consider the function $v_\beta \colon \qut \to \mathbf{R}$ defined as follows. Given a marked quadratic differential $q \in \qut$, choose a flat geodesic representative of $\beta$ on $q$. Let $\{\beta_j\}_{j=1}^m$ be the sequence of saddle connections of this representative if it is singular, or a singleton consisting of the representative itself if it is a cylinder curve. For every $j \in \{1,\dots,m\}$ denote by $i(\beta_j,\Im(q))$ the total transverse measure of $\beta_j$ with respect to $\Im(q)$. Then,
\begin{equation}
\label{eq:vertical}
v_\beta(q) := \min_{j=1,\dots,m} \frac{i(\beta_j,\Im(q))}{\ell_{\beta_j}(q)}.
\end{equation}
By Proposition \ref{prop:flat_rep}, this value is independent of the choice of flat geodesic representative of $\beta$ on $q$. Furthermore, the function $v_\beta \colon \qut \to \mathbf{R}$ is continuous on strata of $\qut$.

Recall that, for any $q \in \qut$, $\smash{\ell_{\min}^\dagger}(q) := \min\{1,\ell_{\min}(q)\}$. The following quantitative estimate for geometric intersection numbers of closed curves will play a crucial role in the proof of Theorem \ref{theo:main}.

\begin{theorem}
	\label{theo:prelim_1} 
	Fix a stratum $\mathcal{Q}^1 \subseteq \qum$. For every compact subset $K \subseteq \mathcal{Q}^1$ and every $\eta > 0$ there exist constants $\rho_1 = \rho_1(K,\eta) > 0$, $\epsilon_1 = \epsilon_1(K,\eta) > 0$, and $\kappa_1 = \kappa_1(K,\eta) > 0$ such that for every compact subset $K' \subseteq \mathcal{Q}^1$, every $0 < \rho < \rho_1$, every $0 < \epsilon < \epsilon_1$, and every $s > 0$, the following holds. Let $q_s \in \qut$, $r > 0$, and $q_e := a_r q_s$. Suppose $-q_e$ is $(K,r,\eta)$-recurrent and has a $(K',r,\rho,\epsilon,s)$-itinerary. Then, for any pair of closed curves $\alpha$ and $\beta$ on $S_g$,
	\begin{align*}
	i(\alpha,\beta) &= i(\alpha,\Re(q_s)) \cdot i(\beta,\Im(q_e)) \cdot e^r \\
	&\phantom{=} + O_{K',\epsilon,s}\left(\ell_{\alpha}(q_s) \cdot \ell_{\min}^\dagger(q_s)^{-1} \cdot \ell_{\beta}(q_e) \cdot \ell_{\min}^\dagger(q_e)^{-5} \cdot v_\beta(q_e)^{-1} \cdot e^{(1-\kappa_1)r}\right).
	\end{align*}
\end{theorem}

\begin{proof}
	Fix a stratum $\mathcal{Q}^1 \subseteq \qum$, a compact subset $K \subseteq \mathcal{Q}^1$, and $\eta > 0$. Let $\rho_1 = \rho_1(K,\eta) > 0$, $\epsilon_1 = \epsilon_1(K,\eta) > 0$, and $\kappa_1 = \kappa_1(K,\eta) > 0$ be as in Theorem \ref{theo:equid_3}. Fix $0 < \rho < \rho_1$, $0 < \epsilon < \epsilon_1$, and $s > 0$. Let $q_s \in \mathcal{Q}^1$, $r > 0$, and $q_e := a_r q_s$. Suppose $-q_e$ is $(K,r,\eta)$-recurrent and has a $(K',r,\rho,\epsilon,s)$-itinerary. Let $\alpha$ and $\beta$ be a pair of closed curves on $S_g$. For the rest of this proof we identify $\alpha$ and $\beta$ with fixed choices of flat geodesic representatives on $q_e$. 
	
	Notice that $\beta$ is either a cylinder curve or a concatenation of $m \preceq \ell_{\beta}(q_e)/\ell_{\min}(q_e)$ saddle connections of $q_e$. As $q_e = a_r q_s$ and as the Teichmüller geodesic flow preserves flat geodesics, $\alpha$ is either a cylinder curve or a concatenation of $\preceq \ell_{\alpha}(q_s)/\ell_{\min}(q_s)$ saddle connections of $q_e$. Then, by Proposition \ref{prop:flat_int},
	\begin{equation}
	\label{eq:X1}
	|i(\alpha,\beta) - I(\alpha,\beta)| \preceq \frac{\ell_{\alpha}(q_s) \cdot \ell_{\beta}(q_e)}{\ell_{\min}(q_s) \cdot \ell_{\min}(q_e)}.
	\end{equation}
	
	Let $\{\beta_j\}_{j=1}^m$ be the sequence of saddle connections of $\beta$ or $\{\beta\}$ if $\beta$ is a cylinder curve. As $q_e = a_r q_s$ for $r > 0$, Proposition \ref{prop:rect_int} provides a rectangular decomposition of $\alpha$ with $n \preceq \ell_{\alpha}(q_s)/\ell_{\min}(q_s)$ horizontal segments $\{\gamma_i\}_{i=1}^n$ of lengths $\ell_{\gamma_i}(q_e) \leq e^r \cdot \ell_{\min}(q_s)/2$. The lengths of these segments add up to $\sum_{i=1}^n \ell_{\gamma_i}(q_e) = e^r \cdot i(\Re(q_s),\alpha)$. Furthermore, the following estimate holds
	\begin{equation}
	\label{eq:X2}
	\bigg\vert I(\alpha,\beta) - \sum_{i=1}^n \sum_{j=1}^m \#(\gamma_i \cap \beta_j)\bigg\vert \preceq \frac{\ell_\alpha(q_s) \cdot \ell_{\beta}(q_e)}{\ell_{\min}(q_s) \cdot \ell_{\min}(q_e)}.
	\end{equation}
	
	Putting together (\ref{eq:X1}) and (\ref{eq:X2}) we deduce
	\begin{equation}
	\label{eq:X25}
	\bigg\vert i(\alpha,\beta) - \sum_{i=1}^n \sum_{j=1}^m \#(\gamma_i \cap \beta_j)\bigg\vert \preceq \frac{\ell_\alpha(q_s) \cdot \ell_{\beta}(q_e)}{\ell_{\min}(q_s) \cdot \ell_{\min}(q_e)}.
	\end{equation}
	
	Let $0 < \delta < \smash{\ell^\dagger_{\min}}(q_e)/8$, to be fixed later. By Proposition \ref{prop:bump_4}, there exists a pair of functions $\phi_{0,\delta}, \thinspace \phi_{1,\delta} \in H_{q_e}^1(X)$ satisfying
	\begin{enumerate}
		\item $0 \leq \phi_{0,\delta} \leq \phi_{1,\delta}$,
		\item $ \displaystyle \int_X \phi_{0,\delta} \thinspace dA_{q_e} \leq i(\beta,\Im(q_e)) \leq \int_X \phi_{1,\delta} \thinspace dA_{q_e}$,
		\item $ \displaystyle \int_X (\phi_{1,\delta} -\phi_{0,\delta}) \thinspace dA_{q_e} \preceq \delta \cdot \ell_{\beta}(q_e)$,
		\item $\|\phi_{k,\delta}\|_{1,q_e} \preceq \ell_{\beta}(q_e) \cdot\ell^\dagger_{\min}(q_e)^{-3} \cdot v_\beta(q_e)^{-1} \cdot \delta^{-1}$ for $k \in \{0,1\}$,
	\end{enumerate}
	and such that for every horizontal segment $\gamma$ of $q_e$,
	\begin{gather}
	\int_{\gamma} \phi_{0,\delta} \thinspace d\Re(q) - \sum_{j=1}^m \#(\gamma \cap \beta_j)\preceq \frac{\ell_{\beta}(q_e)}{\ell_{\min}(q_e)}, \label{eq:X3}\\
	\sum_{j=1}^m \#(\gamma \cap \beta_j) - \int_{\gamma} \phi_{1,\delta} \thinspace d\Re(q) \preceq \frac{\ell_{\beta}(q_e)}{\ell_{\min}(q_e)}. \label{eq:X4}
	\end{gather}
	Recall $n \preceq \ell_{\alpha}(q_s)/\ell_{\min}(q_s)$. Adding up (\ref{eq:X3}) and (\ref{eq:X4}) over $\{\gamma_i\}_{i=1}^n$ we deduce
	\begin{gather}
	\sum_{i=1}^n \int_{\gamma_i} \phi_{0,\delta} \thinspace d\Re(q) - \sum_{i=1}^n \sum_{j=1}^m \#(\gamma_i \cap \beta_j)\preceq \frac{\ell_\alpha(q_s) \cdot \ell_{\beta}(q_e)}{\ell_{\min}(q_s) \cdot \ell_{\min}(q_e)}, \label{eq:X5}\\
	\sum_{i=1}^n \sum_{j=1}^m \#(\gamma_i \cap \beta_j) - \sum_{i=1}^n \int_{\gamma_i} \phi_{1,\delta} \thinspace d\Re(q) \preceq \frac{\ell_\alpha(q_s) \cdot \ell_{\beta}(q_e)}{\ell_{\min}(q_s) \cdot \ell_{\min}(q_e)}. \label{eq:X6}
	\end{gather}
	
	Let $X$ be the underlying Riemann surface of $q_e$. As $-q_e$ is $(K,r,\eta)$-recurrent and has a $(K',r,\rho,\epsilon,s)$-itinerary, and as the horizontal segments $\{\gamma_i\}_{i=1}^n$ have lengths $\ell_{\gamma_i}(q_e) \leq e^r \cdot \ell_{\min}(q_s)/2$, Theorem \ref{theo:equid_3} ensures that, for every $k \in \{0,1\}$ and every $i \in \{1,\dots,n\}$,
	\begin{gather}
	\label{eq:X7}
	\bigg\vert \int_{\gamma_i} \phi_{k,\delta} \thinspace d\Re(q) - \ell_{\gamma_i}(q_e) \cdot \int_X \phi_{k,\delta} \thinspace d A_{q_e} \bigg\vert \\
	\preceq_{K',\epsilon,s}  \ell_{\min}(q_s)^+ \cdot \ell_{\min}^\dagger(q_e)^{-1} \cdot \|\phi_{k,\delta}\|_{1,q_e} \cdot e^{(1-\kappa_1)r}. \nonumber
	\end{gather}
	Recall $n \preceq \ell_{\alpha}(q_s)/\ell_{\min}(q_s)$ and $\sum_{i=1}^n \ell_{\gamma_i}(q_e) = e^r \cdot i(\Re(q_s),\alpha)$. From (\ref{eq:X7}) and the triangle inequality we deduce that, for $k \in \{0,1\}$,
	\begin{gather}
	\label{eq:X8}
	\bigg\vert \sum_{i=1}^n \int_{\gamma_i} \phi_{k,\delta} \thinspace d\Re(q) - e^r \cdot  i(\alpha,\Re(q_s)) \cdot \int_X \phi_{k,\delta} \thinspace d A_q \bigg\vert\\
	 \preceq_{K',\epsilon,s} \ell_{\alpha}(q_s) \cdot \ell_{\min}^\dagger(q_s)^{-1}\cdot \ell_{\min}^\dagger(q_e)^{-1} \cdot \|\phi_{k,\delta}\|_{1,q_e} \cdot e^{(1-\kappa_1)r}. \nonumber
	\end{gather}
	From (\ref{eq:X8}) and condition (4) above we deduce that, for $k \in \{0,1\}$,
	\begin{gather}
	\label{eq:X9}
	\bigg\vert \sum_{i=1}^n \int_{\gamma_i} \phi_{k,\delta} \thinspace d\Re(q) - e^r \cdot i(\alpha,\Re(q_s)) \int_X \phi_{k,\delta} \thinspace d A_q \bigg\vert \\
	\preceq_{K',\epsilon,s} \ell_{\alpha}(q_s) \cdot \ell_{\min}^\dagger(q_s)^{-1} \cdot \ell_{\beta}(q_e) \cdot \ell_{\min}^\dagger(q_e)^{-4} \cdot  v_\beta(q_e)^{-1} \cdot \delta^{-1} \cdot e^{(1-\kappa_1)r}. \nonumber
	\end{gather}
	
	Notice $v_\beta(q_e) \leq 1$. Putting together (\ref{eq:X3}) and (\ref{eq:X4}) with (\ref{eq:X9}) yields
	\begin{gather}
	e^r \cdot  i(\alpha,\Re(q_s)) \cdot \int_X \phi_{0,\delta} \thinspace d A_q - \sum_{i=1}^n \sum_{j=1}^m \#(\gamma_i \cap \beta_j)\label{eq:X10} \\ 
	\preceq_{K',\epsilon,s} \ell_{\alpha}(q_s) \cdot \ell_{\min}^\dagger(q_s)^{-1} \cdot \ell_{\beta}(q_e) \cdot \ell_{\min}^\dagger(q_e)^{-4} \cdot  v_\beta(q_e)^{-1} \cdot \delta^{-1} \cdot e^{(1-\kappa_1)r}, \nonumber\\
	\sum_{i=1}^n \sum_{j=1}^m \#(\gamma_i \cap \beta_j) - e^r \cdot i(\alpha,\Re(q_s)) \cdot \int_X \phi_{1,\delta} \thinspace d A_q \label{eq:X11}\\
	\preceq_{K',\epsilon,s} \ell_{\alpha}(q_s) \cdot \ell_{\min}^\dagger(q_s)^{-1} \cdot \ell_{\beta}(q_e) \cdot \ell_{\min}^\dagger(q_e)^{-4} \cdot v_\beta(q_e)^{-1} \cdot \delta^{-1} \cdot e^{(1-\kappa_1)r}. \nonumber
	\end{gather}
	From (\ref{eq:X10}), (\ref{eq:X11}), and condition (2) above we deduce
	\begin{gather}
	\bigg\vert \sum_{i=1}^n \sum_{j=1}^m \#(\gamma_i \cap \beta_j) - e^r \cdot i(\alpha,\Re(q_s) \cdot \thinspace i(\beta, \Im(q_e)) \bigg| \label{eq:X12}\\
	\preceq_{K',\epsilon,s} \ell_{\alpha}(q_s) \cdot \ell_{\min}^\dagger(q_s)^{-1} \cdot \ell_{\beta}(q_e) \cdot \ell_{\min}^\dagger(q_e)^{-4} \cdot  v_\beta(q_e)^{-1} \cdot \delta^{-1} \cdot e^{(1-\kappa_1)r} \nonumber\\
	 + e^r \cdot i(\alpha,\Re(q_s)) \cdot \left(\int_X (\phi_{1,\delta} - \phi_{0,\delta}) \thinspace dA_{q_e} \right). \nonumber
	\end{gather}
	From (\ref{eq:X12}) and condition (3) above we deduce
	\begin{gather}
	\bigg\vert \sum_{i=1}^n \sum_{j=1}^m \#(\gamma_i \cap \beta_j) - e^r \cdot i(\alpha,\Re(q_s)) \cdot i(\beta, \Im(q_e)) \bigg| \label{eq:X13}\\
	\preceq_{K',\epsilon,s} \ell_{\alpha}(q_s) \cdot \ell_{\min}^\dagger(q_s)^{-1} \cdot \ell_{\beta}(q_e) \cdot \ell_{\min}^\dagger(q_e)^{-4} \cdot  v_\beta(q_e)^{-1} \cdot \delta^{-1} \cdot e^{(1-\kappa_1)r} \nonumber\\
	+ e^r \cdot \thinspace i(\alpha,\Re(q_s)) \cdot \delta \cdot \ell_{\beta}(q_e). \nonumber
	\end{gather}
	Notice $i(\alpha,\Re(q_s)) \leq \ell_\alpha(q_s)$. Rearranging terms in (\ref{eq:X13}) yields
	\begin{gather}
	\bigg\vert \sum_{i=1}^n \sum_{j=1}^m \#(\gamma_i \cap \beta_j) - e^r \cdot i(\alpha,\Re(q_s)) \cdot i(\beta, \Im(q_e)) \bigg| \label{eq:X14}\\
	\preceq_{K',\epsilon,s} \ell_{\alpha}(q_s) \cdot \ell_{\min}^\dagger(q_s)^{-1} \cdot \ell_{\beta}(q_e) \cdot \ell_{\min}^\dagger(q_e)^{-4} \cdot  v_\beta(q_e)^{-1} \cdot (\delta^{-1} \cdot e^{(1-\kappa_1)r} + \delta \cdot e^r).\nonumber
	\end{gather}
	
	Putting together (\ref{eq:X25}) and (\ref{eq:X14}) we deduce
	\begin{gather}
	|i(\alpha,\beta) - e^r \cdot i(\alpha,\Re(q_s)) \cdot i(\beta, \Im(q_e)) | \label{eq:X15}\\
	\preceq_{K',\epsilon,s} \ell_{\alpha}(q_s) \cdot \ell_{\min}^\dagger(q_s)^{-1} \cdot \ell_{\beta}(q_e) \cdot \ell_{\min}^\dagger(q_e)^{-4} \cdot v_\beta(q_e)^{-1} \cdot (\delta^{-1} \cdot e^{(1-\kappa_1)r} + \delta \cdot e^r).\nonumber
	\end{gather}
	Letting $\delta = e^{-\kappa_1 r/2} \cdot \thinspace \ell_{\min}^\dagger(q_e)/8 < \ell_{\min}^\dagger(q_e)/8$ we conclude
	\begin{gather*}
	|i(\alpha,\beta) - e^r \cdot i(\alpha,\Re(q_s) \cdot i(\beta, \Im(q_e)) | \\
	\preceq_{K',\epsilon,s} \ell_{\alpha}(q_s) \cdot \ell_{\min}^\dagger(q_s)^{-1} \cdot \ell_{\beta}(q_e) \cdot \ell_{\min}^\dagger(q_e)^{-5} \cdot  v_\beta(q_e)^{-1} \cdot e^{(1-\kappa_1/2)r}.\qedhere
	\end{gather*}
\end{proof}

Theorem \ref{theo:prelim_1} can be upgraded via a density argument to an estimate for geometric intersection numbers of closed curves with respect to arbitrary geodesic currents. To state this upgraded version, we first discuss some aspects of the theory of geodesic currents.

\subsection*{Geodesic currents.} Endow the connected, oriented, closed, genus $g \geq 2$ surface $S_g$ with an arbitrary hyperbolic metric. The projective tangent bundle $PTS_g$ admits a $1$-dimensional foliation by lifts of geodesics on $S_g$. A \textit{geodesic current} on $S_g$ is a Radon transverse measure of the geodesic foliation of $PTS_g$. Equivalently, a geodesic current on $S_g$ is a $\pi_1(S_g)$-invariant Radon measure on the space of unoriented geodesics of the universal cover of $S_g$. Endow the space of geodesic currents on $S_g$ with the weak-$\star$ topology. Different choices of hyperbolic metrics on $S_g$ yield canonically identified spaces of geodesic currents \cite[Fact 1]{Bon88}. Denote the space of geodesic currents on $S_g$ by $\mathcal{C}_g$. This space supports a natural $\mathbf{R}^+$ scaling action and a natural $\mcg$ action \cite[\S 2]{RS19}. The space of projective geodesic currents $\mathcal{PC}_g := \mathcal{C}_g/\mathbf{R}^+$ is compact \cite[Corollary 5]{Bon88}.

Weighted closed curves on $S_g$ embed into $\mathcal{C}_g$ by considering their geodesic representatives with respect to the previously fixed hyperbolic metric. By work of Bonahon \cite[Proposition 2]{Bon88}, this embedding is dense. Moreover, the geometric intersection number of closed curves extends in a unique way to a continuous, symmetric, bilinear pairing $i(\cdot,\cdot)$ on $\mathcal{C}_g$ \cite[Proposition 3]{Bon88}. This pairing is invariant with respect to the diagonal action of $\mcg$ on $\mathcal{C}_g \times \mathcal{C}_g$.

Many different metrics on $S_g$ embed into $\mathcal{C}_g$ in such a way that the geometric intersection number of the metric with any closed curve is equal to the length of the geodesic representatives of the closed curve with respect to the metric. We refer to the geodesic current corresponding to any such metric as its \textit{Liouville current}. Examples of metrics admitting Liouville currents include:
 \renewcommand{\labelitemi}{\textperiodcentered}
\begin{itemize}
	\item Hyperbolic metrics \cite{Bon88},
	\item Negatively curved Riemannian metrics \cite{O90},
	\item Negatively curved Riemannian metrics with cone singularities of angle $\geq 2\pi$ \cite{HP97},
	\item Singular flat metrics induced by quadratic differentials \cite{DLR10},
	\item Singular flat metrics with cone singularities of angle $\geq 2\pi$ \cite{BL18}.
\end{itemize}

\subsection*{An upgraded preliminary quantitative estimate for geometric intersection numbers.} Denote by $q \in \mathcal{C}_g$ the Liouville current of the singular flat metric on $S_g$ induced by $q \in \qut$. The following quantitative estimate for geometric intersection numbers of closed curves with respect to arbitrary geodesic currents is a direct consequence of Theorem \ref{theo:prelim_1} and the fact that weighted closed curves are dense in the space of geodesic currents.

\begin{theorem}
	\label{theo:prelim_2} 
	Fix a stratum $\mathcal{Q}^1 \subseteq \qum$. For every compact subset $K \subseteq \mathcal{Q}^1$ and every $\eta > 0$ there exist constants $\rho_1 = \rho_1(K,\eta) > 0$, $\epsilon_1 = \epsilon_1(K,\eta) > 0$, and $\kappa_1 = \kappa_1(K,\eta) > 0$ such that for every compact subset $K' \subseteq \mathcal{Q}^1$, every $0 < \rho < \rho_1$, every $0 < \epsilon < \epsilon_1$, and every $s > 0$, the following holds. Let $q_s \in \qut$, $r > 0$, and $q_e := a_r q_s$. Suppose $-q_e$ is $(K,r,\eta)$-recurrent and has a $(K',r,\rho,\epsilon,s)$-itinerary. Then, for any geodesic current $\alpha$ and any closed curve $\beta$ on $S_g$,
	\begin{align*}
	i(\alpha,\beta) &= i(\alpha,\Re(q_s)) \cdot i(\beta,\Im(q_e)) \cdot e^r \\
	&\phantom{=}+ O_{K',\epsilon,s}\left(i(\alpha,q_s) \cdot \ell_{\min}^\dagger(q_s)^{-1} \cdot \ell_{\beta}(q_e) \cdot \ell_{\min}^\dagger(q_e)^{-5} \cdot v_\beta(q_e)^{-1} \cdot e^{(1-\kappa_1)r}\right).
	\end{align*}
\end{theorem}

\section{Recurrence estimates}

\subsection*{Outline of this section.} In this section we show that most Teichmüller geodesic segments joining a point in Teichmüller space to points in a mapping class group orbit of Teichmüller space satisfy the recurrence conditions needed to apply Theorem \ref{theo:prelim_2}. See Theorems  \ref{theo:thin_traj} and \ref{theo:rec_1}. The estimates of Eskin, Mirzakhani, and Rafi on the number of thin Teichmüller geodesic segments joining a point in Teichmüller space to points in a mapping class group orbit of Teichmüller space \cite{EMR19} play a crucial role in the proofs of this section.

\subsection*{Thin Teichmüller geodesic segments.} For the rest of this section we fix an integer $g \geq 2$ and a connected, oriented, closed surface $S_g$ of genus $g$. Let us recall some of the notation introduced in \S 1. Recall that $\tt$ denotes the Teichmüller space of marked complex structures on $S_g$ and that $d_\mathcal{T}$ denotes the Teichmüller metric on $\mathcal{T}_g$. Recall that $\mathrm{Mod}_g$ denotes the mapping class group of $S_g$ and that this group acts properly discontinuously on $\tt$ by changing the markings.

Recall that $\qut$ denotes the Teichmüller space of marked, unit area quadratic differentials on $S_g$ and that $\pi \colon \qut \to \tt$ denotes the natural projection to $\mathcal{T}_g$. The Teichmüller metric on $\mathcal{T}_g$ is complete and its unit speed geodesics are precisely the projection to $\mathcal{T}_g$ of the Teichmüller geodesic flow orbits on $\qut$. Recall that $S(X) := \pi^{-1}(X)$ for every $X \in \mathcal{T}_g$ and that $q_s, q_e \colon \mathcal{T}_g \times \mathcal{T}_g - \Delta \to \mathcal{Q}^1\mathcal{T}_g$ denote the maps which to every pair of distinct points $X,Y \in \mathcal{T}_g$ assign the quadratic differentials $q_s(X,Y) \in S(X)$ and $q_e(X,Y) \in S(Y)$ corresponding to the tangent directions at $X$ and $Y$ of the unique Teichmüller geodesic segment from $X$ to $Y$.

Recall that $\qum$ denotes the moduli space of unit area quadratic differentials on $S_g$ and that $\{a_t\}_{t \in \mathbf{R}}$ denotes the Teichmüller geodesic flow on $\qum$. Denote by $p \colon \qut \to \qum$ the natural forgetful map. Fix a stratum $\mathcal{Q}^1 \subseteq \qum$. Let $K \subseteq \mathcal{Q}^1$ compact, $T > 0$, and $\eta \in (0,1)$. Recall that $q \in \qut$ is said to be $(K,T,\eta)$-recurrent if
\[
|\{t \in [0,T] \colon a_t \thinspace p(q) \notin K\}| \leq \eta T.
\]

Fix a stratum $\mathcal{Q}^1 \subseteq \qum$. Let $X,Y \in \mathcal{T}_g$, $R > 0$, $K \subseteq \mathcal{Q}^1$ compact, and $\eta \in (0,1)$. Denote by $N(X,Y,R,K,\eta) \subseteq \mcg$ the set of all mapping classes $\mc \in \mcg$ such that $\mc.Y =X$ or $\mc.Y \neq X$, $r:=d_{\mathcal{T}}(X,\mc.Y) \leq R$, and $-q_e(X,\mc.Y)$ is not $(K,r,\eta)$-recurrent. 

The \textit{principal stratum} of $\qum$ is the stratum of unit area quadratic differentials on $S_g$ all of whose zeroes are simple. This stratum is the unique top dimensional stratum of $\qum$. It is an open subset of $\qum$ and its complex dimension is $6g-6$.

The following theorem of Eskin, Mirzakhani, and Rafi shows that most Teichmüller geodesic segments from a point in Teichmüller space to points in a mapping class group orbit of Teichmüller space spend an arbitrarily large fraction of their time in a compact subset of the principal stratum; compare to the estimate in (\ref{eq:max_growth}). This result will play a crucial role in our proof of Theorem \ref{theo:main}. 

\begin{theorem} \cite[Theorem 1.7 and Lemma 3.9]{EMR19}
	\label{theo:thin_traj}
	Let $\mathcal{Q}^1 \subseteq \qum$ be the principal stratum. Then, for every $\eta > 0$ there exists a compact subset $K = K(\eta) \subseteq \mathcal{Q}^1$ such that for every compact subset $\mathcal{K} \subseteq \mathcal{T}_g$, every $X,Y \in \mathcal{K}$, and every $R > 0$,
	\[
	|N(X,Y,R,K,\eta)| \preceq_{\mathcal{K},\eta} e^{(6g-6-\eta/2)R}.
	\]
\end{theorem}

\subsection*{Sampling Teichmüller geodesic flow orbits.} Fix a stratum $\mathcal{Q}^1 \subseteq \qum$. Let $K' \subseteq \mathcal{Q}^1$ be a compact subset, $T > 0$, $\rho \in (0,1)$, $\epsilon > 0$, and $s > 0$. Recall that a sequence of positive real numbers $\{s_n\}_{n=0}^{N+1}$ with $N \in \mathbf{N}^+$ is said to be a $(K',T,\rho,\epsilon,s)$-itinerary of $q \in \qum$ if 
\begin{enumerate}
	\item $a_{s_n} q \in K'$ for every $n \in \{1, \dots, N\}$,
	\item $s_0 = \rho T$,
	\item $s_N < T$, 
	\item $s_{N+1} = T$,
	\item $s_{n+1} \geq s_n + s$ for every $n \in \{1,\dots, N-1\}$,
	\item $s_{n+1} - s_n \leq \epsilon s_n$ for every $n \in \{0,\dots,N\}$.
\end{enumerate}
Finding such itineraries will be crucial for our proof of Theorem \ref{theo:main}. We now describe a sampling procedure that fulfils this objective. This procedure is inspired by methods introduced in \cite[\S5.2]{AF08}.

Fix a stratum $\mathcal{Q}^1 \subseteq \qum$. Let $K \subseteq \mathcal{Q}^1$ be a compact subset, $T > 0$, $\rho \in (0,1)$, and $s > 0$. Denote by $K' \subseteq \mathcal{Q}^1$ the compact enlargement of $K$ given by
\begin{equation}
\label{eq:enlarge}
K' := \bigcup_{t \in [-s,s]} a_t K.
\end{equation}
Let $q \in \qum$. Fix $\sh_0 :=  \rho T$. Starting from $n = 0$, define $\sh_{n+1}$ inductively as follows: if $a_t q \in K'$ for every $t \in [\sh_n,\sh_n+s]$ then $\sh_{n+1} := \sh_n + s$, otherwise
\[
\sh_{n+1} := \inf\{ t > \sh_n + s \ | \ a_t q \in K\}.
\]
Notice that, by definition, $\sh_{n+1} \geq \sh_n + s$. In particular, there exists $N \in \mathbf{N}$ such that $\sh_N < T$ and $\sh_{N+1} \geq T$, where we allow $\sh_{N+1} = +\infty$. Once this step is reached, stop the inductive definition of $\sh_{N+2}$. Consider the sequence $\{s_n\}_{n=0}^{N+1}$ of positive real numbers given by $s_n := \sh_n$ if $n \in \{0,\dots,N\}$ and $s_{N+1} := T$. By construction, $a_{s_n} q \in K'$ for every $n \in \{1,\dots,N\}$. Furthermore, conditions (2), (3), (4), and (5) above are directly satisfied. To ensure $\{s_n\}_{n=0}^{N+1}$ is a $(K',T,\rho,\epsilon,s)$-itinerary of $q$, it remains to check that $N > 0$ and that condition (6) above holds for suitable choices of $\epsilon$. 

More specifically, we show that if condition (6) above fails to hold for some $\epsilon > 0$, then the Teichmüller geodesic flow orbit $\{a_t q\}_{t \geq 0}$ has a large excursion outside of $K$. Theorem \ref{theo:thin_traj} ensures that, in the context of the proof of Theorem \ref{theo:main}, this is a very unlikely event.

\begin{proposition}
	\label{theo:bad_sampling}
	Fix a stratum $\mathcal{Q}^1 \subseteq \qum$. Let $K \subseteq \mathcal{Q}^1$ be a compact subset, $T > 0$, $\rho\in (0,1)$, and $s > 0$ be such that $\rho T \geq s$.  Fix $q \in \qum$. Denote by $\{s_n\}_{n=0}^{N+1}$ with $N \in \mathbf{N}$ the sequence of positive real numbers produced by the sampling procedure above. Let $\epsilon> 0$ be such that $\rho (1+\epsilon) < 1$ and $T \geq s/\rho\epsilon$. Suppose $\{s_n\}_{n=0}^{N+1}$ is not a $(K',T,\rho,\epsilon,s)$-itinerary of $q$. Then,
	\[
	|\{t \in [0,T] \colon a_t q \notin K\}| > \rho\epsilon T.
	\]
\end{proposition}

\begin{proof}
	By the discussion above, the only way $\{s_n\}_{n=0}^{N+1}$ can fail to be a $(K',T,\rho,\epsilon,s)$-itinerary of $q$ is if $N = 0$ or condition (6) above fails to hold. As $\rho(1+\epsilon) < 1$, if condition (6) above holds, then $N > 0$. Thus, for the rest of this proof we will assume
	\begin{equation}
	\label{eq:exc0}
	\exists n \in \{0,\dots,N\} \colon s_{n+1} - s_n > \epsilon s_n.
	\end{equation}
	
	Let us introduce some notation that will help us keep track of the excursions of the forward orbit $\{a_t q\}_{t \geq 0}$ outside of $K$. Fix $t_0 := \rho T$. Starting from $k = 1$, inductively define
	\begin{align*}
	t_{2k-1} &:= \inf\{t > t_{2k-2} \ | \ a_t q \notin K'\},\\
	t_{2k} &:= \inf \{t > t_{2k-1} \ | \ a_t q \in K\}.
	\end{align*}
	If $t_k = + \infty$ for some $k \in \mathbf{N}^+$, define $t_j := +\infty$ for every $j >k$. Notice that, by the definition of the times $\{t_k\}_{k \in \mathbf{N}}$ and of the compact enlargement $K' \supseteq K$, for every $k \in \mathbf{N}^+$,
	\begin{equation}
	\label{eq:exc}
	a_t q \notin K, \ \forall t \in (t_{2k-1}-s,t_{2k}).
	\end{equation}
	 The condition $\rho T \geq s$ ensures that $t_{2k-1} - s \geq 0$ for every $k \in \mathbf{N}^+$. 
	 
	 To take advantage of (\ref{eq:exc}), we relate the times $\{t_k\}_{k \in \mathbf{N}}$ to the times $\{s_n\}_{n =0}^{N+1}$ as follows. We claim that, for every $n \in \{0,\dots,N\}$ such that $\sh_{n+1} \neq \sh_n + s$, 
	\begin{equation}
	\label{eq:exc2}
	\exists k \in \mathbf{N}^+ \colon t_{2k-1} < \sh_n + s \text{ and } \sh_{n+1} = t_{2k}.
	\end{equation}
	Indeed, suppose  $n \in \{0,\dots,N\}$ satisfies $\sh_{n+1} \neq \sh_n + s$. Let $k \in \smash{\mathbf{N}^+}$ be the largest positive integer such that $t_{2k-2} \leq \sh_n$. We claim that $t_{2k-1} < \sh_n + s$. Indeed, suppose $t_{2k-1} \geq \sh_n + s$. Then, by the definiton of $t_{2k-1}$, $a_t q \in K'$ for every $t \in [t_{2k-2}, t_{2k-1}] \supseteq [\sh_n,\sh_n + s]$. This forces $\sh_{n+1} = \sh_n + s$. We also claim that $t_{2k} > \sh_n + s$. Indeed, suppose $t_{2k} \leq \sh_n + s$. The maximality of $k \in \mathbf{N}^+$ implies $t_{2k} \geq \sh_n$. By the defintion of $t_{2k}$, $a_{t_{2k}}q \in K$. It follows that, by the definition of $K'$, $a_t q \in K'$ for every $t \in [t_{2k}-s,t_{2k}+s] \supseteq [\sh_n,\sh_n+s]$.  This forces $\sh_{n+1} = \sh_n + s$.  We have thus proved that $t_{2k-1} < \sh_n + s$ and $t_{2k} > \sh_n + s$. These conditions together with the assumption $\sh_{n+1} \neq \sh_n + s$ force $\sh_{n+1} = t_{2k}$.
	
	We now use (\ref{eq:exc0}), (\ref{eq:exc}), and (\ref{eq:exc2}) to show the forward orbit $\{a_t q\}_{t \geq 0}$ has a large excursion outside of $K$. By (\ref{eq:exc0}), there exists $n \in \{0,\dots,N\}$ such that $s_{n+1} - s_n > \epsilon s_n$. By construction, $s_n \geq \rho T$. It follows that $s_{n+1} - s_n > \rho \epsilon T$. By definition, $s_n = \sh_n$ and $s_{n+1} \leq \sh_{n+1}$. In particular, $\sh_{n+1} - \sh_n > \rho \epsilon T$. As $T \geq s/\rho\epsilon$, this implies $\sh_{n+1} - \sh_n > s$. From (\ref{eq:exc2}) we deduce there exists $k \in \mathbf{N}^+$ such that $t_{2k-1} < \sh_n + s$ and $\sh_{n+1} = t_{2k}$. If $n < N$ then $t_{2k} - (t_{2k-1}-s) > \sh_{n+1} - \sh_n > \rho\epsilon T$ and $[t_{2k-1} - s, t_{2k}] \subseteq [0,s_{n+1}] \subseteq [0,T]$.  If $n = N$ then $t_{2k} = \sh_{N+1} \geq T$, $T - (t_{2k-1} -s) > s_{N+1} - s_N > \rho\epsilon T$, and $[t_{2k-1}-s,T] \subseteq [0,T]$. Applying (\ref{eq:exc}) in either case shows that
	\[
	|\{t \in [0,T] \colon a_t q \notin K\}| > \rho\epsilon T. \qedhere
	\]
\end{proof}

Using Theorem \ref{theo:thin_traj} and Proposition \ref{theo:bad_sampling} we show that the sampling procedure above yields suitable itineraries in most cases of interest. Fix a stratum $\mathcal{Q}^1 \subseteq \qum$. Let $X,Y \in \mathcal{T}_g$, $R > 0$, $K' \subseteq \mathcal{Q}^1$ compact, $\rho \in (0,1)$, $\epsilon > 0$, and $s > 0$. Denote by $N'(X,Y,R,K',\rho,\epsilon,s) \subseteq \mcg$ the set of all mapping classes $\mc \in \mcg$ such that $\mc.Y = X$ or $\mc.Y \neq X$, $r:=d_\mathcal{T}(X,\mc.Y) \leq R$, and $-p(q_e(X,\mc.Y))$ has no $(K',r,\rho,\epsilon,s)$-itineraries. The following result ensures that few mapping classes belong to this set if $K' \subseteq \mathcal{Q}^1$ is a large enough compact subset of the principal stratum; compare to the estimate in (\ref{eq:max_growth}). This result will play a crucial role in our proof of Theorem \ref{theo:main}. 

\begin{theorem}
	\label{theo:rec_1}
	Let $\mathcal{Q}^1 \subseteq \qum$ be the principal stratum. Then, for every $\rho \in (0,1)$, every $\epsilon > 0$ such that $\rho(1+\epsilon)<1$, and every $s > 0$, there exists a compact subset $K' = K'(\rho,\epsilon,s) \subseteq \mathcal{Q}^1$ and a constant $\kappa_2 = \kappa_2(\rho,\epsilon) > 0$ such that for every $\mathcal{K} \subseteq \mathcal{T}_g$ compact, every $X,Y \in \mathcal{K}$, and every $R > 0$,
	\[
	|N'(X,Y,R,K',\rho,\epsilon,s)| \preceq_{\mathcal{K},\rho,\epsilon,s} e^{(6g-6-\kappa_2)R}.
	\]
\end{theorem}

\begin{proof}
	Let $\rho \in (0,1)$, $\epsilon > 0$ such that $\rho(1+\epsilon)<1$, and $s > 0$. Fix $\eta = \eta(\rho,\epsilon) > 0$ satisfying $\eta < \rho\epsilon$. Let $K = K(\eta) \subseteq \mathcal{Q}^1$ be as in Theorem \ref{theo:thin_traj} and $K' = K'(\rho,\epsilon,s) \subseteq \mathcal{Q}^1$ be the compact enlargement of $K$ defined in (\ref{eq:enlarge}). Fix $\mathcal{K} \subseteq \mathcal{T}_g$ compact, $X,Y \in \mathcal{K}$, and $R > 0$. Let $\mc \in N'(X,Y,R,K',\rho,\epsilon,s)$. Suppose $\mc.Y \neq X$. Denote $r := d_\mathcal{T}(X,\mc.Y) \leq R$. By definition, $-p(q_e(X,\mc.Y))$ has no $(K',r,\rho,\epsilon,s)$-itineraties. Thus, under the assumption $r \geq \max\{s/\rho,s/\rho\epsilon\}$, Theorem \ref{theo:bad_sampling} ensures 
	\[
	|\{t \in [0,r] \colon a_t p(-q_e(X,\mc.Y)) \notin K\}| \geq \rho \epsilon r > \eta r.
	\]
	By definition, this implies $\mc \in N(X,Y,R,K,\eta)$. Denote by $E = E(\mathcal{K},\rho,\epsilon,s) \subseteq \mcg$ the finite subset of all mapping classes $\mc \in \mcg$ such that the Teichmüller distance between $\mc.\mathcal{K}$ and $\mathcal{K}$ is $\leq \max\{s/\rho,s/\rho\epsilon\}$. It follows that
	\[
	N'(X,Y,R,K',\rho,\epsilon,s) \subseteq N(X,Y,R,K,\eta) \cup E.
	\]
	Using Theorem \ref{theo:thin_traj} we conclude
	\[
	|N'(X,Y,R,K',\rho,\epsilon,s)| \leq |N(X,Y,R,K,\eta)| + |E| \preceq_{\mathcal{K},\rho,\epsilon,s} e^{(6g-6-\eta/2)R}. \qedhere
	\]
\end{proof}

\section{Controlling the error terms}

\subsection*{Outline of this section.} A crucial step in our proof of Theorem \ref{theo:main} will be to control the error terms obtained when applying Theorem \ref{theo:prelim_2}. In this section we explain how to control these error terms. See Theorems \ref{theo:sparse_1} and \ref{theo:sparse_4}. The proofs of these results are based on estimates established in \cite{Ara20b} for countings of mapping class group orbits in thin sectors of Teichmüller space. These estimates in turn rely on the metric hyperbolicity properties of the Teichmüller geodesic flow proved by Athreya, Bufetov, Eskin, and Mirzakhani in \cite{ABEM12}.

\subsection*{Thin sectors of Teichmüller space.} For the rest of this section we fix an integer $g \geq 2$ and a connected, oriented, closed surface $S_g$ of genus $g$. Recall that $\tt$ denotes the Teichmüller space of marked complex structures on $S_g$, that $d_\mathcal{T}$ denotes the Teichmüller metric on $\mathcal{T}_g$, and that $B_R(X) \subseteq \mathcal{T}_g$ denotes the closed Teichmüller metric ball of radius $R > 0$ center at $X \in \mathcal{T}_g$. Recall that $\mcg$ denotes the mapping class group of $S_g$, that this group acts properly discontinuously on $\tt$ by changing the markings, and that this action preserves the Teichmüller metric.

Recall that $\qut$ denotes the Teichmüller space of marked, unit area quadratic differentials on $S_g$. The forgetful map $\pi \colon \qut \to \tt$ makes $\qut$ into a smooth sphere bundle over $\tt$. This bundle supports a natural smooth measure $\mu$ called the \textit{Masur-Veech measure}. See \cite[\S2]{Ara20b} for several equivalent definitions. This measure is invariant under the natural marking changing action of $\mcg$ on $\qut$. The pushforward $\mathbf{m} := \pi_* \mu$ is a smooth $\mcg$-invariant measure on $\mathcal{T}_g$.

Recall that $S(X) := \pi^{-1}(X)$ for every $X \in \mathcal{T}_g$ and that $q_s, q_e \colon \mathcal{T}_g \times \mathcal{T}_g - \Delta \to \mathcal{Q}^1\mathcal{T}_g$ denote the maps which to every pair of distinct points $X,Y \in \mathcal{T}_g$ assign the quadratic differentials $q_s(X,Y) \in S(X)$ and $q_e(X,Y) \in S(Y)$ corresponding to the tangent directions at $X$ and $Y$ of the unique Teichmüller geodesic segment from $X$ to $Y$. For every $X \in \mathcal{T}_g$ and every $V \subseteq S(X)$ consider the sector 
\[
\text{Sect}_V(X) := \{X\} \cup \{Y \in \mathcal{T}_g \setminus \{X\} \ | \ q_s(X,Y) \in V \}.
\] 

Recall that $\qum$ denotes the moduli space of unit area quadratic differentials on $S_g$. For the rest of this section, denote by $\mathcal{Q}^1 \subseteq \qum$ the principal stratum of $\qum$. As in (\ref{eq:compact}), for every $\delta > 0$, denote by $K_\delta \subseteq \mathcal{Q}^1$ the compact subset of $\mathcal{Q}^1$ given by
\[
K_\delta := \{q \in \mathcal{Q}^1 \ | \ \ell_{\min}(q) \geq \delta\}.
\]

Recall that $p \colon \qut \to \qum$ denotes the natural forgetful map. Given $A \subseteq \mathcal{T}_g$ and $r > 0$, denote by $\mathrm{Nbhd}_r(A) \subseteq \mathcal{T}_g$ the set of points in $\mathcal{T}_g$ at Teichmüller distance at most $r$ from $A$. The following estimate was proved in \cite{Ara20b}. The proof of this estimate uses the metric hyperbolicity properties of the Teichmüller geodesic flow established in \cite{ABEM12}.

\begin{theorem}
	\cite[Theorem 7.16]{Ara20b}
	\label{theo:thin_sector_meas}
	There exist constants $r_0 = r_0(g) > 0$ and $\kappa_3 = \kappa_3(g) > 0$ such that the following holds. Let $\mathcal{K} \subseteq \mathcal{T}_g$ be a compact subset, $X \in \mathcal{K}$,  and $\delta > 0$. Denote $V_\delta := S(X) \setminus p^{-1}(K_\delta)$. Then, for every $0 < r < r_0$ and every $R > 0$,
	\[
	\mathbf{m}(\mathrm{Nbhd}_{r}(B_R(X) \cap \mathrm{Sect}_{V_\delta}(X) \cap \mathrm{Mod}_g \cdot \mathcal{K}))
	\preceq_{\mathcal{K}} \delta \cdot e^{(6g-6)R} + e^{(6g-6 - \kappa_3) R}.
	\]
\end{theorem}

\subsection*{Counting in thin sectors of Teichmüller space.} The following lemma will allow us to deduce an estimate for countings of mapping class group orbits in thin sectors of Teichmüller space from Theorem \ref{theo:thin_sector_meas}. It is a direct consequence of the proper discontinuity of the mapping class group action on Teichmüller space and of the invariance of the Teichmüller metric with respect to this action.

\begin{lemma}
	\label{lem:embedded_balls}
	Let $\mathcal{K} \subseteq \mathcal{T}_g$ be a compact subset. Then, for every $Y \in \mcg \cdot \mathcal{K}$ and every $0 < r < 1$,
	\[
	|\{\mc \in \mcg \ | \ B_{r}(Y) \cap B_{r}(\mc.Y) \neq \emptyset\}| \preceq_{\mathcal{K}} 1.
	\]
\end{lemma}

Let $X,Y \in \mathcal{T}_g$ and $V \subseteq S(X)$. For every $R > 0$ consider the set of mapping classes
\[
F(X,Y,R,V) := \{\mc \in \mcg \ | \ \mc.Y \in B_R(X) \cap \mathrm{Sect}_V(X)\}.
\]
Directly from Theorem \ref{theo:thin_sector_meas}, Lemma \ref{lem:embedded_balls}, and the mapping class group invariance of the Teichmüller metric and the measure $\mathbf{m}$, we deduce the following result. 

\begin{theorem}
	\label{theo:thin_count_1}
	There exists a constant $\kappa_3 = \kappa_3(g) > 0$ such that the following holds. Let $\mathcal{K} \subseteq \mathcal{T}_g$ be a compact subset, $X,Y \in \mathcal{K}$, and $\delta > 0$. Denote $V_\delta := S(X) \setminus p^{-1}(K_\delta)$. Then, for every $R > 0$,
	\[
	|F(X,Y,R,V_\delta)| \preceq_\mathcal{K} \delta \cdot e^{(6g-6)R} + e^{(6g-6 - \kappa_3) R}.
	\]
\end{theorem}

\subsection*{Sparse subsets of thin mapping classes.} Let $X,Y \in \mathcal{T}_g$, $C > 0$, and $\kappa > 0$. Recall that a subset of mapping classes $M \subseteq \mcg$ is said to be $(X,Y,C,\kappa)$-sparse if
\[
|\{\mc \in M \ | \ d_\mathcal{T}(X,\mc.Y) \leq R \}| \leq C \cdot e^{(6g-6-\kappa)R}.
\]
The following consequence of  Theorem \ref{theo:thin_count_1} will play a crucial role in our proof of Theorem \ref{theo:main}. 

\begin{theorem}
	\label{theo:sparse_1}
	There exists a constant $\kappa_3 = \kappa_3(g) > 0$ such that for every compact subset $\mathcal{K} \subseteq \mathcal{T}_g$ there exists a constant $C  = C(\mathcal{K}) > 0$ with the following property. For every $X,Y \in \mathcal{K}$ and every $0 < \kappa < \kappa_3$, there exists an $(X,Y,C,\kappa)$-sparse subset of mapping classes $M = M(X,Y,\kappa) \subseteq \mcg$ such that for every $\mc \in \mcg \setminus M$, if $r := d_\mathcal{T}(X,\mc.Y)$ and $q_s := q_s(X,\mc.Y)$, then
	\[
	\ell_{\min}(q_s)  \geq e^{-\kappa r}.
	\]
\end{theorem}

\begin{proof}
	Let $\kappa_3 := \kappa_3(g) > 0$ be as in Theorem \ref{theo:thin_count_1} and $\mathcal{K} \subseteq \mathcal{T}_g$ be a compact subset. Fix $X,Y \in \mathcal{K}$ and $0 < \kappa <\kappa_3$. For every $n \in \mathbf{N}$ denote
	\begin{gather*}
	\delta_n := e^{-\kappa n},\\
	V_n := S(X) \setminus p^{-1}(K_{\delta_n}) ,\\
	F_n:= \{\mc \in \mcg \ | \ d_\mathcal{T}(X,\mc.Y) \in [n,n+1), \ \mc.Y \in \mathrm{Sect}_{V_n}(X)\}.
	\end{gather*}
	Consider the subset of mapping classes $M = M(X,Y,\kappa) \subseteq \mcg$ given by
	\[
	M := \bigcup_{n \in \mathbf{N}} F_n.
	\]
	
	We claim $M \subseteq \mcg $ is $(X,Y,C,\kappa)$-sparse for some constant $C  = C(\mathcal{K}) > 0$. Given $R > 0$ denote
	\[
	M_R := \{\mc \in M \ | \ d_\mathcal{T}(X,\mc.Y) \leq R\}.
	\]
	To prove the claim we need to show that, for every $R > 0$,
	\[
	|M_R| \preceq_{\mathcal{K}}  e^{(6g-6-\kappa)R}.
	\]
	Notice that, for every $R > 0$,
	\[
	M_R \subseteq \bigcup_{n=0}^{\lfloor R \rfloor} F_n.
	\]
	Recall $0 < \kappa <\kappa_3$. By Theorem \ref{theo:thin_count_1}, for every $n \in \mathbf{N}$,
	\[
	|F_n| \leq |F(X,Y,n+1,V_n)|  \preceq_{\mathcal{K}} \delta_n \cdot e^{(6g-6)(n+1)} + e^{(6g-6 - \kappa_3)(n+1)} \preceq_g e^{(6g-6-\kappa)n}.
	\]
	It follows that, for every $R > 0$,
	\[
	|M_R| \leq \sum_{n =0}^{\lfloor R \rfloor} |F_n| \preceq_\mathcal{K} \sum_{n =0}^{\lfloor R \rfloor}  e^{(6g-6-\kappa)n} \preceq_g e^{(6g-6-\kappa)R}.
	\]
	
	Now let $\mc \in \mcg \setminus M$. Denote $r := d_\mathcal{T}(X,\mc.Y)$ and $q_s := q_s(X,\mc.Y)$. Let $n \in \mathbf{N}$ be the unique non-negative integer such that $r \in [n,n+1)$. By the definition of $M$, we have $\mc \notin F_n$. Thus, by the definition of $F_n$, we must have $\ell_{\min}(q_s) \geq \delta_n = e^{-\kappa n} \geq e^{-\kappa r}$.
\end{proof}

\subsection*{Horizontally thin sectors of Teichmüller space.} To prove Theorem \ref{theo:main} we will need a more refined version of Theorem \ref{theo:thin_sector_meas}. To state this version, we first introduce some terminology.

Let $\beta$ be a closed curve on $S_g$. Consider the function $h_\beta \colon \qut \to \mathbf{R}$ defined as follows. Given $q \in \qut$, choose a flat geodesic representative of $\beta$ on $q$. Let $\{\beta_j\}_{j=1}^m$ be the sequence of saddle connections of this representative if it is singular, or a singleton consisting of the representative itself if it is a cylinder curve. For every $j \in \{1,\dots,m\}$ denote by $i(\beta_j,\Re(q))$ the total transverse measure of $\beta_j$ with respect to the singular measured foliation $\Re(q)$. Then,
\begin{equation}
\label{eq:hbeta}
h_\beta(q) := \min_{j=1,\dots,m} \frac{i(\beta_j,\Re(q))}{\ell_{\beta_j}(q)}.
\end{equation}
By Proposition \ref{prop:flat_rep}, this value is independent of the choice of flat geodesic representative of $\beta$ on $q$. Furthermore, the function $h_\beta \colon \qut \to \mathbf{R}$ is continuous on strata of $\qut$.

For every $\delta > 0$ consider the subset $H_{\beta,\delta} \subseteq \qut$ given by
\begin{equation}
\label{eq:Hbeta}
H_{\beta,\delta} := \{ q \in \qut \ | \ p(q) \in K_\delta, \ h_\beta(q) \geq \delta \}.
\end{equation}
The following refinement of Theorem \ref{theo:thin_sector_meas} can be proved using the methods introduced in the proof of \cite[Theorem 7.16]{Ara20b}. We refer the reader to Appendix A for a detailed proof.

\begin{theorem}
	\label{theo:slope_sector_meas}
	There exist constants $r_0 = r_0(g) > 0$ and $\kappa_4 = \kappa_4(g) > 0$ such that the following holds. Let $\mathcal{K} \subseteq \mathcal{T}_g$ be a compact subset, $X \in \mathcal{K}$, $\beta$ be a closed curve on $S_g$, and $\delta \in (0,1)$. Denote $V_{\beta,\delta} := S(X) \setminus H_{\beta,\delta}$. Then, for every $0 < r < r_0$ and every $R > 0$,
	\[
	\mathbf{m}(\mathrm{Nbhd}_{r}(B_R(X) \cap \mathrm{Sect}_{V_{\beta,\delta}}(X) \cap \mathrm{Mod}_g \cdot \mathcal{K}))
	\preceq_{\mathcal{K},\beta} \delta \cdot e^{(6g-6)R} + \delta^{-1} \cdot e^{(6g-6 - \kappa_4) R}.
	\]
\end{theorem}

\subsection*{Counting in horizontally thin sectors of Teichmüller space.} Directly from Theorem \ref{theo:slope_sector_meas}, Lemma \ref{lem:embedded_balls}, and the mapping class group invariance of the Teichmüller metric and the measure $\mathbf{m}$, we deduce the following result. 

\begin{theorem}
	\label{theo:slope_count}
	There exists $\kappa_4 = \kappa_4(g) > 0$ such that the following holds. Let $\mathcal{K} \subseteq \mathcal{T}_g$ compact, $X,Y \in \mathcal{K}$, $\beta$ a closed curve on $S_g$, and $\delta \in (0,1)$. Denote $V_{\beta,\delta} := S(X) \setminus H_{\beta,\delta}$. Then, for $R > 0$,
	\[
	|F(X,Y,R,V_{\beta,\delta})| \preceq_\mathcal{K,\beta} \delta \cdot e^{(6g-6)R} + \delta^{-1} \cdot e^{(6g-6 - \kappa_4) R}.
	\]
\end{theorem}

\subsection*{Sparse subsets of horizontally thin mapping classes.} A slightly more refined application of the methods used in the proof of Theorem \ref{theo:sparse_1} yields the following consequence of  Theorem \ref{theo:slope_count}.

\begin{theorem}
	\label{theo:sparse_3}
	There exists a constant $\kappa_4 = \kappa_4(g) > 0$ such that for every compact subset $\mathcal{K} \subseteq \mathcal{T}_g$ and every closed curve $\beta$ on $S_g$, there exists a constant $C  = C(\mathcal{K},\beta) > 0$ with the following property. For every $X,Y \in \mathcal{K}$ and every $0 < \kappa < \kappa_4$, there exists an $(X,Y,C,\kappa)$-sparse subset $M = M(X,Y,\beta,\kappa) \subseteq \mcg$ such that for every $\mc \in \mcg \setminus M$, if $r := d_\mathcal{T}(X,\mc.Y)$ and $q_s := q_s(X,\mc.Y)$, then
	\[
	\ell_{\min}(q_s) \geq e^{-\kappa r}, \quad h_\beta(q_s) \geq e^{-\kappa r}.
	\]
\end{theorem}

\begin{proof}
	Let $\kappa_4 := \kappa_4(g) > 0$ be as in Theorem \ref{theo:slope_count} and $\mathcal{K} \subseteq \mathcal{T}_g$ be a compact subset. Fix $X,Y \in \mathcal{K}$ and $0 < \kappa <\kappa_4/2$. For every $n \in \mathbf{N}$ denote
	\begin{gather*}
	\delta_n := e^{-\kappa n},\\
	V_n := S(X) \setminus H_{\beta,\delta_n} ,\\
	F_n:= \{\mc \in \mcg \ | \ d_\mathcal{T}(X,\mc.Y) \in [n,n+1), \ \mc.Y \in \mathrm{Sect}_{V_n}(X)\}.
	\end{gather*}
	Consider the subset of mapping classes $M = M(X,Y,\beta,\kappa) \subseteq \mcg$ given by
	\[
	M := \bigcup_{n \in \mathbf{N}} F_n.
	\]
	
	We claim $M$ is $(X,Y,C,\kappa)$-sparse for some constant $C  = C(\mathcal{K},\beta) > 0$. Given $R > 0$ denote
	\[
	M_R := \{\mc \in M \ | \ d_\mathcal{T}(X,\mc.Y) \leq R\}.
	\]
	To prove the claim we need to show that, for every $R > 0$,
	\[
	|M_R| \preceq_{\mathcal{K},\beta}  e^{(6g-6-\kappa)R}.
	\]
	Notice that, for every $R > 0$,
	\[
	M_R \subseteq \bigcup_{n=0}^{\lfloor R \rfloor} F_n.
	\]
	Recall $0 < \kappa < \kappa_4/2$. By Theorem \ref{theo:thin_count_1}, for every $n \in \mathbf{N}$,
	\[
	|F_n| \leq |F(X,Y,n+1,V_n)|  \preceq_{\mathcal{K},\beta} \delta_n \cdot e^{(6g-6)(n+1)} + \delta_n^{-1} \cdot e^{(6g-6 - \kappa_4)(n+1)} \preceq_g e^{(6g-6-\kappa)n}.
	\]
	It follows that, for every $R > 0$,
	\[
	|M_R| \leq \sum_{n =0}^{\lfloor R \rfloor} |F_n| \preceq_\mathcal{K,\beta} \sum_{n =0}^{\lfloor R \rfloor}  e^{(6g-6-\kappa)n} \preceq_g e^{(6g-6-\kappa)R}.
	\]
	
	Now let $\mc \in \mcg \setminus M$. Denote $r := d_\mathcal{T}(X,\mc.Y)$ and $q_s := q_s(X,\mc.Y)$. Let $n \in \mathbf{N}$ be the unique non-negative integer such that $r \in [n,n+1)$. By the definition of $M$, we have $\mc \notin F_n$. Thus, by the definition of $F_n$, we must have $\ell_{\min}(q_s) \geq \delta_n = e^{-\kappa n} \geq e^{-\kappa r}$ and $h_\beta(q) \geq \delta_n = e^{-\kappa n} \geq e^{-\kappa r}$.
\end{proof}

\subsection*{Sparse subsets of vertically thin mapping classes.} Let $\beta$ be a closed curve on $S_g$. Let us recall the definition of the function $v_\beta \colon \qut \to \mathbf{R}$ introduced in (\ref{eq:vertical}). Given a marked quadratic differential $q \in \qut$, choose a flat geodesic representative of $\beta$ on $q$. Let $\{\beta_j\}_{j=1}^m$ be the sequence of saddle connections of this representative if it is singular, or a singleton consisting of the representative itself if it is a cylinder curve. For every $j \in \{1,\dots,m\}$ denote by $i(\beta_j,\Im(q))$ the total transverse measure of $\beta_j$ with respect to the singular measured foliation $\Im(q)$. Then,
\[
v_\beta(q) := \min_{j=1,\dots,m} \frac{i(\beta_j,\Im(q))}{\ell_{\beta_j}(q)}.
\]
By Proposition \ref{prop:flat_rep}, this value is independent of the choice of flat geodesic representative of $\beta$ on $q$. Furthermore, the function $v_\beta \colon \qut \to \mathbf{R}$ is continuous on strata of $\qut$.

Directly from Theorem \ref{theo:sparse_3} we deduce the following result. This result will play a crucial role in our proof of Theorem \ref{theo:main}. 

\begin{theorem}
	\label{theo:sparse_4}
	There exists a constant $\kappa_4 = \kappa_4(g) > 0$ such that for every compact subset $\mathcal{K} \subseteq \mathcal{T}_g$ and every closed curve $\beta$ on $S_g$, there exists a constant $C  = C(\mathcal{K},\beta) > 0$ with the following property. For every $X,Y \in \mathcal{K}$ and every $0 < \kappa < \kappa_4$, there exists an $(X,Y,C,\kappa)$-sparse subset $M = M(X,Y,\beta,\kappa) \subseteq \mcg$ such that for every $\mc \in \mcg \setminus M$, if $r := d_\mathcal{T}(X,\mc.Y)$ and $q_e := q_e(X,\mc.Y)$, then
	\[
	\ell_{\min}(q_e) \geq e^{-\kappa r}, \quad v_{\mc.\beta}(q_e) \geq e^{-\kappa r}.
	\]
\end{theorem}

\begin{proof}
	Let $\kappa_4 = \kappa_4(g) > 0$ be as in Theorem \ref{theo:sparse_3}. Fix $\mathcal{K} \subseteq \mathcal{T}_g$ compact and $\beta$ a closed curve on $S_g$. Let $C = C(\mathcal{K},\beta) > 0$ be as in Theorem \ref{theo:sparse_3}. Fix $X,Y \in \mathcal{K}$ and $0 < \kappa < \kappa_4$. Let $M = M(Y,X,\beta, \kappa) \subseteq \mcg$ be the $(Y,X,C,\kappa)$-sparse subset provided by Theorem \ref{theo:sparse_3}. As the Teichmüller metric is mapping class group invariant, the set $M^{-1} \subseteq \mcg$ of all inverses of $M$ is $(X,Y,C,\kappa)$-sparse.
	
	Directly from the definitions, one can check that, for every $\mc \in \mcg$ such that $\mc.X \neq Y$,
	\begin{gather*}
	\ell_{\min}(q_e(X,\mc.Y)) = \ell_{\min}(q_s(Y,\mc^{-1}.X)), \\
	v_{\mc.\beta}(q_e(X,\mc.Y)) = h_\beta(q_s(Y,\mc^{-1}.X)).
	\end{gather*}
	Fix $\mc \in \mcg \setminus M^{-1}$. Denote $r := d_\mathcal{T}(X,\mc.Y)$ and $q_e := q_e(X,\mc.Y)$. Directly from Theorem \ref{theo:sparse_3} and the identities above we deduce
	\[
	\ell_{\min}(q_e) \geq e^{-\kappa r}, \quad v_{\mc.\beta}(q_e) \geq e^{-\kappa r}. \qedhere
	\]	
\end{proof}

\section{Proofs of the main results}

\subsection*{Outline of this section.}  In this section we combine results from \S 6 -- 8 to prove Theorems \ref{theo:main}, \ref{theo:main_2}, and \ref{theo:main_3}. We actually prove stronger versions of these results, see Theorems \ref{theo:main_strong}, \ref{theo:main_2_unif}, and \ref{theo:main_3_strong}. We also prove a version of Theorem \ref{theo:main} that yields stronger conclusions for simple closed curves, see Theorem \ref{theo:main_2_unif}, and a version of Theorem \ref{theo:main_2} that applies to non-simple closed curves, see Theorem \ref{theo:main_2_strong}.

\subsection*{A quantitative estimate for geometric intersection numbers} For the rest of this section we fix an integer $g \geq 2$ and a connected, oriented, closed surface $S_g$ of genus $g$. Let us recall some of the notation introduced in previous sections. Recall that $\tt$ denotes the Teichmüller space of marked complex structures on $S_g$, or, equivalently, the Teichmüller space of marked hyperbolic structures on $S_g$. Recall that $d_\mathcal{T}$ denotes the Teichmüller metric on $\mathcal{T}_g$. Recall that $\mcg$ denotes the mapping class group of $S_g$, that this group acts properly discontinuously on $\tt$ by changing the markings, and that this action preserves the Teichmüller metric.

Recall that $\qut$ denotes the Teichmüller space of marked, unit area quadratic differentials on $S_g$ and that $\pi \colon \qut \to \tt$ denotes the natural projection to $\mathcal{T}_g$. Recall that the Teichmüller metric on $\mathcal{T}_g$ is complete and that its unit speed geodesics are precisely the projection to $\mathcal{T}_g$ of the Teichmüller geodesic flow orbits on $\qut$. Recall that $\mathcal{MF}_g$ denotes the space of singular measured foliations on $S_g$ and that $\Re(q),\Im(q) \in \mf$ denote the vertical an horizontal foliations of $q \in \qut$.

Recall that $\Delta \subseteq \mathcal{T}_g \times \mathcal{T}_g$ denotes the diagonal of $\mathcal{T}_g \times \mathcal{T}_g$, that $S(X) := \pi^{-1}(X)$ for every $X \in \mathcal{T}_g$, and that $q_s, q_e \colon \mathcal{T}_g \times \mathcal{T}_g - \Delta \to \mathcal{Q}^1\mathcal{T}_g$ denote the maps which to every pair of distinct points $X,Y \in \mathcal{T}_g$ assign the quadratic differentials $q_s(X,Y) \in S(X)$ and $q_e(X,Y) \in S(Y)$ corresponding to the tangent directions at $X$ and $Y$ of the unique Teichmüller geodesic segment from $X$ to $Y$.

Let $X,Y \in \mathcal{T}_g$, $C > 0$, and $\kappa > 0$. Recall that a subset of mapping classes $M \subseteq \mcg$ is said to be $(X,Y,C,\kappa)$-sparse if the following estimate holds for every $R > 0$,
\[
|\{\mc \in M \ | \ d_\mathcal{T}(X,\mc.Y) \leq R \}| \leq C \cdot e^{(6g-6-\kappa)R}.
\]

Recall that $\mathcal{C}_g$ denotes the space of geodesic currents on $S_g$. Recall that the space of projective geodesic currents $\mathcal{PC}_g := \mathcal{C}_g/\mathbf{R}^+$ is compact. Recall that $i(\cdot,\cdot)$ denotes the geometric intersection number pairing on $\mathcal{C}_g$ and that this pairing is symmetric, bilinear, and invariant with respect to the diagonal action of $\mcg$ on $\mathcal{C}_g \times \mathcal{C}_g$. Recall that, for every $q \in \qut$, we denote by $q \in \mathcal{C}_g$ the Liouville current of the singular flat metric induced by $q$ on $S_g$. Analogously, for every $X \in \mathcal{T}_g$, we denote by $X \in \mathcal{C}_g$ the Liouville current of the hyperbolic metric induced by $X$ on $S_g$.

Recall that, for every closed curve $\beta$ on $S_g$ and every marked hyperbolic structure $X \in \mathcal{T}_g$, we denote by $\ell_{\beta}(X)$ the length of the unique geodesic representative of $\beta$ with respect to $X$. Recall that the length $\ell_{\lambda}(X)$ of a singular measured foliation $\lambda \in \mathcal{MF}_g$ with respect to a marked hyperbolic structure $X \in \mathcal{T}_g$ can be defined in an analogous way.

The following result is a stonger version of Theorem \ref{theo:main}. The proof combines results from \S 6 -- 8. We refer the reader back to these sections for some of the terminology that will be used in the proof.

\begin{theorem}
	\label{theo:main_strong}
	There exists a constant $\kappa = \kappa(g) > 0$ such that for every compact subset $\mathcal{K} \subseteq \mathcal{T}_g$ and every closed curve $\beta$ on $S_g$, there exists a constant $C  = C(\mathcal{K},\beta) > 0$ with the following property. For every $X,Y \in \mathcal{K}$ there exists an $(X,Y,C,\kappa)$-sparse subset of mapping classes $M = M(X,Y,\beta) \subseteq \mcg$ such that for every geodesic current $\alpha$ on $S_g$ and every $\mc \in \mcg \setminus M$, if $r := d_\mathcal{T}(X,\mc.Y)$, $q_s := q_s(X,\mc.Y)$, and $q_e := q_e(X,\mc.Y)$, then
	\[
	i(\alpha,\mc.\beta) = i(\alpha,\Re(q_s)) \cdot i(\mc.\beta,\Im(q_e)) \cdot e^r + O_{\mathcal{K}}\left(i(\alpha,X) \cdot\ell_{\beta}(Y) \cdot e^{(1-\kappa)r}\right).
	\]
\end{theorem}

\begin{proof}
	Denote by $\mathcal{Q}^1 \subseteq \qum$ the principal stratum of $\qum$. Fix $\eta \in (0,1)$. Let $K  = K(\eta) \subseteq \mathcal{Q}^1$ be as in Theorem \ref{theo:thin_traj}. Let $\rho_1 = \rho_1(K,\eta) > 0$, $\epsilon_1 = \epsilon_1(K,\eta) > 0$, and $\kappa_1 = \kappa_1(K,\eta) > 0$ be as in Theorem \ref{theo:prelim_2}. Fix $0 < \rho < \rho_1$, $0 < \epsilon < \epsilon_1$ such that $\rho (1+\epsilon) < 1$, and $s > 0$. Let $K' = K'(\rho,\epsilon,s) \subseteq \mathcal{Q}^1$ and $\kappa_2 = \kappa_2(\rho,\epsilon) > 0$ be as in Theorem \ref{theo:rec_1}. Let $\kappa_3 = \kappa_3(g) > 0$ be as in Theorem \ref{theo:sparse_1}. Let $\kappa_4 = \kappa_4(g) > 0$ be as in Theorem \ref{theo:sparse_4}. We will consider all the objects introduced above as depending only on $g$.
	
	Fix $\mathcal{K} \subseteq \mathcal{T}_g$ compact and $X,Y \in \mathcal{K}$. Let $M_1 = M_1(X,Y,K,\eta) \subseteq \mcg$ be the set of all mapping classes $\mc \in \mcg$ such that $\mc.Y = X$ or  $\mc.Y \neq X$ and $-q_e(X,\mc.Y)$ is not $(K,r,\eta)$-recurrent for $r := d_\mathcal{T}(X,\mc.Y)$. By Theorem \ref{theo:thin_traj}, $M_1 \subseteq \mcg$ is $(X,Y,C_1,\eta/2)$-sparse for some constant $C_1 = C_1(K,\eta) > 0$. Let $M_2 = M_2(X,Y,K',\rho,\epsilon,s) \subseteq \mcg$ be the set of all mapping classes $\mc \in \mcg$ such that $\mc.Y = X$ or  $\mc.Y \neq X$ and $-q_e(X,\mc.Y)$ has no $(K',r,\rho,\epsilon,s)$-itineraties for $r := d_\mathcal{T}(X,\mc.Y)$. By Theorem \ref{theo:rec_1}, $M_2 \subseteq \mcg$ is $(X,Y,C_2,\kappa_2)$-sparse for some constant $C_2 = C_2(\mathcal{K},\rho,\epsilon,s) > 0$. 
	
	Denote $\kappa_0 = \kappa_0(g) := \min\{\kappa_2,\kappa_3,\kappa_4,\eta/2\} > 0$. Consider $0 < \kappa < \kappa_0$ to be fixed later. By Theorem \ref{theo:sparse_1}, there exists a constant $C_3 = C_3(\mathcal{K}) > 0$ and an $(X,Y,C_3,\kappa)$-sparse subset $M_3 = M_3(X,Y,\kappa) \subseteq \mcg$ such that for every $\mc \in \mcg \setminus M_3$, if $r := d_\mathcal{T}(X,\mc.Y)$ and $q_s := q_s(X,\mc.Y)$, then $\ell_{\min}(q_s)  \geq e^{-\kappa r}$. Fix a closed curve $\beta$ on $S_g$. By Theorem \ref{theo:sparse_4}, there exists a constant $C_4 =  C_4(\mathcal{K},\beta) > 0$ and an  $(X,Y,C_4,\kappa)$-sparse subset $M_4 = M_4(X,Y,\beta,\kappa) \subseteq \mcg$ such that for every $\mc \in \mcg \setminus M_4$, if $r := d_\mathcal{T}(X,\mc.Y)$ and $q_e := q_e(X,\mc.Y)$, then $\ell_{\min}(q_e) \geq e^{-\kappa r}$ and $v_{\mc.\beta}(q_e) \geq e^{-\kappa r}$.
	
	Consider the subset of mapping classes $M = M(X,Y,\beta) \subseteq \mcg$ given by
	\[
	M := M_1 \cup M_2 \cup M_3 \cup M_4.
	\]
	Denote $C = C(\mathcal{K},\beta) := \max\{C_1,C_2,C_3,C_4\} > 0$. As $M \subseteq \mcg$ is a finite union of $(X,Y,C,\kappa)$-sparse subsets, it is also $(X,Y,C,\kappa)$-sparse. Let $\mc \in \mcg \setminus M$. Denote $r := d_\mathcal{T}(X,\mc.Y)$, $q_s := q_s(X,\mc.Y)$, and $q_e := q_e(X,\mc.Y)$. By the definition of $M$, the following conditions hold:
	\begin{enumerate}
		\item $\mc.Y \neq X$,
		\item $-q_e(X,\mc.Y)$ is $(K,r,\eta)$-recurrent,
		\item $-q_e(X,\mc.Y)$ has a $(K',r,\rho,\epsilon,s)$-itinerary,
		\item $\ell_{\min}(q_s) \geq e^{-\kappa r}$,
		\item $\ell_{\min}(q_e) \geq e^{-\kappa r}$,
		\item $v_{\mc.\beta}(q_e) \geq e^{-\kappa r}$.
	\end{enumerate}
	
	Fix a geodesic current $\alpha$ on $S_g$. Recall that $\kappa_1 = \kappa_1(g) > 0$ is as in Theorem \ref{theo:prelim_2} and that $K'$, $\epsilon$, and $s$ depend only on $g$. As conditions (1), (2), and (3) above hold, we can use Theorem \ref{theo:prelim_2} to deduce
	\begin{align*}
	i(\alpha,\mc.\beta) &= i(\alpha,\Re(q_s)) \cdot i(\mc.\beta,\Im(q_e)) \cdot e^r \\
	&\phantom{=}+ O_{g}\left(i(\alpha,q_s) \cdot \ell_{\min}^\dagger(q_s)^{-1} \cdot \ell_{\mc.\beta}(q_e) \cdot \ell_{\min}^\dagger(q_e)^{-5} \cdot v_{\mc.\beta}(q_e)^{-1} \cdot e^{(1-\kappa_1)r}\right).
	\end{align*}
	Using conditions (4), (5), and (6) above, this estimate reduces to
	\begin{equation}
	\label{eq:pre_almost}
	i(\alpha,\mc.\beta) = i(\alpha,\Re(q_s)) \cdot i(\mc.\beta,\Im(q_e)) \cdot e^r + O_g\left(i(\alpha,q_s) \cdot \ell_{\mc.\beta}(q_e)\cdot e^{(1-\kappa_1 + 7\kappa)r}\right).
	\end{equation}
	Let $0 < \kappa < \kappa_0$ be small enough so that $7\kappa < \kappa_1/2$. Under this assumption, (\ref{eq:pre_almost}) reduces to
	\begin{equation}
	\label{eq:almost_est}
	i(\alpha,\mc.\beta) = i(\alpha,\Re(q_s)) \cdot i(\mc.\beta,\Im(q_e)) \cdot e^r + O_{g}\left(i(\alpha,q_s) \cdot \ell_{\mc.\beta}(q_e)\cdot e^{(1-\kappa_1/2)r}\right).
	\end{equation}
	As $\mathcal{K} \subseteq \mathcal{T}_g$ is compact, as $X,Y \in \mathcal{K}$, as $q_s,\mc^{-1}.q_e \in \pi^{-1}(\mathcal{K})$, and as $\mathcal{C}_g$ is projectively compact,
	\begin{equation}
	\label{eq:SS1}
	i(\alpha,q_s) \preceq_{\mathcal{K}} i(\alpha,X), \quad \ell_{\mc.\beta}(q_e) = \ell_{\beta}(\mc^{-1}.q_e) \preceq_{\mathcal{K}} \ell_{\beta}(Y)
	\end{equation}
	Putting together (\ref{eq:almost_est}) and (\ref{eq:SS1}) we conclude
		\begin{equation*}
	i(\alpha,\mc.\beta) = i(\alpha,\Re(q_s)) \cdot i(\mc.\beta,\Im(q_e)) \cdot e^r + O_{\mathcal{K}}\left(i(\alpha,X) \cdot \ell_{\beta}(Y) \cdot e^{(1-\kappa_1/2)r}\right). \qedhere
	\end{equation*}
\end{proof}

\subsection*{Effective convergence to the boundary at infinity of Teichmüller space} Directly from Theorem \ref{theo:main_strong} we deduce the following version of Theorem \ref{theo:main_2} for general closed curves. Notice that, in contrast with Theorem \ref{theo:main_2}, the sparse subset of mapping classes that needs to be discarded to ensure the estimate in this result holds depends on the closed curve $\beta$ of interest.

\begin{theorem}
	\label{theo:main_2_strong}
	There exists a constant $\kappa = \kappa(g) > 0$ such that for every compact subset $\mathcal{K} \subseteq \mathcal{T}_g$ and every closed curve $\beta$ on $S_g$, there exists a constant $C  = C(\mathcal{K},\beta) > 0$ with the following property. For every $X,Y \in \mathcal{K}$ there exists an $(X,Y,C,\kappa)$-sparse subset of mapping classes $M = M(X,Y,\beta) \subseteq \mcg$ such that for every $\mc \in \mcg \setminus M$, if $r := d_\mathcal{T}(X,\mc.Y)$ and $q_s := q_s(X,\mc.Y)$, then
	\[
	\ell_{\beta}(\pi(a_r q_s)) = i(\beta,\Re(q_s)) \cdot \ell_{\Im(q_s)}(\pi(a_r q_s))+ O_{\mathcal{K}}\left( \ell_{\beta}(Y) \cdot e^{(1-\kappa)r}\right). 
	\]
\end{theorem}

\begin{proof}
	Let $\kappa = \kappa(g) > 0$ be as in Theorem \ref{theo:main_strong}. Fix $\mathcal{K} \subseteq \mathcal{T}_g$ compact and $\beta$ a closed curve on $S_g$. Let $C = C(\mathcal{K},\beta) > 0$ be as in Theorem \ref{theo:main_strong}. Fix $X,Y \in \mathcal{K}$. Let $M = M(Y,X,\beta) \subseteq \mcg$ be the $(Y,X,C,\kappa)$-sparse subset of mapping classes provided by Theorem \ref{theo:main_strong}. As the Teichmüller metric is mapping class group invariant, the set $M^{-1} \subseteq \mcg$ of all inverses of $M$ is $(X,Y,C,\kappa)$-sparse.
	
	 Directly from the definitions, one can check that, for every $\mc \in \mcg$ such that $\mc.X \neq Y$,
	\begin{gather*}
	\Re(q_s(X,\mc^{-1}.Y)) = \mc^{-1}.\Im(q_e(Y,\mc.X)),\\
	\Im(q_e(X,\mc^{-1}.Y)) = \mc^{-1}.\Re(q_s(Y,\mc.X)).
	\end{gather*}
	Fix $\mc \in \mcg \setminus M^{-1}$. Denote $r := d_\mathcal{T}(X,\mc.Y)$, $q_s := q_s(X,\mc.Y)$, and $q_e := q_e(X,\mc.Y)$. Applying Theorem \ref{theo:main_strong} to the Liouville current $Y \in \mathcal{C}_g$ yields the following estimate,
	\begin{equation}
	\label{eq:AA1}
	i(Y,\mc^{-1}.\beta) = i(Y,\mc^{-1}.\Im(q_e)) \cdot i(\mc^{-1}.\beta,\mc^{-1}.\Re(q_s)) \cdot e^r + O_{\mathcal{K}}\left(i(Y,X) \cdot \ell_\beta(Y) \cdot e^{(1-\kappa)r}\right).
	\end{equation}
	Notice that $\Im(q_s) = e^r \cdot \Im(q_e)$ and $\mc.Y  = \pi(q_e) = \pi(a_r q_s)$. Rearranging terms in (\ref{eq:AA1}) and using the definition of Liouville currents, we conclude
	\[
	\ell_{\beta}(\pi(a_r q_s)) = i(\beta,\Re(q_s)) \cdot \ell_{\Im(q_s)}(\pi(a_r q_s))+ O_{\mathcal{K}}\left( \ell_\beta(Y) \cdot e^{(1-\kappa)r}\right). \qedhere
	\]
\end{proof}

\subsection*{Dehn-Thurston coordinates.} Our next goal is to prove versions of Theorems \ref{theo:main_strong} and \ref{theo:main_2_strong} that yield stronger conclusions for simple closed curves. The proofs will rely on Theorem \ref{theo:main_strong} and some facts about Dehn-Thurston coordinates of singular measured foliations.

Dehn-Thurston coordinates parametrize the space of singular measured foliations $\mf$ in terms of intersection and twisting numbers with respect to the components of a pair of pants decomposition of $S_g$ \cite[\S2.6]{PH92}. Consider the piecewise linear manifold $\IT:= \mathbf{R}^2 / \langle-1\rangle$ endowed with the quotient Euclidean metric. The product $\IT^{3g-3}$ is a piecewise linear manifold which we endow with the product metric. Every set of Dehn-Thurston coordinates provides a homeomorphism $F \colon \mf \to \IT^{3g-3}$ equivariant with respect to the natural $\mathbf{R}^+$ scaling actions on $\mf$ and $\IT^{3g-3}$. 

Given a finite collection $\Gamma := \smash{\{\gamma_i\}_{i=1}^k}$ of simple closed curves on $S_g$, denote by $I_\Gamma \colon \mf \to \mathbf{R}^{k}$ the map which to every $\lambda \in \mf$ assigns the vector $I_\Gamma(\lambda) := \smash{(i(\gamma_i,\lambda))_{i=1}^k} \in \mathbf{R}^k$. The following result shows that any set of Dehn-Thurston coordinates can be written as a continuous, piecewise linear function of finitely many intersection numbers with respect to simple closed curves.

\begin{proposition}
	\label{prop:DT_int}
	\cite[Appendix C]{FLP12}
	Let $F \colon \mf \to \IT^{3g-3}$ be a set of Dehn-Thurston coordinates of $\mf$. Then, there exists a finite collection $\Gamma := \{\gamma_i\}_{i=1}^{k}$ of simple closed curves on $S_g$ and a continuous, piecewise linear map $H \colon \mathbf{R}^k \to \IT^{3g-3}$ such that $F = H \circ I_{\Gamma}$.
\end{proposition}

Given a geodesic current $\alpha$ on $S_g$, consider the function $i(\alpha,\cdot) \colon \mf \to \mathbf{R}$ it induces on $\mf$. The following result shows that this function is Lipschitz with respect to any set of Dehn-Thurston coordinates. A version of this result for Liouville currents of hyperbolic metrics was first proved by Luo and Stong \cite{LS02}. We refer the reader to Appendix B for a detailed proof of the general case.

\begin{proposition}
	\label{prop:ml_lip}
	Let $F \colon \mf \to \IT^{3g-3}$ be a set of Dehn-Thurston coordinates of $\mf$ and $K \subseteq \mathcal{C}_g$ be a compact subset of geodesic current on $S_g$. Then, there exists a constant $L  = L(F,K) > 0$ such that for every $\alpha \in K$, the composition $i(\alpha,\cdot) \circ F^{-1} \colon \IT^{3g-3} \to \mathbf{R}$ is $L$-Lipschitz.
\end{proposition}

\subsection*{Uniform estimates over simple closed curves.} We now state and prove the aforementioned versions of Theorems \ref{theo:main_strong} and \ref{theo:main_2_strong} that yield stronger conclusions for simple closed curves. The main feature of these results is the fact that the sparse subsets of mapping classes that need to be discarded for the corresponding estimates to hold can be chosen uniformly over all simple closed curves.

\begin{theorem}
	\label{theo:main_unif}
	There exists a constant $\kappa = \kappa(g) > 0$ such that for every compact subset $\mathcal{K} \subseteq \mathcal{T}_g$, there exists a constant $C  = C(\mathcal{K}) > 0$ with the following property. For every $X,Y \in \mathcal{K}$, there exists an $(X,Y,C,\kappa)$-sparse subset of mapping classes $M = M(X,Y) \subseteq \mcg$ such that for every compact subset $K \subseteq \mathcal{C}_g$, every geodesic current $\alpha \in K$, every simple closed curve $\beta$ on $S_g$, and every $\mc \in \mcg \setminus M$, if $r := d_\mathcal{T}(X,\mc.Y)$, $q_s := q_s(X,\mc.Y)$, and $q_e := q_e(X,\mc.Y)$, then
	\[
	i(\alpha,\mc.\beta) = i(\alpha,\Re(q_s)) \cdot i(\mc.\beta,\Im(q_e)) \cdot e^r + O_{\mathcal{K},K}\left(\ell_\beta(Y) \cdot e^{(1-\kappa)r}\right).
	\]
\end{theorem}

\begin{proof}
	Let $\kappa = \kappa(g) > 0$ be as in Theorem \ref{theo:main_strong}. Fix a set of Dehn-Thurston coordinates $F \colon \mf \to \IT^{3g-3}$ of $\mf$. By Proposition \ref{prop:DT_int}, there exists a finite collection $\Gamma := \smash{\{\gamma_j\}_{j=1}^k}$ of simple closed curves on $S_g$ and a continuous, piecewise linear map $H \colon \mathbf{R}^k \to \IT^{3g-3}$ such that $F = H \circ I_{\Gamma}$. Fix a compact subset $\mathcal{K} \subseteq \mathcal{T}_g$. For every $j \in \{1,\dots,k\}$ let $C_j := C(\mathcal{K},\beta_j) > 0$ be as in Theorem \ref{theo:main_strong}. Fix $X,Y \in \mathcal{K}$. For every $j \in \{1,\dots,k\}$ let $M_j := M(Y,X,\beta_j) \subseteq \mcg$ be the $(Y,X,C_j,\kappa)$-sparse subset of mapping classes provided by Theorem \ref{theo:main_strong}. For every $j \in \{1,\dots,k\}$ denote by $M_j^{-1} \subseteq \mcg$ the set of all inverses of $M_j$. Consider the subset of of mapping classes $M = M(X,Y) \subseteq \mcg$ given by
	\[
	M := \bigcup_{j=1}^k M_j^{-1}.
	\]
	
	Denote $C = C(\mathcal{K}) := \smash{\max_{j=1,\dots,k} C_j} > 0$. As the Teichmüller metric is mapping class group invariant, for every $j \in \{1,\dots,k\}$, the set $\smash{M_j^{-1}} \subseteq \mcg$ of all inverses of $M_j$ is $(X,Y,C,\kappa)$-sparse. As $M \subseteq \mcg$ is a finite union of $(X,Y,C,\kappa)$-sparse subsets, it is also $(X,Y,C,\kappa)$-sparse. 
	
	Recall that, for every $\mc \in \mcg$ such that $\mc.X \neq Y$, the following identities hold,
	\begin{gather*}
	\Re(q_s(X,\mc^{-1}.Y))  = \mc^{-1}.\Im(q_e(Y,\mc.X)),\\
	\Im(q_e(X,\mc^{-1}.Y))  = \mc^{-1}.\Re(q_s(Y,\mc.X)).
	\end{gather*}
	Fix a closed curve $\beta$ on $S_g$ and a mapping class $\mc \in \mcg \setminus M$. Denote $r := d_\mathcal{T}(X,\mc.Y)$, $q_s := q_s(X,\mc.Y)$, and $q_e := q_e(X,\mc.Y)$. By Theorem \ref{theo:main_strong}, for every $j \in \{1,\dots,k\}$,
	\[
	i(\beta,\mc^{-1}.\gamma_j) = i(\beta,\mc^{-1}.\Im(q_e)) \cdot i(\mc^{-1}.\gamma_j,\mc^{-1}.\Re(q_s)) \cdot e^r + O_\mathcal{K}\left(\ell_\beta(Y) \cdot \ell_{\gamma_j}(X) \cdot e^{(1-\kappa)r}\right).
	\]
	Rearranging terms in this estimate, we deduce that, for every $j \in \{1,\dots,k\}$,
	\begin{equation}
	\label{eq:A1}
	i(\gamma_j,\mc.\beta) =  i(\gamma_j,\Re(q_s)) \cdot i(\mc.\beta,\Im(q_e)) \cdot e^r + O_\mathcal{K}\left(\ell_{\beta}(Y) \cdot \ell_{\gamma_j}(X)  \cdot e^{(1-\kappa)r}\right).
	\end{equation}
	Denote by $\|\cdot\|$ the Euclidean norm on $\mathbf{R}^k$. It follows from (\ref{eq:A1}) that
	\begin{equation}
	\label{eq:A2}
	\| I_\Gamma(\mc.\beta) -  e^r \cdot i(\mc.\beta,\Im(q_e)) \cdot I_\Gamma(\Re(q_s))\| \preceq_{\mathcal{K}} \ell_{\beta}(Y) \cdot e^{(1-\kappa)r}.
	\end{equation}
	Denote by $d$ the product metric on $\IT^{3g-3}$. As the map $H \colon \mathbf{R}^k \to \IT^{3g-3}$ is continuous and piecewise linear, it is also Lipschitz. Recall that $F = H \circ I_\Gamma$. These facts together with (\ref{eq:A2}) allow us to deduce
	\begin{equation}
	\label{eq:A3}
	d(F(\mc.\beta),e^r \cdot i(\mc.\beta,\Im(q_e)) \cdot F(\Re(q_s))) \preceq_{\mathcal{K}} \ell_{\beta}(Y) \cdot e^{(1-\kappa)r}.
	\end{equation}
	
	Fix a geodesic current $\alpha \in K$. By Proposition \ref{prop:ml_lip}, the map $i(\alpha,\cdot) \circ F^{-1} \colon \IT^{3g-3} \to \mathbf{R}$ is $L$-Lipschitz for some constant $L = L(F,K) > 0$. This fact together with (\ref{eq:A3}) allows us to conclude
	\[
	i(\alpha,\mc.\beta) = i(\alpha,\Re(q_s)) \cdot i(\mc.\beta,\Im(q_e)) \cdot e^r + O_{\mathcal{K},K}\left( \ell_\beta(Y) \cdot e^{(1-\kappa)r}\right). \qedhere
	\]
\end{proof}

The same arguments used to deduce Theorem \ref{theo:main_2_strong} from Theorem \ref{theo:main_strong} yield the following stronger version of Theorem \ref{theo:main_2} as a consequence of Theorem \ref{theo:main_unif}.

\begin{theorem}
	\label{theo:main_2_unif}
	There exists a constant $\kappa = \kappa(g) > 0$ such that for every compact subset $\mathcal{K} \subseteq \mathcal{T}_g$, there exists a constant $C  = C(\mathcal{K}) > 0$ with the following property. For every $X,Y \in \mathcal{K}$, there exists an $(X,Y,C,\kappa)$-sparse subset of mapping classes $M = M(X,Y) \subseteq \mcg$ such that for every simple closed curve $\beta$ on $S_g$ and every $\mc \in \mcg \setminus M$, if $r := d_\mathcal{T}(X,\mc.Y)$ and $q_s := q_s(X,\mc.Y)$, then
	\[
	\ell_{\beta}(\pi(a_r q_s)) = i(\beta,\Re(q_s)) \cdot \ell_{\Im(q_s)}(\pi(a_r q_s))+ O_{\mathcal{K}}\left( \ell_\beta(Y) \cdot e^{(1-\kappa)r}\right). 
	\]
\end{theorem}

\subsection*{Comparing the Teichmüller and Thurston metrics.} Recall that $d_\mathrm{Thu}$ denotes the Thurston metric on $\mathcal{T}_g$. Recall that $\mathcal{S}_g$ denotes the set of all simple closed curves on $S_g$. By work of Thurston \cite{Thu98}, for every pair of marked hyperbolic structures $X,Y \in \mathcal{T}_g$,
\begin{equation}
\label{eq:Thu}
d_\mathrm{Thu}(X,Y) = \log\left(\sup_{\beta \in \mathcal{S}_g} \frac{\ell_{\beta}(Y)}{\ell_{\beta}(X)}\right).
\end{equation}

Recall that, for every marked hyperbolic structure $X \in \mathcal{T}_g$, we denote by $D_X \colon \mf \to \smash{\mathbf{R}^+}$ the function which to every singular measured foliation $\lambda \in \mf$ assigns the value
\begin{equation*}
D_X(\lambda) := \sup_{\beta \in \mathcal{S}_g}\left( \frac{i(\beta,\lambda)}{\ell_\beta(X)}\right).
\end{equation*}

Combining (\ref{eq:Thu}) with Theorem \ref{theo:main_2_unif}, we deduce the following stronger version of Theorem \ref{theo:main_3}.

\begin{theorem}
	\label{theo:main_3_strong}
	There exists a constant $\kappa = \kappa(g) > 0$ such that for every compact subset $\mathcal{K} \subseteq \mathcal{T}_g$, there exists a constant $C  = C(\mathcal{K}) > 0$ with the following property. For every $X,Y \in \mathcal{K}$, there exists an $(X,Y,C,\kappa)$-sparse subset of mapping classes $M = M(X,Y) \subseteq \mcg$ such that for every $\mc \in \mcg \setminus M$, if $r := d_\mathcal{T}(X,\mc.Y)$, $q_s := q_s(X,\mc.Y)$, and $q_e := q_e(X,\mc.Y)$, then
	\[
	d_\mathrm{Thu}(X,\mc.Y) = d_\mathcal{T}(X,\mc.Y) + \log D_X(\Re(q_s)) + \log \ell_{\Im(q_e)}(\mc.Y) + O_\mathcal{K}\left(e^{-\kappa r}\right).
	\]
\end{theorem}

\begin{proof}
	Let $\kappa = \kappa(g) > 0$ be as in Theorem \ref{theo:main_2_unif}. Fix a compact subset $\mathcal{K} \subseteq \mathcal{T}_g$. Let $C = C(\mathcal{K}) > 0$ be as in Theorem \ref{theo:main_2_unif}. Fix $X,Y \in \mathcal{K}$. Let $M = M(X,Y) \subseteq \mcg$ be the $(X,Y,C,\kappa)$-sparse subset of mapping classes provided by Theorem \ref{theo:main_2_unif}.
	
	Fix $\mc \in \mcg \setminus M$. Denote $r := d_\mathcal{T}(X,\mc.Y)$, $q_s := q_s(X,\mc.Y)$, and $q_e := q_e(X,\mc.Y)$. By Theorem \ref{theo:main_2_unif}, for every simple closed curve $\beta$ on $S_g$,
	\[
	\ell_{\beta}(\mc.Y) = i(\beta,\Re(q_s)) \cdot \ell_{\Im(q_e)}(\mc.Y) \cdot e^r + O_\mathcal{K}\left( \ell_{\beta}(Y) \cdot e^{(1-\kappa)r}\right).
	\]
	Dividing this estimate by $\ell_{\beta}(X)$ for every simple closed curve $\beta$ on $S_g$ yields
	\begin{equation}
	\label{eq:B1}
	\frac{\ell_{\beta}(\mc.Y)}{\ell_\beta(X)} = \frac{i(\beta,\Re(q_s))}{\ell_\beta(X)} \cdot \ell_{\Im(q_e)}(\mc.Y) \cdot e^r + O_\mathcal{K}\left(\frac{\ell_\beta(Y)}{\ell_\beta(X)} \cdot e^{(1-\kappa)r}\right).
	\end{equation}
	As $\mathcal{K} \subseteq \mathcal{T}_g$ is compact, as $X,Y \in \mathcal{K}$, and as $\mf$ is projectively compact, $\ell_\beta(Y)  \preceq_{\mathcal{K}} \ell_\beta(X)$ for every simple closed curve $\beta$ on $S_g$. This fact together with (\ref{eq:B1}) implies
	\begin{equation*}
	\frac{\ell_{\beta}(\mc.Y)}{\ell_\beta(X)} = \frac{i(\beta,\Re(q_s))}{\ell_\beta(X)} \cdot \ell_{\Im(q_e)}(\mc.Y) \cdot e^r + O_\mathcal{K}\left(e^{(1-\kappa)r}\right).
	\end{equation*}
	Taking supremum over all simple closed curves $\beta$ on $S_g$ and using (\ref{eq:Thu}) we deduce
	\[
	\exp({d_\mathrm{Thu}(X,\mc.Y)}) = D_X(\Re(q_s)) \cdot \ell_{\Im(q_e)}(\mc.Y) \cdot e^r + O_\mathcal{K} \left( e^{(1-\kappa)r}\right).
	\]
	More explicitly, there exists a constant $A = A(\mathcal{K})  > 0$ such that 
	\begin{equation}
	\label{eq:Y1}
	\left|\exp(d_\mathrm{Thu}(X,\mc.Y)) - D_X(\Re(q_s)) \cdot \ell_{\Im(q_e)}(\mc.Y) \cdot e^r \right| \leq A \cdot e^{(1-\kappa)r}.
	\end{equation}
	
	By the mean value theorem, for every $x,y \in \smash{\mathbf{R}^+}$,
	\begin{equation}
	\label{eq:mvt}
	|\log(x) - \log(y)| \leq  \max\left\lbrace \frac{1}{x}, \frac{1}{y} \right\rbrace \cdot |x -y|.
	\end{equation}
	As $\mathcal{K} \subseteq \mathcal{T}_g$ is compact, as $X,Y \in \mathcal{K}$, and as $q_s, \mc^{-1}.q_e \in \pi^{-1}(\mathcal{K})$,
	\[
	D_X(\Re(q_s)) \cdot \ell_{\Im(q_e)}(\mc.Y) = D_X(\Re(q_s)) \cdot \ell_{\mc^{-1}.\Im(q_e)}(Y) \succeq_{\mathcal{K}} 1.
	\]
	Thus, there exists a constant $r_0 = r_0(\mathcal{K}) > 0$ such that if $r > r_0$, then 
	\begin{equation}
	\label{eq:Y2}
	D_X(\Re(q_s)) \cdot \ell_{\Im(q_e)}(\mc.Y) \cdot e^r - A \cdot e^{(1-\kappa)r} \succeq_{\mathcal{K}} e^r.
	\end{equation}
	It follows from (\ref{eq:mvt}), (\ref{eq:Y2}), and (\ref{eq:Y1}) that, under the assumption $r > r_0$,
	\[
	d_{\mathrm{Thu}}(X,\mc.Y) = r + \log D_X(\Re(q_s)) + \log \ell_{\Im(q_e)}(\mc.Y) + O_\mathcal{K}\left(e^{-\kappa r} \right).
	\]
	The same estimate holds without the assumption $r > r_0$ by increasing the implicit constant.  Recalling that $r := d_\mathcal{T}(X,\mc.Y)$, we conclude
	\[
	d_\mathrm{Thu}(X,\mc.Y) = d_\mathcal{T}(X,\mc.Y) + \log D_X(\Re(q_s)) + \log \ell_{\Im(q_e)}(\mc.Y) + O_\mathcal{K}\left(e^{-\kappa r}\right). \qedhere
	\]
\end{proof}

\appendix

\section{Horizontally thin sectors of Teichmüller space}

\subsection*{Outline of this appendix.} In this appendix we give a complete proof of Theorem \ref{theo:slope_sector_meas}. We follow the general outline of the proof of \cite[Theorem 7.16]{Ara20b}. 
We refer the reader to \cite{Ara20b} for more details on some of the methods used in the proof.

\subsection*{Period coordinates.} For the rest of this appendix we fix an integer $g \geq 2$ and a connected, oriented, closed surface $S_g$ of genus $g$. Denote by $\qt$ the Teichmüller space of marked holomorphic quadratic differentials on $S_g$. This space can be stratified according to the order of the zeroes and the condition of being the square of an Abelian differential. Let $\mathcal{Q} \subseteq \qt$ be one of these strata. Given a quadratic differential $q_0 \in \mathcal{Q}$, recording the holonomy of enough edges of a triangulation by saddle connections of $q_0$ with respect to quadratic differentials $q \in \mathcal{Q}$ sufficiently close to $q_0$ provides local coordinates near $q_0$ for the stratum $\mathcal{Q}$. These coordinates are called \textit{period coordinates}.

More precisely, for every quadratic differential $q_0 \in \mathcal{Q}$ and every triangulation by saddle connections $\Delta$ of $q_0$, there exists a finite collection $\{\gamma_i\}_{i=1}^d$ with $d := \mathrm{dim}_{\mathbf{C}}(\mathcal{Q})$ of oriented edges of $\Delta$ and an open subset $U \subseteq \mathcal{Q}$ containing $q_0$ such that the triangulation $\Delta$ persists on every $q \in U$, such that a branch of $\sqrt{q}$ can be chosen continuously over all $q \in U$, and such that the map $\Phi \colon U \to \mathbf{C}^d$ defined for every $q \in U$ as follows is a diffeomorphism,
\[
\Phi(q) := \left( \int_{\gamma_i} \sqrt{q}\right)_{i=1}^d \in \mathbf{C}^d.
\]

\subsection*{The strongly unstable foliation.} Recall that $\qut$ denotes the Teichmüller space of marked, holomorphic, unit area quadratic differentials on $S_g$. Recall that $\mathcal{MF}_g$ denotes the space of singular measured foliations on $S_g$ and that $\Re(q),\Im(q) \in \mf$ denote the vertical an horizontal foliations of $q \in \qut$. For every $q_0 \in \qut$, its \textit{strongly unstable leaf} $\allowbreak\alpha^{uu}(q_0) \subseteq \qut$ is given by
\[
\alpha^{uu}(q_0) := \{ q \in \qut \ | \ \Im(q) = \Im(q_0) \}. 
\]
In period coordinates, the intersection of $\allowbreak\alpha^{uu}(q_0)$ with the stratum containing $q_0$ is given by all quadratic differentials whose periods have the same imaginary parts as those of $q_0$. The strongly unstable leaves give rise to a topological foliation $\mathcal{F}^{uu}$ called the \textit{strongly unstable foliation} of $\qut$.  

\subsection*{The Euclidean metric.} Period coordinates give rise to a local Lipschitz class of metrics on $\qut$. This class contains metrics induced by fiberwise norms on the tangent bundle of $\qut$. Explicit constructions of these metrics can be found in \cite[\S3.5]{ABEM12} and \cite[Definition 3.2]{F17}. For the rest of this appendix we fix one such metric, denote it by $d_E$, and refer to it as the \textit{Euclidean metric} of $\qut$. Denote by $\| \cdot \|_E$ the corresponding fiberwise norm on the tangent bundle of $\qut$. 

Recall that $\mathcal{T}_g$ denotes the Teichmüller space of marked complex structures on $S_g$. Recall that $\pi \colon \qut \to \tt$ denotes the natural projection. For the rest of this appendix we denote by $\mathcal{Q} \subseteq \qt$ the principal stratum of $\qt$, i.e., the stratum of marked, unit area, holomorphic quadratic differentials on $S_g$ all of whose zeroes are simple. This stratum is an open subset of $\qt$ and has complex dimension $\mathrm{dim}_{\mathbf{C}}(\mathcal{Q}^1) = 6g-6$. Denote by $\mathcal{Q}^1 \subseteq \qut$ the unit area locus of this stratum. On $\mathbf{C}^{6g-6}$ consider the standard Euclidean norm $\| \cdot \|$.  The Euclidean metric satisfies the following fundamental property.

\begin{lemma}
	\label{lem:euc_fund_prop}
	Let $\mathcal{K} \subseteq \tt$ be a compact subset. Consider a period coordinates chart $\Phi \colon U \to \mathbf{C}^{6g-6}$ defined on an open subset $U \subseteq \mathcal{Q}$. Then, for any $q \in U \cap \pi^{-1}(\mathcal{K})$ and any $v \in T_q \qut$,
	\[
	\| v \|_E \asymp_{\Phi,\mathcal{K}} \|d\Phi_q v\|.
	\]
\end{lemma}

Recall that $d_\mathcal{T}$ denotes the Teichmüller metric on $\tt$. The following result is a direct consequence of the uniform continuity on compact subsets of the projection $\pi \colon \qut \to \tt$. See \cite[Theorem 1.2]{F17} for a related quantitative estimate.

\begin{lemma}
	\cite[Lemma 3.9]{ABEM12}
	\label{lem:ET_1}
	Let $\mathcal{K} \subseteq \tt$ be a compact subset. There exists a continuous function $f \colon \mathbf{R}^+ \to \mathbf{R}^+$ with $f(r) \to 0$ as $r \to 0$ such that for any pair of distinct points $q_1,q_2 \in \pi^{-1}(\mathcal{K})$,
	\[
	d_\mathcal{T}(\pi(q_1),\pi(q_2)) \leq f(d_E(q_1,q_2)).
	\]
\end{lemma}

Directly from Lemma \ref{lem:ET_1} we deduce the following result.

\begin{lemma}
	\label{lem:ET_2}
	Let $\mathcal{K} \subseteq \mathcal{T}_g$ be a compact subset. There exists a constant $s_1 = s_1(\mathcal{K}) > 0$ such that if $q_1,q_2 \in \mathcal{Q}^1\mathcal{T}_g$ satisfy $q_1 \in \pi^{-1}(\mathcal{K})$ and $d_E(q_1,q_2) \leq s_1$, then $d_\mathcal{T}(\pi(q_1),\pi(q_2)) \leq 1$. 
\end{lemma}

\subsection*{The modified Hodge metric.} In \cite[\S 3.3.2]{ABEM12}, a modification of the standard Hodge norm on the fibers of the tangent bundle of $\qut$ is introduced to guarantee the new norm behaves well near the multiple zero locus $\qut \setminus \mathcal{Q}^1$ while retaining the natural hyperbolicity properties of the Hodge norm with respect to the Teichmüller geodesic flow. This modified norm gives rise to a metric along the leaves of the strongly unstable foliation $\mathcal{F}^{uu}$ of $\qut$. We denote this metric by $d_H$ and refer to it as the \textit{modified Hodge metric}. The following theorem is proved in \cite{ABEM12}.

\begin{theorem}
	\cite[Theorem 3.10]{ABEM12}
	\label{theo:EH_1}
	Let $\mathcal{K} \subseteq \mathcal{T}_g$ be a compact subset. Then, for every pair of quadratic differentials $q_1,q_2 \in \pi^{-1}(\mathcal{K})$ on the same leaf of $\mathcal{F}^{uu}$,
	\[
	d_E(q_1,q_2) \preceq_\mathcal{K} d_H(q_1,q_2).
	\]
\end{theorem}

Directly from Lemma \ref{lem:ET_1} and Theorem \ref{theo:EH_1} we deduce the following result.

\begin{lemma}
	\label{leqm:EH_2}
	Let $\mathcal{K} \subseteq \mathcal{T}_g$ compact. There exists $s_1 := s_1(\mathcal{K})>0$ such that if $q_1,q_2 \in \mathcal{Q}^1\mathcal{T}_g$ are on the same leaf of $\mathcal{F}^{uu}$ and satisfy $q_1 \in \pi^{-1}(\mathcal{K})$ and $d_H(q_1,q_2) \leq s_1$, then $d_\mathcal{T}(\pi(q_1),\pi(q_2)) \leq 1$. 
\end{lemma}

\subsection*{The Thurston measure.} The space $\mf$ of singular measured foliations on $S_g$ can be endowed with a natural piecewise integral linear structure using train track coordinates; we will introduce these coordinates in more detail later in this appendix. In particular, $\mf$ carries a natural Lebesgue class measure $\nu$ called the \textit{Thurston measure}. This measure gives zero mass to the subset of singular measured foliations on $S_g$ having a singularity with more than three prongs \cite[Lemma 2.4]{Mir08a}. 

\subsection*{Sectors of Teichmüller space.} Recall that $\Delta \subseteq \mathcal{T}_g \times \mathcal{T}_g$ denotes the diagonal of $\mathcal{T}_g$, that $S(X) := \pi^{-1}(X)$ for every $X \in \mathcal{T}_g$, and that $q_s \colon \mathcal{T}_g \times \mathcal{T}_g \to \mathcal{Q}^1\mathcal{T}_g$ denotes the map which to every pair of distinct points $X,Y \in \mathcal{T}_g$ assigns the quadratic differential $q_s(X,Y) \in S(X)$ corresponding to the tangent direction at $X$ of the unique Teichmüller geodesic segment from $X$ to $Y$. Recall that, for every $X \in \mathcal{T}_g$ and every subset $V \subseteq S(X)$, the sector $\text{Sect}_V(X) \subseteq \tt$ is given by
\[
\text{Sect}_V(X) := \{X\} \cup \{Y \in \mathcal{T}_g \setminus \{X\} \ | \ q_s(X,Y) \in V \}.
\] 

Recall that $\mu$ denotes the Masur-Veech measure on $\qut$ and that $\mathbf{m} := \pi_* \mu$ denotes its pushforward to $\mathcal{T}_g$. Recall that for every $A \subseteq \mathcal{T}_g$ and every $r > 0$ we denote by $\mathrm{Nbhd}_r(A) \subseteq \mathcal{T}_g$ the set of points in $\mathcal{T}_g$ at Teichmüller distance at most $r$ from $A$. Consider the natural $\mathbf{R}^+$ scaling action on $\mf$. Denote by $\overline{\nu}$ the function which to every measurable subset $A \subseteq \mf$ assigns the value $\overline{\nu}(A) := \nu((0,1] \cdot A)$. Given $V \subseteq \mathcal{Q}^1\mathcal{T}_g$ and $s > 0$, denote by $V(s) \subseteq \mathcal{Q}^1\mathcal{T}_g$ the set of all $q_1 \in \mathcal{Q}^1\mathcal{T}_g$ such that there exists $q_2 \in V$ on the same leaf of $\mathcal{F}^{uu}$ as $q_1$ satisfying $d_H(q_1,q_2)<s$. In \cite{Ara20b} we prove the following general bound for the measure of sectors of Teichmüller space. The proof relies in a crucial way on the metric hyperbolicity properties of the Teichmüller geodesic flow proved in \cite{ABEM12}.

\begin{theorem}
	\cite[Theorem 7.10]{Ara20b}
	\label{theo:sector_general}
	There exist constants $r_0 = r_0(g) > 0$, $C = C(g) > 0$, and $\kappa = \kappa(g) >  0$ such that for every compact subset $\mathcal{K} \subseteq \tt$, every $X \in \mathcal{K}$, every measurable subset $V \subseteq S(X)$, every $0 < r < r_0$, and every $R > 0$,
	\begin{align*}
	\mathbf{m}(\mathrm{Nbhd}_{r}(B_R(X) \cap \mathrm{Sect}_{V}(X) \cap \mathrm{Mod}_g \cdot \mathcal{K})) 
	\preceq_\mathcal{K} \overline{\nu}(\Re(V( C e^{-\kappa R}))) \cdot e^{(6g-6)R} + e^{(6g-6-\kappa)R}.
	\end{align*}
\end{theorem}

\subsection*{Triangulations of quadratic differentials.} By a \textit{marked triangulation} of $S_g$ we mean an isotopy class of triangulations of $S_g$. By a \textit{triangulation} of a quadratic differential $q$ we mean a triangulation of its underlying Riemann surface whose edges are saddle connections of $q$. A triangulation of a quadratic differential $q$ is said to be \textit{$L$-bounded} for some $L > 0$ if its edges have flat length $\leq L$. Given a marked quadratic differential $q \in \qut$ and a triangulation $\Delta'$ of $q$, one can pull back $\Delta'$ via the marking of $q$ to obtain a marked triangulation $\Delta$ on $S_g$. Recall that $\mcg$ denotes the mapping class group of $S_g$ and that this group acts properly discontinuously on $\tt$ and $\qut$  by changing markings.

\begin{lemma}
	\label{lem:triang_compact}
	Let $\Delta$ be a marked triangulation of $S_g$, $\mathcal{K} \subseteq \tt$ be a compact subset, and $L > 0$. Then, the subset of marked quadratic differentials $q \in  \pi^{-1}(\mcg \cdot\mathcal{K})$ having an $L$-bounded triangulation $\Delta'$ that pulls back to $\Delta$ via the marking of $q$ has compact closure.
\end{lemma}

\begin{proof}
	Let $q \in  \pi^{-1}(\mcg \cdot\mathcal{K})$ and $\Delta'$ be an $L$-bounded triangulation of $q$ which pulls back to $\Delta$ via the marking of $q$. Fix a simple closed curve $\alpha$ on $S_g$. As $\Delta'$ is $L$-bounded and pulls back to $\Delta$ via the marking of $q$, one can bound the flat length $\ell_\alpha(q)$ of any geodesic representatives of $\alpha$ on $q$ uniformly in terms $\alpha$, $\Delta$, and $L$. This together with the fact that $q \in  \pi^{-1}(\mcg \cdot\mathcal{K})$ implies the hyperbolic length $\ell_\alpha(\pi(q))$ of the unique geodesic representative of $\alpha$ with respect to the marked hyperbolic structure on $S_g$ induced by $\pi(q) \in \mathcal{T}_g$ via uniformization can be bounded uniformly in terms of $\alpha$, $\Delta$, $L$, and $\mathcal{K}$. As the bundle $\pi \colon \qut \to \tt$ has compact fibers and as the only was of escaping to infinity in $\mathcal{T}_g$ is to develop a simple closed curve of unbounded hyperbolic length, this finishes the proof.
\end{proof}

\subsection*{Delaunay triangulations of quadratic differentials.} Recall that for every quadratic differential $q$ we denote by $\ell_{\min}(q)$ the length of its shortest saddle connections and by $\mathrm{diam}(q)$ its diameter. Every quadratic differential $(X,q)$ admits a triangulation by saddle connections which is Delaunay with respect to the singularities of $q$ and the singular flat metric induced by $q$ on $X$ \cite[\S 4]{MS91}. We refer to any such triangulation as a \textit{Delaunay triangulation} of $q$. For the purposes of this appendix we will not need to appeal to the explicit construction of these triangulations. Rather, it will suffice to know they exist and satisfy the following properties.

\begin{lemma}
	\cite[Lemma 3.11]{ABEM12} 
	\label{lem:delaunay}
	Let $q \in \qut$ and $\Delta$ be a Delaunay triangulation of $q$. 
	\begin{enumerate}
		\item If $\gamma$ is a saddle connection of $q$ that belongs to $\Delta$, then $\ell_{\gamma}(q) \preceq_g \mathrm{diam}(q)$.
		\item If $\gamma$ is a saddle connection of $q$ such that $\ell_{\gamma}(q) \leq \sqrt{2} \cdot \ell_{\min}(q)$, then $\gamma$ belongs to $\Delta$.
	\end{enumerate}
\end{lemma}

Directly from Lemmas \ref{lem:triang_compact} and \ref{lem:delaunay} we deduce the following result.

\begin{proposition}
	\label{prop:delaunay_comp}
	For every compact subset $\mathcal{K} \subseteq \mathcal{T}_g$ there exists a constant $L = L(\mathcal{K}) > 0$ and a finite collection of marked triangulations $\{\Delta_i\}_{i=1}^n$ on $S_g$ with the following property. Let $q \in \pi^{-1}(\mathcal{K})$ and $\gamma$ be a saddle connection of $q$ attaining the minimal flat length among saddle connections of $q$. Then, there exists $i \in \{1,\dots,n\}$ and an $L$-bounded triangulations $\Delta'$ of $q$ having $\gamma$ as one of its edges and which pulls back to $\Delta_i$ via the marking of $q$.
\end{proposition}

\begin{proof}
	Fix a compact subset $\mathcal{K} \subseteq \mathcal{T}_g$. By Lemma \ref{lem:delaunay}, Delaunay triangulations of quadratic differentials always contain the shortest saddle connections. By Lemma \ref{lem:delaunay}, there exists a constant $L = L(\mathcal{K}) > 0$ such that every Delaunay triangulation of a quadratic differential in $\pi^{-1}(\mathcal{K})$ is $L$-bounded. Thus, it suffices to show these triangulations pull back to finitely many marked triangulations on $S_g$. Triangulations of quadratic differentials in $\qut$ have at most $4g-4$ vertices. In particular, there are finitely many combinatorial types of triangulations on $S_g$ that arise as pullbacks of triangulations on $\qut$. By Lemma \ref{lem:triang_compact} and the proper discontinuity of the $\mcg$ action on $\qut$, there are finitely many marked triangulations on $S_g$ of a given combinatorial type that arise as pullbacks of $L$-bounded triangulations of quadratic differentials in $\pi^{-1}(\mathcal{K})$. This finishes the proof.
\end{proof}

\subsection*{Triangulations induced by closed curves.} Let $(X,q)$ be a quadratic differential. Two saddle connections of $q$ are said to be \textit{disjoint} if they only intersect at singularities of $q$. A \textit{complex} of $q$ is a subset $K \subseteq X$ consisting of disjoint saddle connections and triangles of $q$ with the condition that a triangle belongs to $K$ if and only if all of its sides belong to $K$. Let $K \subseteq X$ be a complex of $q$. Denote by $\partial K \subseteq X$ the topological boundary $K$. Denote by $\ell_K(q)$ the maximal flat length among saddle connections of $K$. Denote by $|K|$ the number of saddle connections of $K$. A saddle connection and a complex of $q$ are said to be disjoint if they only intersect at singularities of $q$.

\begin{proposition}
	\label{prop:triang_enlarge}
	\cite[Proof of Proposition 3.11]{MT02} 
	Let $q$ be a quadratic differential. Suppose $K$ is a complex of $q$ with $\partial K \neq \emptyset$ and $\gamma$ is a saddle connection of $q$ which is either disjoint from $K$ or crosses $\partial K$. Then, there exists a complex $K' \supseteq K$ such that $|K'| > |K|$ and $\ell_{K'}(q) \leq 2 \ell_{K}(q) + \ell_{\alpha}(q)$.
\end{proposition}

The following lemma will play an important role in our proof of Theorem \ref{theo:slope_sector_meas}.

\begin{lemma}
	\label{lem:beta_Delaunay}
	There exists a constant $C =  C(g) > 0$ with the following property. Let $q \in \qut$, $\beta$ be a closed curve on $S_g$, and $\beta'$ be a saddle connection of a singular flat geodesic representative of $\beta$ on $q$. Denote $L = L(q,\beta) := C \cdot \max\{\mathrm{diam}(q),\ell_{\beta}(q)\}$. Then, there exists an $L$-bounded triangulation $\Delta'$ of $q$ having $\beta'$ as one of its edges.
\end{lemma}

\begin{proof}
	Let $q \in \qut$, $\beta$ be a closed curve on $S_g$, and $\beta'$ be a saddle connection of a singular flat geodesic representative of $\beta$ on $q$. Starting with the complex $K := \{\beta'\}$, apply Proposition \ref{prop:triang_enlarge} inductively using appropriate saddle connections $\gamma$ of a fixed Delaunay triangulation of $q$ until the complex obtained is a triangulation. By Lemma \ref{lem:delaunay} and Proposition \ref{prop:triang_enlarge}, there exists a constant $C = C(g) > 0$ such that this triangulation is $L$-bounded for $L = L(q,\beta) := C \cdot \max\{\mathrm{diam}(q),\ell_{\beta}(q)\}$.
\end{proof}

The same arguments used to prove Proposition \ref{prop:delaunay_comp} allow us to deduce the following result as a consequence of Lemma \ref{lem:triang_compact} and Proposition \ref{lem:beta_Delaunay}.

\begin{proposition}
	\label{prop:curve_triang}
	For every compact subset $\mathcal{K} \subseteq \mathcal{T}_g$ and every closed curve $\beta$ on $S_g$ there exists a constant $L = L(\mathcal{K},\beta)$ and a finite collection of marked triangulations $\{\Delta_i\}_{i=1}^n$ on $S_g$ with the following property. Let $q \in \pi^{-1}(\mathcal{K})$ and $\beta'$ be a saddle connection of a singular flat geodesic representative of $\beta$ on $q$. Then, there exists $i \in \{1,\dots,n\}$ and an $L$-bounded triangulation $\Delta'$ of $q$ having $\beta'$ as one of its edges and which pulls back to $\Delta_i$ via the marking of $q$.
\end{proposition}

\subsection*{The function $\boldsymbol{\ell_{\min}(q)}$ and the Euclidean metric.} Given a piecewise smooth path $\rho \colon [0,1] \to \qut$, denote by $\ell_E(\rho) = \smash{\int_0^1 \| \rho'(t) \|_E \thinspace dt}$ its length with respect to the Euclidean metric.  In \cite{Ara20b} we combine Lemma \ref{lem:euc_fund_prop} with Proposition \ref{prop:delaunay_comp} to prove the following result.

\begin{proposition}
	\cite[Proposition 7.13]{Ara20b}
	\label{prop:lmin_euclidean}
	Let $\mathcal{K} \subseteq \mathcal{T}_g$ be a compact subset. There exist constants $s_1 = s_1(\mathcal{K}) > 0$ and $C = C(\mathcal{K}) > 0$ such that if $\rho \colon [0,1] \to \qut$ is a piecewise smooth path with $\ell_E(\rho) < s_1$, $\rho(0) \in \pi^{-1}(\mathcal{K}) \cap \mathcal{Q}^1$, and $\rho(1) \in \mathcal{Q}^1$, then
	\[
	\ell_{\min}(\rho(1)) \leq \ell_{\min}(\rho(0)) +  C \cdot \ell_E(\rho).
	\]
\end{proposition}

\subsection*{The function $\boldsymbol{h_{\beta}(q)}$ and the Euclidean metric.} Let $\beta$ be a closed curve on $S_g$. Let us recall the definition of the function $h_\beta \colon \qut \to \mathbf{R}$ introduced in (\ref{eq:hbeta}). Given $q \in \qut$, choose a flat geodesic representative of $\beta$ on $q$. Let $\{\beta_j\}_{j=1}^m$ be the sequence of saddle connections of this representative if it is singular, or a singleton consisting of the representative itself if it is a cylinder curve. For every $j \in \{1,\dots,m\}$ denote by $i(\beta_j,\Re(q))$ the total transverse measure of $\beta_j$ with respect to the singular measured foliation $\Re(q)$. Then,
\[
h_\beta(q) := \min_{j=1,\dots,m} \frac{i(\beta_j,\Re(q))}{\ell_{\beta_j}(q)}.
\]
By Proposition \ref{prop:flat_rep}, this value is independent of the choice of flat geodesic representative of $\beta$ on $q$. Furthermore, the function $h_\beta \colon \qut \to \mathbf{R}$ is continuous on strata of $\qut$. 

The methods introduced in the proof of \cite[Proposition 7.13]{Ara20b} allow us to deduce the following result as a consequence of Lemma \ref{lem:euc_fund_prop} and Proposition \ref{prop:curve_triang}.

\begin{proposition}
	\label{prop:hbeta_euclidean}
	Let $\mathcal{K} \subseteq \mathcal{T}_g$ be a compact subset and $\beta$ be a closed curve on $S_g$. There exist constants $s_1 = s_1(\mathcal{K}) > 0$ and $C = C(\mathcal{K},\beta) > 0$ such that if $\rho \colon [0,1] \to \mathcal{Q}^1$ is a piecewise smooth path with $\ell_E(\rho) < s_1$, $\rho(0) \in \pi^{-1}(\mathcal{K})$, and $\ell_{\min}(\rho(t)) \geq \delta$ for some $\delta > 0$ and every $t \in [0,1]$, then
	\[
	h_\beta(\rho(1)) \leq h_\beta(\rho(0)) +  C \cdot \delta^{-1} \cdot \ell_E(\rho).
	\]
\end{proposition}

\begin{proof}
	Fix a compact subset $\mathcal{K} \subseteq \tt$ and a closed curve $\beta$ on $S_g$. Recall that $\mathrm{Nbhd}_1(\mathcal{K}) \subseteq \tt$ denotes the compact subset of points in $\tt$ at Teichmüller distance at most $1$ from $\mathcal{K}$. Let $\{\Delta_i\}_{i=1}^n$ be the finite collection of marked triangulations of $S_g$ provided by Proposition \ref{prop:curve_triang} for the compact subset $\mathrm{Nbhd}_1(\mathcal{K}) \subseteq \mathcal{T}_g$ and the closed curve $\beta$ on $S_g$. Among these triangulations, consider those with $4g-4$ vertices. These triangulations gives rise to finitely many period coordinate charts $\Phi \colon U \to \mathbf{C}^{6g-6}$ defined on open subsets $U \subseteq \mathcal{Q}^1$ of the principal stratum. By Lemma \ref{lem:euc_fund_prop}, for any such period coordinate chart, any $q \in U \cap \pi^{-1}(\mathrm{Nbhd}_1(\mathcal{K}))$, and any $v \in T_q \qut$, 
	\begin{equation}
	\label{eq:euc_comparison}
	\| v \|_E \asymp_{\mathcal{K},\beta} \|d\Phi_q v\|.
	\end{equation}
	
	Let $s_1 = s_1(\mathcal{K}) > 0$ be as in Lemma	\ref{lem:ET_2}. Consider a piecewise smooth path $\rho \colon [0,1] \to \mathcal{Q}^1$ such that $\ell_E(\rho) < s_1$, $\rho(0) \in \pi^{-1}(\mathcal{K})$, and $\ell_{\min}(\rho(t)) \geq \delta$ for some $\delta > 0$ and every $t \in [0,1]$. By Lemma \ref{lem:ET_2}, $\rho([0,1]) \subseteq \pi^{-1}(\mathrm{Nbhd}_1(\mathcal{K}))$. Fix $t_0 \in [0,1]$. Let $\smash{\{\beta_j\}_{j=1}^m}$ be the sequence of saddle connections of a singular flat geodesic representative of $\beta$ on $\rho(t_0)$. There exists a relatively open interval $I \subseteq [0,1]$ containing $t_0$ such that all these saddle connections persist on the quadratic differentials in $\rho(I)$. As a consequence of this fact and (\ref{eq:euc_comparison}) we deduce
	\begin{gather}
	\bigg| \frac{d}{dt}\bigg|_{t = t_0^\pm} i(\beta_j,\Re(\rho(t))) \bigg| \preceq_{\mathcal{K},\beta}  \bigg\| \frac{d}{dt}\bigg|_{t = t_0^\pm} \rho(t)  \bigg\|_E, \label{eq:C1}\\
	\bigg| \frac{d}{dt}\bigg|_{t = t_0^\pm} i(\beta_j,\Im(\rho(t))) \bigg| \preceq_{\mathcal{K},\beta}  \bigg\| \frac{d}{dt}\bigg|_{t = t_0^\pm} \rho(t)  \bigg\|_E. \label{eq:C2}
	\end{gather}
	
	Notice that, if two consecutive saddle connections of the sequence $\smash{\{\beta_j\}_{j=1}^m}$ form an angle equal to $\pi$ on one side of the singularity at which they intersect, then there exist arbitrarily small deformations of $\rho(t_0)$ for which any flat geodesic representative of $\beta$ does not contain these saddle connections. See Figure \ref{fig:hor} for an example. Nevertheless, there exists a relatively open interval $I \subseteq [0,1]$ containing $t_0$ such that all the saddle connections $\smash{\{\beta_j\}_{j=1}^m}$ persist on the quadratic differentials in $\rho(I)$ and such that for every $t \in I$ there exists a singular flat geodesic representative of $\beta$ on $\rho(t)$ whose sequence of saddle connections $\smash{\{\beta_k(t)\}_{k=1}^{m(t)}}$ satisfies the following property: for every $k \in \{1,\dots,m(t)\}$ there exists a subset of indices $J_k(t) \subseteq \{1,\dots,m\}$ such that the corresponding saddle connections $\smash{\{\beta_j\}_{j\in J_k(t)}}$ on $\rho(t_0)$ are consecutive and form angles equal to $\pi$ on the same side of the singularities at which they intersect, and sequences $\{\epsilon_j(t)\}_{j \in J_k(t)}, \{\epsilon_j'(t)\}_{j \in J_k(t)}$ of $\{0,1\}$ such that
	\begin{gather}
	i(\beta_k(t),\Re(\rho(t))) = \sum_{j \in J_k(t)}  (-1)^{\epsilon_j(t)} \cdot i(\beta_j,\Re(\rho(t))), \label{eq:C3}\\
	i(\beta_k(t),\Im(\rho(t))) = \sum_{j \in J_k(t)}  (-1)^{\epsilon_j'(t)} \cdot i(\beta_j,\Im(\rho(t))). \label{eq:C4}
	\end{gather}
	
	\begin{figure}[h]
		\centering
		\begin{subfigure}[b]{0.4\textwidth}
			\centering
			\includegraphics[width=0.7\textwidth]{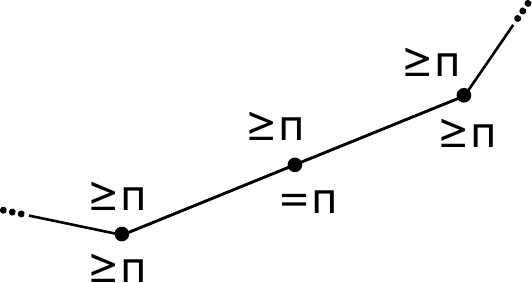}
			\caption{Before deformation.}
		\end{subfigure}
		\quad \quad \quad
		~ 
		\begin{subfigure}[b]{0.4\textwidth}
			\centering
			\includegraphics[width=0.7\textwidth]{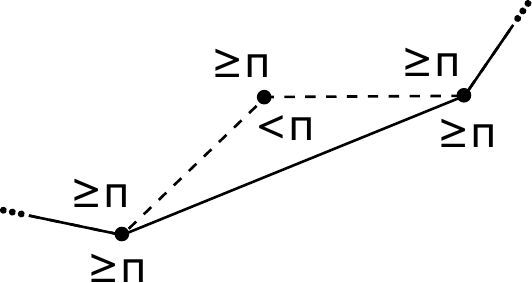}
			\caption{After deformation.}
		\end{subfigure}
		\captionsetup{width=\linewidth}
		\caption{Small deformation changing saddle connections of a geodesic representative.} 
		\label{fig:hor}
	\end{figure}

	Let $v \colon [0,1] \to \mathbf{R}^2$ be a smoothly varying non-zero vector. Denote $v(t) := (x(t),y(t))$ for every $t \in [0,1]$. Consider the Euclidean norm $\| \cdot \|$ on $\mathbf{R}^2$. A direct computation shows that
	\begin{equation}
	\label{eq:C5}
	\bigg| \frac{d}{dt}\bigg|_{t=0^+} \frac{x(t)}{\|v(t)\|}\bigg| \preceq \|v(0)\|^{-1} \cdot \bigg\| \frac{d}{dt} \bigg|_{t=0^+} v(t) \bigg\|. 
	\end{equation}
	Using (\ref{eq:C1}), (\ref{eq:C2}), (\ref{eq:C3}), (\ref{eq:C4}), (\ref{eq:C5}), and the assumption $\ell_{\min}(\rho(t_0)) \geq \delta$, we deduce
	\begin{equation*}
	\bigg| \limsup_{t\to t_0^\pm} \frac{h_\beta(\rho(t)) - h_\beta(\rho(t_0))}{t-t_0} \bigg| \preceq_{\mathcal{K},\beta} \delta^{-1} \cdot \bigg\| \frac{d}{dt}\bigg|_{t=t_0^\pm} \rho(t) \bigg\|_E.
	\end{equation*}
	As this estimate holds for every $t_0 \in [0,1]$, we conclude that, for some constant $C =C(\mathcal{K},\beta) > 0$,
	\[
	h_\beta(\rho(1)) \leq h_\beta(\rho(0)) +  C \cdot \delta^{-1} \cdot \ell_E(\rho). \qedhere
	\]
\end{proof}

Fix a closed curve $\beta$ on $S_g$. As in (\ref{eq:hbeta}), for every $\delta > 0$ consider the subset $H_{\beta,\delta} \subseteq \qut$ given by
\begin{equation}
\label{eq:H}
H_{\beta,\delta} := \{ q \in \mathcal{Q}^1 \ | \ \ell_{\min}(q)\geq \delta, \ h_\beta(q) \geq \delta \}.
\end{equation}
Recall that for every $V \subseteq \mathcal{Q}^1\mathcal{T}_g$ and every $s > 0$ we denote by $V(s) \subseteq \mathcal{Q}^1\mathcal{T}_g$ the set of all $q_1 \in \mathcal{Q}^1\mathcal{T}_g$ such that there exists $q_2 \in V$ on the same leaf of $\mathcal{F}^{uu}$ as $q_1$ satisfying $d_H(q_1,q_2)<s$. Putting together Theorem \ref{theo:EH_1}, Lemma \ref{leqm:EH_2}, and Propositions \ref{prop:lmin_euclidean} and \ref{prop:hbeta_euclidean}, we deduce the following result. This result will play a crucial role in our proof of Theorem \ref{theo:slope_sector_meas}.

\begin{proposition}
	\label{prop:thick}
	For every compact subset $\mathcal{K} \subseteq \mathcal{T}_g$ and every closed curve $\beta$ on $S_g$ there exist constants $s_1 = s_1(\mathcal{K}) > 0$ and $C =C(\mathcal{K},\beta) > 0$ such that for every $X \in \mathcal{K}$ and every $\delta \in (0,1)$, if $V_{\beta,\delta} := S(X) \setminus H_{\beta,\delta}$, then, for every $0 < s < s_1$,
	\[
	V_{\beta,\delta}(s) \subseteq \pi^{-1}(B_1(X)) \setminus H_{\beta, \delta + C \delta^{-1} s}.
	\]
\end{proposition}

\subsection*{Train track coordinates.} We now discuss some aspects of the theory of train track coordinates. For more details we refer the reader to \cite{PH92}. A \textit{train track} $\tau$ on $S_g$ is an embedded $1$-complex satisfying the following conditions:
\begin{enumerate}
	\item Each edge of $\tau$ is a smooth path with a well defined tangent vector at each endpoint. All edges at a given vertex are tangent.
	\item For each component $R$ of $S_g \setminus \tau$, the double of $R$ along the smooth part of the boundary $\partial R$ has negative Euler characteristic.
\end{enumerate}
The vertices of $\tau$ where three or more edges meet are called \textit{switches}. By considering the inward pointing tangent vectors of the edges incident to a switch, one can divide these edges into \textit{incoming} and \textit{outgoing} edges. A train track $\tau$ on $S_g$ is said to be \textit{maximal} if all the components of $S_g \setminus \tau$ are trigons, i.e., the interior of a disc with three non-smooth points on its boundary.

A singular measured foliation $\lambda \in \mf$ is said to be carried by a train track $\tau$ on $S_g$ if it can be obtained by collapsing the complementary regions in $S_g$ of a measured foliation of a tubular neighborhood of $\tau$ whose leaves run parallel to the edges of $\tau$. In this situation, the invariant transverse measure of $\lambda$ corresponds to a counting measure $v$ on the edges of $\tau$ satisfying the \textit{switch conditions}: at every switch of $\tau$ the sum of the measures of the incoming edges equals the sum of the measures of the outgoing edges. Every $\lambda \in \mf$ is carried by some maximal train track $\tau$ on $S_g$.

Given a maximal train track $\tau$ on $S_g$, denote by $V(\tau) \subseteq \smash{(\mathbf{R}_{\geq0})^{18g-18}}$ the $6g-6$ dimensional closed cone of non-negative counting measures on the edges of $\tau$ satisfying the switch conditions. The set $V(\tau)$ can be identified with the closed cone $U(\tau) \subseteq \mf$ of singular measured foliations carried by $\tau$. These identifications give rise to coordinates on $\mf$ called \textit{train track coordinates}. The transition maps of these coordinates are piecewise integral linear. In particular, $\mf$ can be endowed with a natural $6g-6$ dimensional piecewise integral linear structure. The Thurston measure $\nu$ is the unique, up to scaling, piecewise integral linear measure on $\mf$. 

\subsection*{Train tracks dual to triangulations.} Let $\Delta$ be a marked triangulation on $S_g$. On each of the triangles of $\Delta$ consider a $1$-complex as in Figure \ref{fig:complex_1}; the edges of this complex that do not intersect the sides of the triangle will be referred to as \textit{inner edges}. Join these complexes along the edges of $\Delta$ as in Figure \ref{fig:complex_2} to obtain a complex on $S_g$. We say that a train track $\tau$ on $S_g$ is dual to $\Delta$ if it can be obtained from this complex by deleting one inner each in each triangle of $\Delta$.

	\begin{figure}[h]
	\centering
	\begin{subfigure}[b]{0.4\textwidth}
		\centering
		\includegraphics[width=0.7\textwidth]{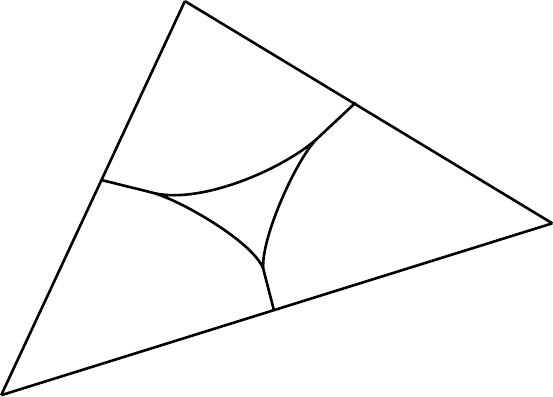}
		\caption{The dual $1$-complex in a triangle.}
		\label{fig:complex_1}
	\end{subfigure}
	\quad \quad \quad
	~ 
	\begin{subfigure}[b]{0.4\textwidth}
		\centering
		\includegraphics[width=0.7\textwidth]{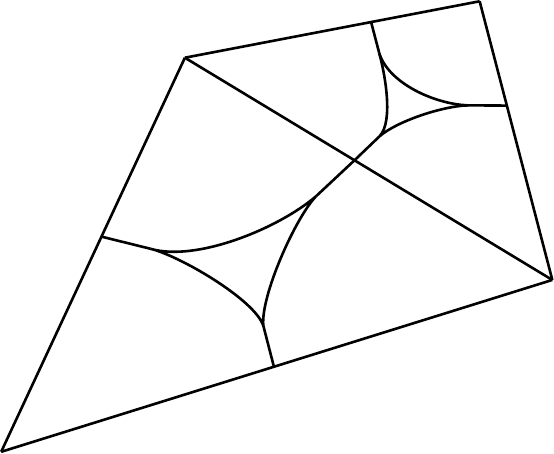}
		\caption{Joining the dual $1$-complexes.}
		\label{fig:complex_2}
	\end{subfigure}
	\caption{The $1$-complex dual to a triangulation.} 
	\label{fig:complex}
	\end{figure}

Let $q \in \qut$ and $\Delta'$ be a triangulation of $q$. Denote by $\Delta$ the marked triangulation of $S_g$ obtained by pulling back $\Delta'$ via the marking of $q$. The vertical foliation $\Re(q) \in \mf$ is carried by a train track dual to $\Delta$. Indeed, let $T'$ be a triangle of $\Delta'$. Label the edges of $T'$ by $a,b,c$ so that
\[
i(a,\Re(q)) = i(b,\Re(q)) + i(c,\Re(q)).
\]
This labelling is unique unless one of the edges of $T'$ is vertical, in which case there exist two such labellings. On $T'$ consider a  $1$-complex as in Figure \ref{fig:complex_1}. Delete the inner edge of this complex opposite to $a$ as in Figure \ref{fig:tt_1}. Consider the corresponding $1$-complexes on all the triangles of $\Delta'$. Joining these complexes along the edges of $\Delta'$ as in Figure \ref{fig:tt_2} and pulling back the resulting complex to $S_g$ yields a train track $\tau$ dual to $\Delta$. This train track carries $\Re(q) \in \mf$. For each edge $e$ of $\Delta'$, the counting measure of the coresponding edge of $\tau$ is equal to $i(e,\Re(q))$. If $q \in \mathcal{Q}^1$, the train track $\tau$ obtained through this construction is maximal.

	\begin{figure}[h]
	\centering
	\begin{subfigure}[b]{0.4\textwidth}
		\centering
		\includegraphics[width=0.7\textwidth]{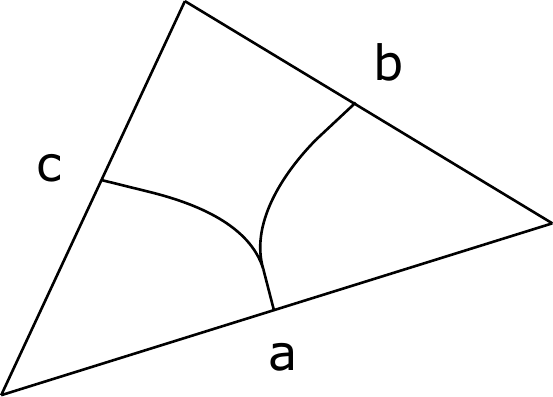}
		\caption{The dual train track in a triangle.}
		\label{fig:tt_1}
	\end{subfigure}
	\quad \quad \quad
	~ 
	\begin{subfigure}[b]{0.4\textwidth}
		\centering
		\includegraphics[width=0.7\textwidth]{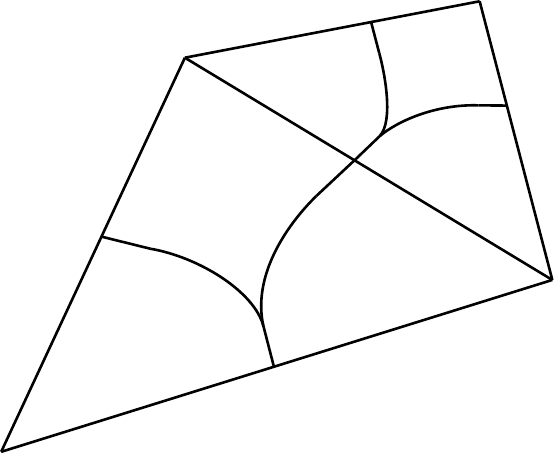}
		\caption{Joining the dual train tracks.}
		\label{fig:tt_2}
	\end{subfigure}
	\captionsetup{width=\linewidth}
	\caption{The dual train track carrying the vertical foliation of a quadratic differential.} 
	\label{fig:tt}
	\end{figure}

Directly from the discussion above and Propositions \ref{prop:delaunay_comp} and \ref{prop:curve_triang} we deduce the following result.

\begin{proposition}
	\label{lem:finite_tt}
	For every compact subset $\mathcal{K} \subseteq \mathcal{T}_g$ and every closed curve $\beta$ on $S_g$ there exists a finite collection of maximal train tracks $\{\tau_i\}_{i=1}^n$ on $S_g$ with the following property. Let $q \in \pi^{-1}(\mathcal{K}) \cap \mathcal{Q}^1$ and $\gamma$ be a saddle connection of $q$ that either attains the minimal flat length among saddle connections of $q$ or that is a saddle connection of a singular flat geodesic representative of $\beta$ on $q$. Then, there exists $i \in \{1,\dots,n\}$ such that $\tau_i$ carries $\Re(q)$ and such that the corresponding counting measure on the edges of $\tau_i$ gives weight $i(\gamma,\Re(q))$ to one of the edges.
\end{proposition}

\subsection*{Thurston measure estimates.}  Recall that $\nu$ denotes the Thurston measure on $\mf$ and that $\overline{\nu}$ denotes the function which to every measurable subset $A \subseteq \mf$ assigns the value $\overline{\nu}(A) := \nu((0,1] \cdot A)$. Recall the definition of the subsets $H_{\beta,\delta} \subseteq \mathcal{T}_g$ in (\ref{eq:H}). Proposition \ref{lem:finite_tt} leads to the following estimate. This estimate will play a crucial role in our proof of Theorem \ref{theo:slope_sector_meas}.

\begin{proposition}
	\label{prop:meas_bd}
	Let $\mathcal{K} \subseteq \mathcal{T}_g$ compact and $\beta$ be a closed curve on $S_g$. Then, for every $\delta > 0$, 
	\[
	\overline{\nu}(\Re(\pi^{-1}(\mathcal{K}) \setminus H_{\beta,\delta})) \preceq_\mathcal{K,\beta} \delta.
	\]
\end{proposition}

\begin{proof}
	Fix a compact subset $\mathcal{K} \subseteq \mathcal{T}_g$, a closed curve $\beta$ on $S_g$, and $\delta > 0$. Recall that the Thurston measure $\nu$ on $\mf$ gives zero mass to the set of singular measured foliations having a singularity with more than three prongs. If $q \in \qut \setminus \mathcal{Q}^1$ then $\Re(q) \in \mf$ belongs to this set.	It follows that
	\begin{equation}
	\label{eq:tt1}
	\overline{\nu}(\Re(\pi^{-1}(\mathcal{K}) \setminus H_{\beta,\delta})) =\overline{\nu}(\Re(\pi^{-1}(\mathcal{K}) \cap \mathcal{Q}^1 \setminus H_{\beta,\delta})).
	\end{equation}
	
	Consider the finite collection of maximal train tracks $\{\tau_i\}_{i=1}^n$ on $S_g$ provided by Lemma \ref{lem:finite_tt} for the compact subset $\mathcal{K} \subseteq \tt$ and the closed curve $\beta$ on $S_g$. Recall that for every $i \in \{1,\dots,n\}$ the closed cone $V(\tau_i) \subseteq \smash{(\mathbf{R}_{\geq0})^{18g-18}}$ of counting measures on the edges of $\tau_i$ satisfying the switch conditions is naturally identified with the the closed cone $U(\tau_i) \subseteq \mf$ of singular measured foliations on $S_g$ carried by $\tau_i$. Denote by $\|\lambda \|_{\tau_i}$ the Euclidean norm of $\lambda \in U(\tau_i)$ via this identification. Denote by $\mu_i$ the Lebesgue measure on $V(\tau_i) \subseteq \smash{(\mathbf{R}_{\geq0})^{18g-18}}$. Let $\smash{\mathrm{Ext}}_\lambda(X) > 0$ be the extremal length of $\lambda \in \mf$ with respect to $X \in \mathcal{T}_g$ as defined in \cite{Ker80}. As $\mathcal{K}$ is compact and as $U(\tau_i) \subseteq \mf$ is projectively compact, $\| \lambda \|_{\tau_i} \preceq_{\mathcal{K},\beta} \smash{\sqrt{\mathrm{Ext}_\lambda(X)}}$ for every $i \in \{1,\dots,n\}$, every $\lambda \in U(\tau_i)$, and every $X \in \mathcal{K}$. For every $(X,q) \in \qut$ we have $\mathrm{Ext}_{\Re(q)}(X) = \mathrm{Area}(q)= 1$. These facts together with Lemma \ref{lem:finite_tt} and the definition of $H_{\beta,\delta} \subseteq \qut$ imply that, for some constant $C = C(\mathcal{K},\beta) > 0$,
	\begin{equation}
	\label{eq:tt2}
	\overline{\nu}(\Re(\pi^{-1}(\mathcal{K})\cap \mathcal{Q}^1\backslash H_{\beta,\delta}))
	\leq \sum_{i=1}^n\sum_{j=1}^{18g-18} \mu_i(\{v \in V(\tau_i) \ | \ \|v\|_{\tau_i} \leq C, \ v_j \leq \delta \}) \preceq_{\mathcal{K},\beta} \delta.
	\end{equation}
	
	Putting together (\ref{eq:tt1}) and (\ref{eq:tt2}) we conclude
	\[
	\overline{\nu}(\Re(\pi^{-1}(\mathcal{K}) \setminus H_{\beta,\delta})) \preceq_\mathcal{K,\beta} \delta. \qedhere
	\]
\end{proof}

\subsection*{Horizontally thin sectors of Teichmüller space.} Theorem \ref{theo:slope_sector_meas}, which we restate here for the reader's convenience, follows directly from Theorem \ref{theo:sector_general} and Propositions \ref{prop:thick} and \ref{prop:meas_bd}.

\begin{theorem}
	\label{theo:slope_sector_meas_new}
	There exist constants $r_0 = r_0(g) > 0$ and $\kappa_4 = \kappa_4(g) > 0$ such that the following holds. Let $\mathcal{K} \subseteq \mathcal{T}_g$ be a compact subset, $X \in \mathcal{K}$, $\beta$ be a closed curve on $S_g$, and $\delta \in (0,1)$. Denote $V_{\beta,\delta} := S(X) \setminus H_{\beta,\delta}$. Then, for every $0 < r < r_0$ and every $R > 0$,
	\[
	\mathbf{m}(\mathrm{Nbhd}_{r}(B_R(X) \cap \mathrm{Sect}_{V_{\beta,\delta}}(X) \cap \mathrm{Mod}_g \cdot \mathcal{K}))
	\preceq_{\mathcal{K},\beta} \delta \cdot e^{(6g-6)R} + \delta^{-1} \cdot e^{(6g-6 - \kappa_4) R}.
	\]
\end{theorem}

\section{Lipschitz functions on the space of measured geodesic laminations}

\subsection*{Outline of this appendix.} In this appendix we give a complete proof of Proposition \ref{prop:ml_lip}. The proof will use train track coordinates and some general facts about convex functions in Euclidean spaces. We refer the reader to Appendix A for some of the notation and terminology that will be used with regards to train track coordinates.

\subsection*{Convex functions in train track coordinates.} For the rest of this appendix we fix an integer $g \geq 2$ and a connected, oriented, closed surface $S_g$ of genus $g$. Recall that $\mf$ denotes the space of singular measured foliations on $S_g$. Let $\tau$ be a maximal train track on $S_g$. Recall that the closed cone $V(\tau) \subseteq \smash{(\mathbf{R}_{\geq0})^{18g-18}}$ of counting measures on the edges of $\tau$ satisfying the switch conditions is naturally identified with the the closed cone $U(\tau) \subseteq \mf$ of singular measured foliations on $S_g$ carried by $\tau$. Denote by $\Phi_\tau \colon U(\tau) \to V(\tau)$ the corresponding identification. We refer to this identification as the \textit{train track chart} induced by $\tau$ on $\mf$. Recall that $\mathcal{C}_g$ denotes the space of geodesic currents on $S_g$ and that $i(\cdot,\cdot)$ denotes the geometric intersection number pairing on $\mathcal{C}_g$. The following theorem of Mirzakhani shows that geometric intersection numbers are convex in train track coordinates.

\begin{theorem}
	\cite[Theorem A.1]{Mir04}
	\label{theo:current_convex}
	Let $\tau$ be a maximal train track on $S_g$, $\Phi_\tau \colon U(\tau) \to V(\tau)$ be the  train track chart induced by $\tau$ on $\mf$, and $\alpha \in \mathcal{C}_g$ be a geodesic current on $S_g$. Then, the composition $i(\alpha,\cdot) \circ \Phi_\tau^{-1} \colon V(\tau) \to \mathbf{R}$ is convex.
\end{theorem}

\subsection*{Convex functions are Lipschitz.} We now discuss some general facts about convex functions in Euclidean spaces. Let $\Omega \subseteq \mathbf{R}^n$ be an open, convex subset and $K \subseteq \Omega$ be a compact subset. Denote by $d(K,\partial\Omega)$ the Euclidean distance between $K$ and $\partial \Omega \subseteq \mathbf{R}^n$ with the convention that $d(K,\partial\Omega) = +\infty$ if $\Omega = \mathbf{R}^n$. Let $r = r(\Omega,K) := \min\{1,d(K,\partial\Omega) /2\}$. Denote by $\mathrm{Nbhd}_r(K) \subseteq \Omega$ the set of points in $\mathbf{R}^n$ at Euclidean distance at most $r$ from $K$. Let $f \colon \Omega \to \mathbf{R}$ be a convex function. Denote 
\[
L = L(\Omega,K,f) := \textstyle \smash{2 \cdot r^{-1} \cdot \sup_{x \in \mathrm{Nbhd}_r(K)} |f(x)|}.
\]
The following well known result shows that convex functions on Euclidean spaces are Lipschitz when restricted to compact subsets of their domain.

\begin{proposition}
	\label{prop:convex_lip}
	Let $\Omega \subseteq \mathbf{R}^n$ be an open, convex subset and $f \colon \Omega \to \mathbf{R}$ be a convex function. Then, for every $K \subseteq \Omega$ compact, the restriction $f|_K \colon K \to \mathbf{R}$ is $L$-Lipschitz for $L = L(\Omega,K,f)$ as above.
\end{proposition}

Let $C \subseteq \smash{(\mathbf{R}_{>0})^n}$ be a cone cut out by homogeneous linear equations and $D \subseteq C$ be a projectively compact cone. Denote by $D' \subseteq  \smash{(\mathbf{R}_{>0})^n}$ the set of points in $D$ with Euclidean norm between $1$ and $2$. Let $\partial C \subseteq \smash{(\mathbf{R}_{\geq 0})^n}$ be the boundary of $C$ when considered as a subset of $\smash{(\mathbf{R}_{\geq 0})^n}$. Denote by $d(D',\partial C)$ the Euclidean distance between $D'$ and $\partial C$. Let $r' = r'(C,D) := \min\{1,d(D',\partial C) /2\}$. Denote by $\mathrm{Nbhd}_{r'}(D') \subseteq \smash{(\mathbf{R}_{>0})^n}$ the set of points in $ \smash{(\mathbf{R}_{>0})^n}$ at Euclidean distance at most $r'$ from $D'$. Let $f \colon C \to \mathbf{R}$ be a convex, homogeneous function. Denote 
\[
L' = L'(C,D,f) := \textstyle \smash{2 \cdot (r')^{-1} \cdot \sup_{x \in \mathrm{Nbhd}_{r'}(D')} |f(x)|}.
\]
Directly from Proposition \ref{prop:convex_lip} we deduce the following result.

\begin{proposition}
	\label{prop:convex_lip_2}
	Let $C \subseteq \smash{(\mathbf{R}_{>0})^n}$ be a cone cut out by homogeneous linear equations and $f \colon C \to \mathbf{R}$ be a convex, homogeneous function. Then, for every projectively compact cone $D \subseteq C$, the restriction $f|_{D} = \colon D \to \mathbf{R}$ is $L'$-Lipschitz for $L' = L'(C,D,f)$ as above.
\end{proposition}

\subsection*{Lipschitz functions in train track coordinates.} Using Theorem \ref{theo:current_convex} and Proposition \ref{prop:convex_lip_2} we show that geometric intersection numbers are Lipschitz in train track coordinates.

\begin{proposition}
	\label{prop:int_lip}
	Let $\tau$ be a maximal train track on $S_g$, $\Phi_\tau \colon U(\tau) \to V(\tau)$ be the train track chart induced by $\tau$ on $\mf$, and $K \subseteq \mathcal{C}_g$ compact. Then, there exists a constant $L = L(\tau,K) > 0$ such that for every geodesic current $\alpha \in K$, the composition $i(\alpha,\cdot) \circ \Phi_\tau^{-1} \colon V(\tau) \to \mathbf{R}$ is $L$-Lipschitz.
\end{proposition}

\begin{proof}
	Given a maximal train track $\tau$ on $S_g$, denote by $d_{\tau}$ the Euclidean metric on $U(\tau)$ induced by the identification $\Phi_\tau \colon U(\tau) \to V(\tau)$. Furthermore, given a piecewise smooth path $\gamma \colon [0,1] \to U(\tau)$, denote by $\ell_{\tau}(\gamma)$ the length of $\gamma$ with respect to this metric. Every singular measured foliation $\lambda \in \mf$ is carried by a maximal train track $\tau$ on $S_g$ in such a way that $\lambda$ belongs to the interior of $U(\tau) \subseteq \mf$. In particular, as $\mf$ is projectively compact, there exist a finite collection $\{\tau_i\}_{i=1}^n$ of maximal train tracks on $S_g$ and a finite collection $\{W_i\}_{i=1}^n$ of open subsets of $\mf$ such that $W_i$ is compactly contained in the interior of $U(\tau_i)$ for every $i \in \{1,\dots,n\}$ and such that $\mf \subseteq \bigcup_{i=1}^n W_i$. 
	
	Fix a maximal train track $\tau$ on $S_g$, a compact subset $K \subseteq \mathcal{C}_g$, and a geodesic currents $\alpha \in K$. Let $\lambda_0,\lambda_1 \in U(\tau)$ arbitrary. Consider a piecewise smooth path $\gamma \colon [0,1] \to U(\tau)$ such that $\gamma(0) = \lambda_1$, $\gamma(1) = \lambda_1$, and $\ell_{\tau}(\gamma) \leq 2 d_\tau(\lambda_0,\lambda_1)$. By the Lebesgue number lemma, there exists a finite partition $0 =: t_0 < t_1 < \cdots < t_{m-1} < t_m := 1$ such that for every $j \in \{0,\dots,m-1\}$ there exists $i(j) \in \{1,\dots,n\}$ satisfying $\gamma([t_{j},t_{j+1}]) \subseteq W_{i(j)}$. As change of coordinates maps between train track coordinates are piecewise linear, the following estimate holds for every $j \in \{0,\dots,m-1\}$,
	\begin{equation}
	\label{eq:D1}
	\ell_{\tau_{i(j)}}\left(\gamma|_{[t_{j},t_{j+1}]}\right) \asymp_{\tau} \ell_{\tau}\left(\gamma|_{[t_{j},t_{j+1}]}\right).
	\end{equation}
	By Proposition \ref{prop:convex_lip_2}, for every $j \in \{0,\dots,m-1\}$,
	\begin{equation}
	\label{eq:D2}
	|i(\alpha,\gamma(t_j)) - i(\alpha,\gamma(t_{j+1}))| \preceq_K d_{\tau_{i(j)}}(\gamma(t_j), \gamma(t_{j+1})) \leq \ell_{\tau_{i(j)}}\left(\gamma|_{[t_{j},t_{j+1}]}\right).
	\end{equation}
	Using the triangle inequality, (\ref{eq:D2}), (\ref{eq:D1}), and the assumption $\ell_{\tau}(\gamma) \leq 2 d_\tau(\lambda_0,\lambda_1)$, we conclude
	\[
	|i(\alpha,\lambda_0) - i(\alpha,\lambda_1)| \leq \sum_{j=0}^{m-1} |i(\alpha,\gamma(t_j)) - i(\alpha,\gamma(t_{j+1}))| \preceq_{\tau,K} d_\tau(\lambda_0,\lambda_1). \qedhere
	\]
\end{proof}

\subsection*{Lipschitz functions in Dehn-Thurston coordinates.} Consider the piecewise linear manifold $\IT:= \mathbf{R}^2 / \langle-1\rangle$ endowed with the quotient Euclidean metric. The product $\IT^{3g-3}$ is a piecewise linear manifold which we endow with the product metric. Recall that any set of Dehn-Thurston coordinates provides a homeomorphism $F \colon \mf \to \IT^{3g-3}$ equivariant with respect to the natural $\mathbf{R}^+$ scaling actions on $\mf$ and $\IT^{3g-3}$. The change of coordinates maps between Dehn-Thurston coordinates and train track coordinates are  piecewise linear \cite[\S 2.6]{PH92}. Thus, Proposition \ref{prop:ml_lip}, which we restate here for the reader's convenience, can be proved using the same arguments as Proposition \ref{prop:int_lip}.  

\begin{proposition}
	\label{prop:ml_lip_new}
	Let $F \colon \mf \to \IT^{3g-3}$ be a set of Dehn-Thurston coordinates of $\mf$ and $K \subseteq \mathcal{C}_g$ be a compact subset of geodesic currents on $S_g$. Then, there exists a constant $L  = L(F,K) > 0$ such that for every $\alpha \in K$, the composition $i(\alpha,\cdot) \circ F^{-1} \colon \IT^{3g-3} \to \mathbf{R}$ is $L$-Lipschitz.
\end{proposition}


\bibliographystyle{amsalpha}


\bibliography{bibliography}

\end{document}